\documentclass[11pt]{amsart}
\usepackage[hmargin=3cm,vmargin=3cm]{geometry}

\usepackage[latin1]{inputenc}
\usepackage{xspace,amssymb,amsfonts,euscript}
\usepackage{amsthm,amsmath,amscd,stmaryrd,latexsym}
\usepackage{palatino, euscript}
\usepackage{enumitem}
\usepackage{pdflscape}

\usepackage{graphicx}
\usepackage[all]{xy}
\usepackage{tikz}
\usetikzlibrary{calc, matrix, arrows, decorations.pathmorphing}
\tikzset{%
	descr/.style={fill=white},
	baseline={([yshift=-\the\dimexpr\fontdimen22\textfont2\relax]
	                    current bounding box.center)},
	->,>=angle 90
}
\usepackage[dvips]{epsfig}
\usepackage{pinlabel}
\usepackage{dynkin-diagrams}

\newcommand{\ig}[2]{\vcenter{\xy (0,0)*{\includegraphics[scale=#1]{fig/#2}} \endxy}}

\newtheorem{thm}{Theorem}[section]
\newtheorem{lemma}[thm]{Lemma}

\newtheorem{prop}[thm]{Proposition}
\newtheorem{cor}[thm]{Corollary}

\newtheorem{conj}[thm]{Conjecture}

\newtheorem{dream}[thm]{Dream}

\newtheorem*{prop*}{Proposition}
\newtheorem*{lemma*}{Lemma}

\theoremstyle{definition}
\newtheorem{defn}[thm]{Definition}

\newtheorem{notation}[thm]{Notation}
\newtheorem{ex}[thm]{Example}

\theoremstyle{remark}
\newtheorem{remark}[thm]{Remark}
\newtheorem{rmk}[thm]{Remark}

\numberwithin{equation}{section}

\def\al{\alpha}

\def\la{\lambda}

\let\phi=\varphi
\let\tilde=\widetilde
\newcommand{\barphi}{\bar{\varphi}}

\usepackage{bbm}

\def\N{{\mathbbm N}}

\def\Z{{\mathbbm Z}}

\def\one{\mathbbm{1}}

\newcommand{\un}{\underline}

\newcommand{\ot}{\otimes}

\newcommand{\pa}{\partial}
\newcommand{\co}{\colon}

\renewcommand{\to}{\rightarrow}

\newcommand{\sumset}{\stackrel{\scriptstyle{\oplus}}{\scriptstyle{\subset}}}

\newcommand{\refequal}[1]{\xy {\ar@{=}^{#1}
(-1,0)*{};(1,0)*{}};
\endxy}

\DeclareMathOperator{\Hom}{Hom}
\DeclareMathOperator{\HOM}{HOM}
\DeclareMathOperator{\End}{End}

\DeclareMathOperator{\BS}{BS}

\DeclareMathOperator{\id}{id}

\DeclareMathOperator{\FT}{FT}
\DeclareMathOperator{\HT}{HT}
\DeclareMathOperator{\half}{ht}
\DeclareMathOperator{\full}{ft}

\DeclareMathOperator{\Schu}{Schu}
\DeclareMathOperator{\bug}{big}
\DeclareMathOperator{\pbug}{pbig}
\DeclareMathOperator{\Cone}{Cone}
\DeclareMathOperator{\JW}{JW}
\newcommand{\pSchu}{\mathrm{Schu}^p}
\def\HB{\mathbb{H}}
\newcommand{\ab}{\mathbf{a}}

\newcommand{\cb}{\mathbf{c}}
\newcommand{\pcb}{\mathbf{c}^p}
\newcommand{\pab}{\mathbf{a}^p}
\newcommand{\xb}{\mathbf{x}}
\newcommand{\pxb}{\mathbf{x}^p}
\def\IC{\mathcal{I}}
\def\HC{\mathcal{H}}
\def\FC{\mathcal{F}}
\def\OC{\mathcal{O}}
\newcommand{\ubr}[2]{\underbrace{#1}_{#2}}
\newcommand{\pb}{c}
\newcommand{\pB}{C}
\DeclareMathOperator{\val}{val}

\newcommand{\lineblue}{\ig{.5}{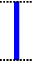}}
\newcommand{\linered}{\ig{.5}{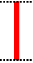}}

\newcommand{\longblue}{\ig{.5}{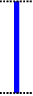}}
\newcommand{\longred}{\ig{.5}{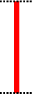}}
\newcommand{\startdotblue}{\ig{.5}{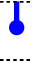}}
\newcommand{\startdotred}{\ig{.5}{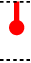}}
\newcommand{\finaldotblue}{\ig{.5}{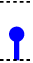}}
\newcommand{\finaldotred}{\ig{.5}{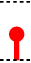}}
\newcommand{\barbblue}{\ig{.5}{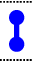}}
\newcommand{\barbred}{\ig{.5}{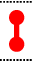}}

\newcommand{\brokenred}{\ig{.5}{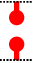}}
\newcommand{\splitred}{\ig{.5}{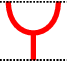}}
\newcommand{\splitblue}{\ig{.5}{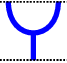}}
\newcommand{\mergered}{\ig{.5}{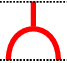}}
\newcommand{\mergeblue}{\ig{.5}{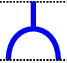}}
\newcommand{\redtoblue}{\ig{.5}{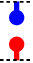}}

\newcommand{\trianglered}{\ig{.5}{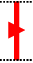}}
\newcommand{\bluepitchforkup}{\ig{.5}{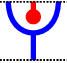}}
\newcommand{\bluepitchforkdown}{\ig{.5}{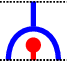}}
\newcommand{\redpitchforkup}{\ig{.5}{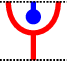}}
\newcommand{\redpitchforkdown}{\ig{.5}{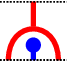}}
\newcommand{\doublebluepitch}{\ig{.5}{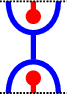}}
\newcommand{\doubleredpitch}{\ig{.5}{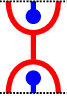}}

\newcommand{\bigpoly}[1]{{
\labellist
\small\hair 2pt
 \pinlabel {$#1$} [ ] at 7 15
\endlabellist
\centering
\ig{1}{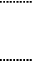}
}}
\newcommand{\ethree}{\; {
\labellist
\tiny\hair 2pt
 \pinlabel {$3$} [ ] at 19 23
\endlabellist
\centering
\ig{.5}{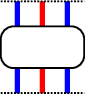}
}\; }
\newcommand{\wtothree}{{
\labellist
\tiny\hair 2pt
 \pinlabel {$3$} [ ] at 20 44
\endlabellist
\centering
\ig{.5}{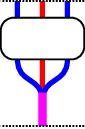}
}}
\newcommand{\wtootherthree}{{
\labellist
\tiny\hair 2pt
 \pinlabel {$3'$} [ ] at 20 44
\endlabellist
\centering
\ig{.5}{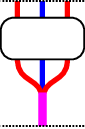}
}}

\newcommand{\threetots}{{
\labellist
\tiny\hair 2pt
 \pinlabel {$3$} [ ] at 20 20
\endlabellist
\centering
\ig{.5}{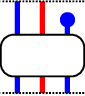}
}}
\newcommand{\threetost}{{
\labellist
\tiny\hair 2pt
 \pinlabel {$3$} [ ] at 20 20
\endlabellist
\centering
\ig{.5}{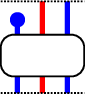}
}}
\newcommand{\othertots}{{
\labellist
\tiny\hair 2pt
 \pinlabel {$3'$} [ ] at 20 20
\endlabellist
\centering
\ig{.5}{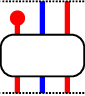}
}}
\newcommand{\othertost}{{
\labellist
\tiny\hair 2pt
 \pinlabel {$3'$} [ ] at 20 20
\endlabellist
\centering
\ig{.5}{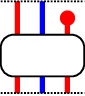}
}}

\newcommand{\sttothree}{{
\labellist
\tiny\hair 2pt
 \pinlabel {$3$} [ ] at 20 28
\endlabellist
\centering
\ig{.5}{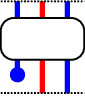}
}}
\newcommand{\tstothree}{{
\labellist
\tiny\hair 2pt
 \pinlabel {$3$} [ ] at 20 28
\endlabellist
\centering
\ig{.5}{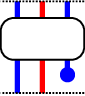}
}}
\newcommand{\sttoother}{{
\labellist
\tiny\hair 2pt
 \pinlabel {$3'$} [ ] at 20 28
\endlabellist
\centering
\ig{.5}{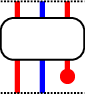}
}}
\newcommand{\tstoother}{{
\labellist
\tiny\hair 2pt
 \pinlabel {$3'$} [ ] at 20 28
\endlabellist
\centering
\ig{.5}{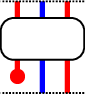}
}}
\newcommand{\idtothree}{{
\labellist
\tiny\hair 2pt
 \pinlabel {$3$} [ ] at 20 28
\endlabellist
\centering
\ig{.5}{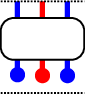}
}}
\newcommand{\threestthree}{{
\labellist
\tiny\hair 2pt
 \pinlabel {$3$} [ ] at 20 18
 \pinlabel {$3$} [ ] at 20 60
\endlabellist
\centering
\ig{.5}{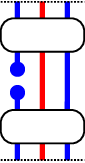}
}}
\newcommand{\threetsthree}{{
\labellist
\tiny\hair 2pt
 \pinlabel {$3$} [ ] at 20 18
 \pinlabel {$3$} [ ] at 20 60
\endlabellist
\centering
\ig{.5}{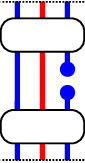}
}}
\newcommand{\threetostosts}{{
\labellist
\tiny\hair 2pt
 \pinlabel {$3$} [ ] at 20 18
\endlabellist
\centering
\ig{.5}{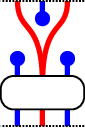}
}}
\newcommand{\threetosttosts}{{
\labellist
\tiny\hair 2pt
 \pinlabel {$3$} [ ] at 20 18
\endlabellist
\centering
\ig{.5}{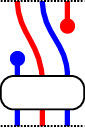}
}}
\newcommand{\threetotstosts}{{
\labellist
\tiny\hair 2pt
 \pinlabel {$3$} [ ] at 20 18
\endlabellist
\centering
\ig{.5}{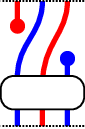}
}}
\newcommand{\ststotstothree}{{
\labellist
\tiny\hair 2pt
 \pinlabel {$3$} [ ] at 20 46
\endlabellist
\centering
\ig{.5}{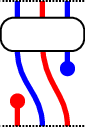}
}}
\newcommand{\threetrithree}{{
\labellist
\tiny\hair 2pt
 \pinlabel {$3$} [ ] at 20 18
 \pinlabel {$3$} [ ] at 20 62
\endlabellist
\centering
\ig{.5}{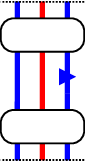}
}}
\newcommand{\othertriother}{{
\labellist
\tiny\hair 2pt
 \pinlabel {$3'$} [ ] at 20 18
 \pinlabel {$3'$} [ ] at 20 62
\endlabellist
\centering
\ig{.5}{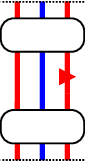}
}}
\newcommand{\othertsother}{{
\labellist
\tiny\hair 2pt
 \pinlabel {$3'$} [ ] at 20 18
 \pinlabel {$3'$} [ ] at 20 62
\endlabellist
\centering
\ig{.5}{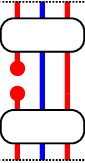}
}}
\newcommand{\othertsthree}{{
\labellist
\tiny\hair 2pt
 \pinlabel {$3'$} [ ] at 20 18
 \pinlabel {$3$} [ ] at 20 62
\endlabellist
\centering
\ig{.5}{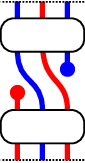}
}}
\newcommand{\threestother}{{
\labellist
\tiny\hair 2pt
 \pinlabel {$3$} [ ] at 20 18
 \pinlabel {$3'$} [ ] at 20 62
\endlabellist
\centering
\ig{.5}{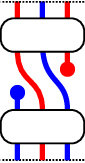}
}}
\newcommand{\wtosts}{\ig{.5}{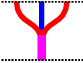}}
\newcommand{\ststow}{\ig{.5}{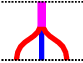}}
\newcommand{\sttow}{\ig{.5}{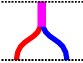}}
\newcommand{\tstow}{\ig{.5}{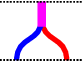}}
\newcommand{\wtost}{\ig{.5}{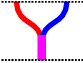}}
\newcommand{\wtots}{\ig{.5}{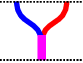}}

\newcommand{\wtos}{\ig{.5}{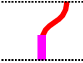}}
\newcommand{\wtot}{\ig{.5}{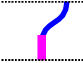}}
\newcommand{\wtoid}{\ig{.5}{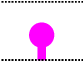}}
\newcommand{\idtow}{\ig{.5}{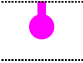}}
\newcommand{\wtostswtri}{\ig{.5}{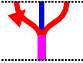}}
\newcommand{\ststowwtri}{\ig{.5}{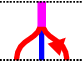}}
\newcommand{\pitchcupred}{\ig{.5}{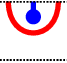}}

\begin{document}

\title{Categorical Diagonalization and $p$-cells}

\author{Ben Elias} \address{University of Oregon, Eugene}

\author{Lars Thorge Jensen} \address{Institute of Advanced Study, Princeton}

\dedicatory{Dedicated to George Lusztig, who built our playground.}

\maketitle

\begin{abstract} In the Iwahori-Hecke algebra, the full twist acts on cell modules by a scalar, and the half twist acts by a scalar and an involution. A categorification of this
statement, describing the action of the half and full twist Rouquier complexes on the Hecke category, was conjectured by Elias-Hogancamp, and proven in type $A$. In this paper we
make analogous conjectures for the $p$-canonical basis, and the Hecke category in characteristic $p$. We prove the categorified conjecture in type $C_2$, where the situation is
already interesting. The decategorified conjecture is confirmed by computer in rank $\le 6$; information is found in the appendix, written by
Joel Gibson. \end{abstract}

\tableofcontents

\section{New conjectures about $p$-cells} \label{sec:intro}

In \cite{EHDiag2}, the first-named author and Hogancamp made a conjecture (proven in type $A$) which can be summed up with the motto: the categorical full twist understands cell
theory. The goal of this paper is to state a more surprising conjecture: the categorical full twist in finite characteristic understands $p$-cell theory. We prove our conjecture in
type $C_2$. We also discuss some interesting phenomena which first occur in type $C_3$.

This first chapter is a working introduction which states the main conjectures of the paper. It is aimed at experts already familiar with the cell theory of the Hecke algebra, and
with its categorification, though perhaps not with diagrammatics, categorical diagonalization, or the $p$-canonical basis. One can skip from there to the chapter on type $C_3$,
aimed at the same audience. The remaining chapters are heavily computational, and aimed at the expert in diagrammatics. For additional background material we recommend \cite{EMTW}.

\subsection{Action of the half twist on the Kazhdan-Lusztig basis}

We begin by discussing the decategorified version of our conjecture.

Let $W$ be a finite Coxeter group with longest element $w_0$. Let $\HB(W)$ denote the Iwahori-Hecke algebra of $W$, a $\Z[v,v^{-1}]$-deformation of the group algebra of $W$. The \emph{half twist} $\half_W \in \HB(W)$ is the standard basis element associated to $w_0$. The \emph{full twist} $\full_W \in
\HB(W)$ is the square of the half twist, \[ \full_W = \half_W^2, \] and it is an important element in the center of $\HB(W)$. A common theme in representation theory is to diagonalize central operators, and this is what we explore here for the full twist.

Let $\{b_w\}_{w \in W}$ denote the Kazhdan-Lusztig basis of $\HB(W)$, which we call the \emph{KL basis}. If $\la$ is a two-sided cell, we let $I_{< \la} \subset \HB(W)$ denote the ideal spanned by KL basis elements in strictly lower cells. Here is a classic theorem; we discuss the attribution below.

\begin{thm} \label{thm:fullintrodecat}
There is an integer $\xb(\la)$ for each two-sided cell $\la$, such that for any $w \in \la$ we have \begin{equation} \label{eq:fulltake1} \full_W \cdot b_w \equiv v^{2 \xb(\la)} b_w \text{ modulo } I_{< \la}. \end{equation} \end{thm}

While the KL basis is not an eigenbasis for $\full_W$, it descends to an eigenbasis in the associated graded of the cell filtration. Morally speaking, cell theory controls the spectral theory of the full twist. Note that there may be multiple cells which share the same\footnote{In type $A$, this happens for the first time in type $A_5$, with the partitions $(3,1,1,1)$ and $(2,2,2)$. However, one can still distinguish between cells in type $A$ by simultaneously diagonalizing full twists of various sizes, see \cite{EHDiag2}.} value of $\xb(\la)$.

Theorem \ref{thm:fullintrodecat} is an immediate consequence of a more refined property of the half twist.

\begin{thm} \label{thm:halfintrodecat} There is an integer $\xb(\la)$ and a non-negative integer $\cb(\la)$ for each two-sided cell $\la$, such that for any $w \in \la$ we have \begin{equation}
\label{eq:halfintro} \half_W \cdot b_{w} \equiv (-1)^{\cb(\la)} v^{\xb(\la)} b_{\Schu_L(w)} \text{ modulo } I_{< \la}. \end{equation} Here $\Schu_L$ is some
involution on $W$ which preserves each left cell, generalizing the (left) Sch\"{u}tzenberger involution in type $A$. Moreover, $\la < \mu$ implies that $\cb(\la) < \cb(\mu)$ and $\xb(\la) < \xb(\mu)$. \end{thm}

Theorem \ref{thm:halfintrodecat} was proven for all finite types by Mathas \cite{Mathas96}, and generalized to Hecke algebras with unequal parameters by Lusztig \cite{LuszLongest15}.
Previously, in \cite[Corollary 5.9]{Lusztig90}, Lusztig proved an analogous result about the action of the half twist braid on the canonical basis of a tensor product representation for
the quantum group. In type $A$, Theorem \ref{thm:halfintrodecat} can be deduced from this result using Schur-Weyl duality. Another proof in type $A$ was given in Graham's thesis
\cite{GraThesis}.

Recall that the action of $w_0$ permutes the two-sided cells. The integers $\xb$ and $\cb$ can be determined using Lusztig's $\ab$-function, via
\begin{equation} \label{cxdef} \cb(\la) = \ab(w_0 \cdot \la), \qquad \xb(\la) = \ab(w_0 \cdot \la) - \ab(\la). \end{equation}
From \eqref{cxdef} it is clear that $\cb(\la)$ is a non-negative integer, but it only matters up to parity in \eqref{eq:halfintro}, and is invisible in \eqref{eq:fulltake1}. Regardless, the precise value of $\cb(\la)$ plays an important role in the categorification. To emphasize its importance we prefer to rewrite \eqref{eq:fulltake1} as
\begin{equation} \label{eq:fulltake2} \full_W \cdot b_w \equiv (-1)^{2 \cb(\la)} v^{2 \xb(\la)} b_w \text{ modulo } I_{< \la}. \end{equation}

In \cite[Theorem 3.1]{Mathas96} Mathas pins down the involution $\Schu_L$ uniquely using cell-theoretic properties. In \cite[p9 and following]{Mathas96} he explores further properties of $\Schu_L$ and in \cite[Corollary 3.14]{Mathas96} he gives a criterion implying that $\Schu_L$ is the identity on a given left cell, which he expects to be necessary as well. As far as we are aware, $\Schu_L$ only has an explicit combinatorial interpretation in type $A$. If desired, \eqref{eq:halfintro} can be viewed as the definition of $\Schu_L(w)$, which picks out the unique non-vanishing coefficient in $\half_W \cdot b_w$ modulo lower terms.


\begin{ex} This is the example we follow throughout the paper. Let $W$ have type $C_2$, with simple reflections $\{s,t\}$. There are three two-sided cells:
\begin{subequations}
\begin{equation} \la_1 = \{\id\}, \quad \la_{\bug} = \{s,t,st,ts,sts,tst\}, \quad \la_0 = \{w_0\}, \end{equation}
where $w_0$ denotes the longest element of $W$. We have
\begin{equation} \xb(\la_1) = 4, \cb(\la_1) = 4, \qquad \xb(\la_{\bug}) = 0, \cb(\la_{\bug}) = 1, \qquad \xb(\la_0) = -4, \cb(\la_0) = 0. \end{equation}
The Sch\"{u}tzenberger involution on $\la_{\bug}$ satisfies
\begin{equation} \Schu_L(s) = sts, \qquad \Schu_L(t) = tst, \qquad \Schu_L(st) = st, \qquad \Schu_L(ts) = ts. \end{equation}
Note that $b_s$ is in the same (left) cell as $b_{ts}$ because $b_t b_s = b_{ts}$ and $b_s b_{ts} = b_s + b_{sts}$, so each appears as a summand in the ideal generated by the other.
\end{subequations}
\end{ex}

\subsection{Action of the half twist on the $p$-canonical basis}

Cells are a notion intrinsic to an algebra with a chosen basis, and the KL basis is not the only interesting basis of $\HB(W)$. A recent player of great importance in modular
representation theory \cite{RicWil} is the \emph{$p$-canonical basis} $\{{}^p b_w\}$ (associated to a prime $p$), which is defined when $W$ is crystallographic, see \cite{JenWilpCan}. We recall the construction of the $p$-canonical basis in Definition \ref{defn:pcanbasis}.
This basis is mysterious: it can be computed in examples, but there is no known algebraic formula (the methods to compute it involve working with the Hecke category rather than the
Hecke algebra). The (two-sided) cells associated to the $p$-canonical basis are called \emph{$p$-cells}, and were first studied systematically by the second author in \cite{JensenABC}.

When the prime $p$ is understood, we write $\pb_w := {}^p b_w$ for ease of reading.

\begin{ex} \label{ex:B2char2} We continue the previous example. The only prime for which $\{\pb_w\}$ and $\{b_w\}$ disagree is $p=2$, so henceforth we set $p=2$. For most $w$ we have $b_w = \pb_w$, the one exception being that
\begin{subequations}
\begin{equation} \pb_{sts} = b_{sts} + b_s. \end{equation}
However, now $b_s \pb_{ts} = \pb_{sts}$ which does not have $\pb_s$ as a summand, so we can not use the same argument as before to deduce that $s$ and $ts$ are in the same cell. Indeed they are not. The big $0$-cell $\la_{\bug}$ now splits into two smaller $p$-cells:
\begin{equation} \la_s = \{s\}, \qquad \la_{\pbug} = \{t,st,ts,sts,tst\}. \end{equation}
The partial order on $p$-cells is
\begin{equation} \la_0 < \la_{\pbug} < \la_s < \la_1. \end{equation}
Readers familiar with distinguished involutions can note that $sts$ now behaves like a distinguished involution in $\la_{\pbug}$. Finally, note that $w_0$ does not act to permute the $p$-cells.
\end{subequations}
\end{ex}

The first question one can ask is whether there is an analog of \eqref{eq:halfintro} and \eqref{eq:fulltake2} for the $p$-canonical basis. If $\la$ is a $p$-cell, we let $I^p_{< \la}$ denote the ideal spanned by strictly lower $p$-cells.

\begin{conj} \label{conj:htdecat} Fix a Weyl group $W$ and a prime $p$. There are integers
$\pxb(\la)$ and $\pcb(\la)$ for each $p$-cell $\la$, together with a \emph{$p$-Sch\"{u}tzenberger involution} $\pSchu_L$ preserving the left $p$-cells, such that
\begin{equation} \label{eq:phalfintro} \half_W \cdot \pb_{w} \equiv (-1)^{\pcb(\la)} v^{\pxb(\la)} \pb_{\pSchu_L(w)} \text{ modulo } I^p_{< \la} \end{equation}
for any $w$ in $\la$. Moreover, $\la < \mu$ implies that $\pcb(\la) < \pcb(\mu)$. \end{conj}

Conjecture \ref{conj:htdecat} has been verified for all Weyl groups in rank $\le 6$, though the ``Moreover'' statement is not currently accessible by computer. We discuss this below.

\begin{rmk} \label{rmk:butnotx} In type $(C_3, p = 2)$ there is an example where $\la < \mu$ but $\pxb(\la) > \pxb(\mu)$. See \eqref{C3eigenvalues}. \end{rmk}

\begin{ex} \label{ex:B2char2againnn} We continue Example \ref{ex:B2char2}. Since $b_w = \pb_w$ for many values of $w$, we see that \eqref{eq:phalfintro} holds for many values of $w$, where the $p$-cells $\la_0$ and $\la_{\pbug}$ and $\la_1$ have the same values for $\xb$ and $\cb$ as their $0$-cell counterparts. The interesting computations occur when $w \in \{s,sts\}$. Below we rewrite the $p$-canonical basis in the KL basis, apply \eqref{eq:halfintro}, and then reinterpret again in the $p$-canonical basis. We have
\begin{subequations}
\begin{eqnarray} \half_W \cdot \pb_{sts} = \half_W \cdot (b_s + b_{sts}) &\equiv& (-1) (b_{sts} + b_s) \text{ modulo } I_{\la_0} \\ \nonumber &=&  (-1) \pb_{sts} \text{ modulo } I^p_{\la_0}. \end{eqnarray}
Note that the ideals $I_{\la_0}$ and $I^p_{\la_0}$ agree. Thus \eqref{eq:phalfintro} holds for $\pb_{sts}$, though $\pSchu_L(sts) = sts$ in contrast to $\Schu_L(sts) = s$. Meanwhile,
\begin{eqnarray} \label{scomp} \half_W \cdot \pb_s &\equiv&  (-1) b_{sts} \text{ modulo } I_{\la_0} \\ \nonumber &=& -\pb_{sts} + \pb_s \text{ modulo } I^p_{\la_0} \\ \nonumber &\equiv& (+1) \pb_s \text{ modulo } I^p_{< \la_s}. \end{eqnarray}
\end{subequations}
This is consistent with $\pxb(\la_s) = 0$ and $\pcb(\la_s) = 2$ and $\pSchu_L(s) = s$. In \eqref{scomp} only the parity of $\pcb$ is evident, but the precise value $2$ comes from the categorification.
\end{ex}

Let us explain the simple reason why the value of $\pxb$ should be constant on two-sided $p$-cells.

\begin{prop} Suppose that $\full_W \cdot \pb_w = v^{2 \pxb(w)} \pb_w$ modulo lower terms, for some integer $\pxb(w)$ depending on $w$. Then $\pxb(w)$ is constant on two-sided
$p$-cells. \end{prop}

\begin{proof} We first argue that the value of $\pxb(w)$ is constant on left $p$-cells. By definition, the left $p$-cell is indecomposable as a \emph{based} left module over the Hecke
algebra. The full twist $\full_W$ is in the center of the Hecke algebra, so acts by an intertwiner on any module. Under the assumption, the $p$-canonical basis of the left cell module
is an eigenbasis, so the eigenspaces are based submodules. If there were distinct eigenvalues then the left cell module would split as a based module, a contradiction.

Since $\full_W$ is central, right multiplication by $\full_W$ also has the $p$-canonical basis as an eigenbasis modulo lower terms, with the same eigenvalues. By the same argument, these eigenvalues are constant on right cells. \end{proof}

\subsection{Numerics and evidence}

Let us write $\xb(w) = \xb(\la)$ when $w \in \la$. For example, in type $C_2$ as above, we have $\xb(s) = \pxb(s) = 0$ and $\cb(s) = 1$ and $\pcb(s) = 2$. The behavior of $s$ is
typical: it seems that $\pcb(w) > \cb(w)$ when $w$ moves to a higher $p$-cell relative to its $0$-cell compatriots. In type $C_4$, the longest element $w$ of the parabolic subgroup
of type $A_3$ appears to have $\pcb(w) = \cb(w) + 2$. There is no reason to expect a global bound on $\pcb(w) - \cb(w)$ in general.

\begin{remark} A small rank heuristic is that $\pcb(w)$ equals the length of the longest chain of $p$-cells from $w$ to $\la_0$. However, this will certainly fail in general (it fails in type $D_4$), as the values of $\pcb$ may skip numbers. \end{remark}

From the example of type $C_2$ one might get the false impression that each $p$-cell $\la'$ is contained in a unique $0$-cell $\la$, but this is not true in general. Moreover, it need
not be the case\footnote{Note that $\pxb(w) = \xb(x)$ for some $x$, since the eigenvalues of the full twist do not depend on the choice of basis.} that $\xb(w) = \pxb(w)$.
Counterexamples arise in type $C_3$, related to the appearance of a non-perverse $p$-canonical basis element. In the same counterexample is a fascinating surprise: two elements swap
eigenvalues! A thorough discussion of this computation, and musings on its categorical significance, can be found in \S\ref{sec:C3}.

No general theory currently exists for determining the values of $\pxb(\la)$ or $\pcb(\la)$. Many other features of ordinary cell theory also do not yet have
analogues for the $p$-canonical basis. The naive analogue of Lusztig's $\ab$-function is not constant on $p$-cells, and no suitable replacement is known\footnote{If Conjecture \ref{conj:htdecat} holds, then following \eqref{cxdef} one could define $\pab(w) := \pcb(w) - \pxb(w)$. This statistic is not monotone over the cell order. Whether this is a satisfactory analogue of Lusztig's $\ab$-function for other purposes remains to be seen.}. The longest element $w_0$ does
not permute the $p$-cells. There is a conjectural definition of $p$-distinguished involutions due to the second author, but the theory is not fully developed. This conjecture does not
yet appear in print elsewhere so we place it below as Conjecture \ref{conj:distinv}. These are the major tools used to define $\xb$ and $\cb$ and $\Schu_L$, and they are not yet
operational for $p$-cell theory. The techniques used by Mathas do not seem to adapt well to the $p$-canonical basis.

\begin{notation} \label{notation:val} For a Laurent polynomial $f$ write $\val(f)$ for the smallest exponent appearing with nonzero coefficient. E.g. $\val(2v^{-1} + 3v) = -1$. Let $\mu_{w,x}^y$ be the coefficient of $\pb_y$ inside the $p$-canonical basis expansion of $\pb_w \cdot \pb_x$. Let $h_x^y$ be the coefficient of the standard basis element $H_y$ inside $\pb_x$. We use the same notation for the ordinary KL basis as well, when the context calls for it. The notation $\ab(w)$ or $\ab(\la)$ will never be used in relation to the $p$-canonical basis and its coefficient, and always refers to Lusztig's usual $\ab$-function. \end{notation}

\begin{conj} \label{conj:distinv}  In each left $p$-cell there is a unique involution $d$ for which $-\val(\mu_{d,d}^d) \ge \val(h_d^1)$. We call these the \emph{$p$-distinguished involutions}. \end{conj}

\begin{rmk} \label{rmk:distinvinequality} For the KL basis, $-\val(\mu_{d,d}^d) \le \val(h_d^1)$ for all $d \in W$ (involution or not), while only distinguished involutions have an equality, in which case $\val(h_d^1) = -\val(\mu^d_{d,d}) = \ab(d)$. For the $p$-canonical
basis, $-\val(\mu_{d,d}^d) > \val(h_d^1)$ can occur. \end{rmk}


The involution $\pSchu_L$ is also mysterious. Often it is forced to disagree with the characteristic zero involution $\Schu_L$, as an orbit of size two in a left $0$-cell gets split into distinct $p$-cells. This happens in Example \ref{ex:B2char2againnn} with the orbit $\{s,sts\}$. What happens can be unpredictable, as seen in the following example.

\begin{ex} In type $G_2$ one has $\Schu_L(s) = ststs$ and $\Schu_L(sts) = sts$. When $p = 3$ the element $s$ forms its own two-sided cell, and instead we have $\pSchu_L(s) = s$ and $\pSchu_L(sts) = ststs$. See \S\ref{app:rank2}. \end{ex}

Let us discuss the evidence for Conjecture \ref{conj:htdecat}.

First we consider type $A$. In type $A_n$ for $n \le 6$, the KL basis and the $p$-canonical basis agree for all primes. The second author proved in \cite[Theorem 4.33]{JensenABC} that
the $p$-cells agree with the $0$-cells in type $A_n$ for all $n$, even though the $p$-canonical basis can be quite different from the KL basis. Although the combinatorics of $p$-cells
is unchanged\footnote{It is only the combinatorics of two-sided cells which is known to be unchanged. There are no known examples where the partial order on left $p$-cells does change, but there is currently no proof that it doesn't change. Thankfully, it was proven in \cite[Corollary 4.8]{JensenCellular} that left $p$-cells within a given two-sided $p$-cell are still incomparable.}, the actual cell ideals in the Hecke algebra do change (e.g. $I_{\le \la} \ne I^p_{\le \la}$), because $\pb_w - b_w$ often involves KL basis elements in higher cells than $w$.
This makes it hard to compare the associated graded of the cell filtration between the ordinary and $p$-canonical settings. Worse still, starting in type $A_{11}$, Williamson
\cite{WillSameCell} found examples where $\pb_w - b_w$ involves KL basis elements in the same cell as $w$. This does not invalidate Conjecture \ref{conj:htdecat}, it only makes it more interesting if true, because \eqref{eq:halfintro} and \eqref{eq:phalfintro} truly diverge. Sadly, computing the $p$-canonical basis in type $A_{11}$ is beyond the reach of current computer programs.

\begin{rmk} In recent work, Lanini and McNamara \cite{LaniniMcNamara21} give a general method whereby a nonzero difference $\pb_w - b_w$ for some $w \in S_n$ will give rise to a nonzero difference $\pb_y - b_y$ living in the same cell, for a certain $y \in S_N$ with $N > n$. \end{rmk}

\begin{rmk} Suppose that the left module for the Hecke algebra coming from a left $p$-cell is isomorphic to the ordinary left $0$-cell module as based modules. Then \eqref{eq:halfintro} implies \eqref{eq:phalfintro} for this cell. The second author proved in \cite[Corollary 4.39]{JensenABC} that all left $p$-cell modules in the same two-sided $p$-cell are isomorphic in type $A$. Unpublished computer calculations done by Williamson suggest that there are no unusual $W$-graphs in type $A_n$ with $n \le 10$, which should imply that the $p$-cell and the $0$-cell give isomorphic based modules in these ranks. \end{rmk}

Outside of type $A$, the $p$-canonical basis is only known in ranks $\le 6$, though much of this data is not yet in print. In rank $\le 4$, one can
find the $p$-canonical basis and the additional code which verifies Conjecture \ref{conj:htdecat} online \cite{HeckeCode}.
These $p$-canonical bases were found by a mix of theoretic work and computer calculation due to Geordie Williamson and the second author. The additional code was written by Joel
Gibson, who has summarized the results of this verification in the appendix.

\subsection{Action of the half twist complex} \label{sec:halftwistcomplexintro}

The formula \eqref{eq:halfintro} was recently categorified in type $A$ by the first author and Matt Hogancamp in \cite[Proposition 4.31, Theorem 6.16]{EHDiag2}, which we now explain.

Let $\HC^0(W)$ denote the Hecke category defined over a field of characteristic $0$. It is a graded additive monoidal category whose Grothendieck group is isomorphic as a ring to the Hecke algebra. In characteristic zero the Hecke category appears in many guises (e.g. using Soergel bimodules, or equivariant perverse sheaves on flag varieties, or projective functors on category $\OC$). In Definition \ref{defn:heckecategory} we review a presentation of the Hecke category by generators and relations. Further references can be found there.

For now we just recall the essential properties of $\HC^0(W)$. It is monoidally generated by objects denoted $B_s$, one for each simple reflection. Its indecomposable objects are
parametrized up to grading shift and isomorphism by $\{B_w\}_{w \in W}$, and the objects $\{B_w\}$ categorify the KL basis $\{b_w\}$. The cell ideals $I_{< \la}$ lift to monoidal ideals
$\IC_{< \la}$. Rouquier \cite{RouqBraid-pp} indicated how one should construct a chain complex for any word in the generators of the braid group of $W$, which depends only on the braid up
to unique homotopy equivalence. Thus to any braid, including the half twist and full twist braids, one has a well-defined object in the bounded homotopy category $K^b(\HC^0)$. One
typically studies Rouquier complexes (and other objects of the homotopy category) by considering their \emph{minimal complex}, the complex obtained after using homotopy equivalence to
remove contractible direct summands. The minimal complex is unique up to isomorphism of complexes (rather than up to homotopy equivalence), and its chain objects in various homological
degrees are invariants of the braid.

\begin{conj} \label{conj:htcatchar0} (See \cite[Conjecture 4.30]{EHDiag2}) Let $W$ be a finite Coxeter group, and  let $\HT_W$ denote the half-twist Rouquier complex in the homotopy category $K^b(\HC^0)$. We have \begin{equation} \label{eq:HTintro} \HT_W \ot B_w \cong \left(
\ubr{\ldots}{\IC_{< \la}} \to B_{\Schu_L(w)}(\xb(\la))[\cb(\la)] \to 0 \right). \end{equation} That is, the indecomposable object $B_w$ is sent by $\HT_W$ to a complex,
whose minimal complex consists of \begin{itemize} \item one copy of $B_{\Schu_L(w)}$ in homological degree $\cb(\la)$ and with a grading shift by
$\xb(\la)$, \item various objects in strictly lower cells and strictly lower homological degree. \end{itemize} \end{conj}

This conjecture was proven in type $A$ in \cite[Proposition 4.31, Theorem 6.16]{EHDiag2}. A proof for dihedral types is work in preparation.

That \eqref{eq:HTintro} categorifies and implies \eqref{eq:halfintro} is obvious, though the precise positions of objects in particular homological degrees are not visible in the Grothendieck group (only the parity of the homological shift is seen). We must emphasize a part of \eqref{eq:HTintro} which is invisible in
the Grothendieck group: the fact that $B_{\Schu_L(w)}(\xb(\la))[\cb(\la)]$ is the unique object in the \emph{maximal} homological degree\footnote{In the terminology of \cite[\S 3.4]{EHDiag2}, the conjecture implies that the half twist is \emph{sharp}: it is a fine enough tool to effectively separate cells using homological degree.} of $\HT_W \ot B_w$. Since the inclusion of the final term in a bounded chain complex is a chain map, this means there is a chain map from the one-term complex containing
$B_{\Schu_L(w)}(\xb(\la))[\cb(\la)]$ into $\HT \ot B_w$. This chain map becomes a homotopy equivalence modulo $\IC_{< \la}$. In the next section we explain in what sense this chain map is functorial. First we discuss characteristic $p$.

Suppose $W$ is crystallographic. Using the same presentation by generators and relations, one can define the Hecke category $\HC^p(W)$ over a field of characteristic $p > 0$. Again, see Definition \ref{defn:heckecategory} for details. It is still generated by objects denoted $B_s$. The indecomposable objects in $\HC^p(W)$ are denoted
$\{{}^p B_w\}_{w \in W}$, and are still parametrized by $W$, but what they are is mysterious. Again, when the prime $p$ is understood we write $\pB_w$ instead of ${}^p B_w$ for
ease of reading. Note that $\pB_s = B_s$. The Grothendieck group of $\HC^p(W)$ is still isomorphic as a ring to the Hecke algebra. The $p$-canonical basis $\{\pb_w\}$ is defined to be the images of the symbols of the indecomposable objects $\{\pB_w\}$ (for a more formal definition, see Definition \ref{defn:pcanbasis}). In particular, the $p$-cells
correspond to monoidal ideals $\IC^p_{< \la}$ which differ from their characteristic zero counterparts.

Rouquier complexes make perfect sense inside $K^b(\HC^p)$, so that the half twist and full twist complexes are still well-defined. These complexes are largely unstudied and mysterious, having differently-behaved minimal complexes relative to their characteristic zero counterparts. For example, the half twist in $K^b(\HC^0)$ is \emph{perverse} in that each indecomposable summand of a chain object has a grading shift equal to its homological shift. This is false for the half twist in $K^b(\HC^p)$. For example, in type $C_2$, \eqref{HTchar0} demonstrates the minimal complex of the half twist in characterstic zero, which is perverse. In contrast, \eqref{HTchar2} demonstrates the minimal complex in characteristic $2$, where one summand sticks out like a sore thumb: a copy of $B_s(1)$ in homological degree $2$.

\begin{conj} \label{conj:htcat} Fix a Weyl group $W$ and a prime $p$. Let $\pxb$, $\pcb$, and $\pSchu$ be defined as in Conjecture \ref{conj:htdecat}. Then for any $w$ inside the $p$-cell $\la$ we have
\begin{equation} \label{eq:HTintrop} \HT_W \ot \pB_w \cong \left( \ubr{\ldots}{\IC^p_{< \la}} \to \pB_{\pSchu_L(w)}(\xb^p(\la))[\cb^p(\la)] \to 0 \right). \end{equation}
\end{conj}

This conjecture is not obvious even in type $A$, where the $p$-cells agree with the $0$-cells. How this conjecture plays out in examples is rather interesting.

\begin{ex} \label{ex:2torsion}
We continue the example of type $C_2$ in characteristic $2$. Note that $\pB_s = B_s$ and $\pB_{sts} = B_s B_t B_s$. Here is the minimal complex for $\HT \ot \pB_s$ (formulas for the differential can be found in \eqref{HTBschar2}):
\begin{equation} \label{HTBschar2intro} \HT \ot \pB_s \cong \left(
\begin{tikzpicture}
\node (a) at (0,0) {$\un{B}_{w_0}(-1)$};
\node (b) at (2.5,0) {$B_s B_t B_s(0)$};
\node (c) at (5,0) {$B_s(0)$};
\path
	(a) edge (b)
	(b) edge (c);
\end{tikzpicture} \right). \end{equation}
This matches Conjecture \ref{conj:htcat} with $\pcb(s) = 2$ and $\pSchu_L(s) = s$. The inclusion $\iota$ of $B_s(0)[2]$ into the last degree of \eqref{HTBschar2intro} is a homotopy equivalence modulo lower cells.

In fact, one can define the complex \eqref{HTBschar2intro} over $\Z$, and then specialize to other base rings. After inverting $2$, the differential from $B_s B_t B_s$ to $B_s$ is projection to a summand, and the complementary summand is $B_{sts}$. Applying Gaussian elimination easily transforms \eqref{HTBschar2intro} into
\begin{equation} \label{HTBschar0intro} \HT \ot B_s \cong \left(
\begin{tikzpicture}
\node (a) at (0,0) {$\un{B}_{w_0}(-1)$};
\node (b) at (2.5,0) {$B_{sts}(0)$};
\path
	(a) edge (b);
\end{tikzpicture} \right). \end{equation}
This matches Conjecture \ref{conj:htcatchar0} with $\cb(s) = 1$, since $\Schu_L(s) = sts$.

Note that \eqref{HTBschar2intro} is an indecomposable complex over $\Z$, and only becomes decomposable when $2$ is inverted. The chain map $\iota$ can also be defined over $\Z$, and $2 \iota$ is nulhomotopic! \end{ex}

\begin{thm} Conjecture \ref{conj:htcat} holds for the dihedral Weyl groups $A_2$, $C_2$, and $G_2$. \end{thm}

For dihedral groups it is tractable to compute $\HT$ and $\HT \ot \pB_w$ by brute force, which is how we prove the theorem. The main method is a tedious process of Gaussian elimination of
complexes, but we use some tricks to speed the process. In this paper we discuss the computation for $C_2$, omitting the case of $G_2$ for reasons of space. We also omit the computation
that type $A_2$ is characteristic independent. The results are far more interesting than the proofs; beyond explaining our methods and giving some illustrative examples, we spare the
reader most of the details. Full details of the computation for $C_2$ can be found in the supplementary document \cite{C2supplement}, which will eventually be enlarged to contain $G_2$ as
well.

We present our musings on type $C_3$ in \S\ref{sec:C3}.

\subsection{Diagonalizing the full twist: characteristic zero}

Earlier we alluded to the fact that the chain map
\[ B_{\Schu_L(w)}(\xb(\la))[\cb(\la)] \to \HT \ot B_w \]
might be functorial. If there were a functor $F^\la$ (depending on $\la$) which sent
\[ B_w \mapsto B_{\Schu_L(w)}(\xb(\la))[\cb(\la)],\] then ``functoriality'' would be the existence of a natural transformation $F^\la \to \HT \ot (-)$ giving rise to the chain map under discussion. However, no such functor $F^\la$ is expected to exist.

Instead, let us consider the full twist complex $\FT_W := \HT_W \ot \HT_W$. Even though $F^\la$ need not exist, its square would be $\one(2\xb(\la))[2\cb(\la)]$, where $\one$ is the identity functor. This functor is a categorification of the eigenvalue $(-1)^{2 \cb(\la)} v^{2 \xb(\la)}$ from \eqref{eq:fulltake2}. It is a relatively straightforward consequence of \eqref{eq:HTintro} (see \cite[Lemma 3.26 and Proposition 3.31]{EHDiag2}) that
\begin{equation} \label{eq:FTintro} \FT_W \ot B_w \cong \left( \ubr{\ldots}{\IC_{< \la}} \to B_w(2\xb(\la))[2\cb(\la)] \to 0 \right) \end{equation}
when $w$ is in cell $\la$.
The inclusion of the final term in this chain complex would be a chain map
\[ B_w(2\xb(\la))[2\cb(\la)] \to \FT_W \ot B_w, \]
which is a homotopy equivalence modulo $\IC_{< \la}$. This homotopy equivalence could be the action of a natural transformation
\[ \al_\la \co \one(2\xb(\la))[2\cb(\la)] \to \FT_W \]
when applied to the object $B_w$.

\begin{defn} Assume that \eqref{eq:FTintro} holds for all $w$ in (two-sided) cell $\la$. A natural transformation $\al_\la \co \one(2\xb(\la))[2\cb(\la)] \to \FT_W$ is called a
\emph{$\lambda$-eigenmap} if $\al_\la \ot B_w$ is a homotopy equivalence modulo $\IC_{< \la}$, for all $w$ in cell $\la$. Equivalently, one can ask that the map induced by $\al_{\la} \ot B_w$ in homological degree $2 \cb(\la)$ (to the minimal complex of $\FT \ot B_w$) is an automorphism of $B_w(2 \xb(\la))$. \end{defn}

An eigenmap represents the ``functorial relationship'' between the operator $\FT_W \ot (-)$ and its categorified eigenvalue. In \cite[Theorem 7.37]{EHDiag2} it was proven (with difficulty) that eigenmaps do indeed exist for each $0$-cell $\la$ in type $A$. This is conjectured (in characteristic zero) for all finite Coxeter groups in \cite[Conjecture 4.32]{EHDiag2}, and proven for dihedral groups in work in preparation \cite{EHFullTwistDihedral}. Here are several remarks on this result, before discussing its importance.

\begin{rmk} There are distinct cells $\la \ne \la'$ which $\full_W$ can not tell apart, satisfying $\xb(\la) = \xb(\la')$ and $\cb(\la) = \cb(\la')$. Conjecturally, $\FT_W$ can distinguish between these cells! We can find a chain map $\al$ which is an eigenmap for cell
$\la$ but not for cell $\la'$, and vice versa (see \cite[\S 2]{EHDiag} for a discussion of this complicated situation).
\end{rmk}

\begin{rmk} \label{rmk:indegree2c} The half twist $\HT_W$ is perverse, so the chain objects in its minimal complex are determined in the Grothendieck group. From this one can prove (see \cite[Proposition 4.21, Corollary 4.29]{EHDiag2}) that the chain objects appearing in homological degree $\cb(\la)$ and cell $\la$ are precisely $\bigoplus B_{\Schu_L(d)}(\cb(\la))$, where $d$ ranges over the distinguished involutions in $\la$. Meanwhile, the full twist is not perverse and much less is known about its minimal complex. One consequence of \eqref{eq:FTintro} and the existence of eigenmaps (see \cite[Proposition 4.31]{EHDiag2}) is that, for the chain objects of the full twist in homological degree $2\cb(\la)$, the summands in cell $\la$ are precisely $\bigoplus B_d(\xb(\la) + \cb(\la))$. \end{rmk}


In the Grothendieck group, the fact that $\full_W$ is diagonalizable with known eigenvalues implies that
\begin{subequations}
\begin{equation} \label{diagdecat} \prod_{\la} (\full_W - v^{2 \xb(\la)}) = 0. \end{equation}
A reasonable categorical lift of the operator $(\full_W - v^{2 \xb(\la)})$ would be the cone of the chain map $\al_{\la}$, the \emph{eigencone}. Then the categorical analogue of \eqref{diagdecat} would be
\begin{equation} \label{diagcat} \bigotimes_{\la} \Cone(\al_{\la}) \simeq 0, \end{equation}
\end{subequations}
in which case we say that $\FT_W$ is \emph{categorically (pre)diagonalizable}. If each $\al_{\la}$ is an eigenmap and certain homological obstructions\footnote{These obstructions measure to what extent the eigencones tensor-commute.} vanish, one can prove that $\FT_W$ is categorically prediagonalizable, see \cite[Proposition 3.42]{EHDiag2}.


Given a diagonalizable operator in linear algebra, Lagrange interpolation gives a method to contruct idempotents projecting to eigenspaces, allowing one to deduce that the vector space
is the direct sum of its eigenspaces. It is this eigenspace decomposition, rather than the equation \eqref{diagdecat}, which is commonly used in practice. In the categorical setup,
given a prediagonalizable endofunctor, one can hope for idempotent functors which project to \emph{eigencategories} (see Remark \ref{rmk:eigenobject} below). A direct sum decomposition
into eigencategories is too much to ask for, but one can hope that the category being acted upon has a filtration whose subquotients are eigencategories. A formalization of these hopes
is made in \cite[Definition 6.16]{EHDiag}, and a prediagonalizable endofunctor for which nice projection functors exist is called \emph{categorically diagonalizable}.

The main theorem of \cite{EHDiag} is a categorification of the Lagrange interpolation construction whereby, given a prediagonalizable functor whose eigenvalues are
\emph{homologically distinct} (i.e. different cells have different values of $\cb$), one can construct projection functors and prove categorical diagonalizability. In dihedral type
$\FT_W$ has homologically distinct eigenvalues. In type $A_n$ for $n \ge 5$ the full twist does not, though \cite{EHDiag2} is able to prove the categorical diagonalizability of the
full twist using additional techniques.

\begin{rmk} \label{rmk:eigenobject} Consider a chain map $\al$ (such as $\al_{\la}$ above) from a shift of the identity $\one(a)[b]$ to a complex $F$. An \emph{eigenobject} for $\al$ is a complex $M$ such that $\al \ot M : M(a)[b] \to F \ot M$ is a homotopy equivalence.
Eigenobjects are the objects in the \emph{$\al$-eigencategory}, a full triangulated subcategory of the homotopy category. Recall that the KL basis is not actually an
eigenbasis for $\full_W$, it is only an eigenbasis modulo lower cells. Similarly, $B_w$ is not a true eigenobject for $\al_{\la}$, only being an eigenobject modulo $\IC_{< \la}$.
However, using projection functors, one can construct genuine eigenobjects for each $\al_{\la}$, justifying the fact that they are called eigenmaps. Note that both projection
functors and eigenobjects are typically infinite complexes! Eigenobjects-modulo-lower-terms are easier to work with than true eigenobjects. \end{rmk}

\subsection{Diagonalizing the full twist: characteristic $p$}

Now we pass to characteristic $p$. There are potentially more $p$-cells than $0$-cells. Though they have the same set of eigenvalues $v^{2\xb(\la)}$ for $\full_W$ (powers of $v^2$), if one keeps track of the invisible factor of $(-1)^{2 \cb(\la)}$ then new ``eigenvalues'' appear. The main conjecture in this paper is the following, which states that these new $p$-cells also admit new eigenmaps.

\begin{conj} \label{conj:ftdiag} Fix a Weyl group $W$ and a prime $p$. Assume Conjecture \ref{conj:htcat}. For any $p$-cell $\la$ there exists a
chain map \begin{equation} \al_{\la} \co \one(2\pxb(\la))[2 \pcb(\la)] \to \FT_W \end{equation}
which is a \emph{$\la$-eigenmap}. This means that, for any $w \in \la$, $\al_{\la} \ot \pB_w$ is a homotopy equivalence modulo $\IC^p_{< \la}$. \end{conj}

\begin{rmk} \label{rmk:indegree2credux} The properties of the half and full twist discussed in Remark \ref{rmk:indegree2c} should
have analogues in characteristic $p$ as well. That is, the $p$-distinguished involutions from Conjecture \ref{conj:distinv} should govern which chain objects appear in certain
homological degrees within the half and full twists. \end{rmk}

To give some justification for this conjecture, we prove our main theorem.

\begin{thm} \label{thm:C2char2diag} Conjecture \ref{conj:ftdiag} holds in type $C_2$ in characteristic $2$. \end{thm}

If Conjecture \ref{conj:ftdiag} holds then, at least for dihedral groups and other Coxeter groups where the eigenvalues are homologically distinct (i.e. $\pcb(\la) \ne \pcb(\mu)$
for distinct $p$-cells $\la \ne \mu$), the machinery from \cite{EHDiag, EHDiag2} immediately applies to prove that $\FT_W$ is categorically diagonalizable. We do not spell
this out in type $C_2$ for reasons of brevity.

To prove Theorem \ref{thm:C2char2diag}, we compute the full twist over $\Z$, with the answer given in \eqref{FTchar2}. We explain our methods in \S\ref{C2methods}, and omit the
gory details (found in \cite{C2supplement}). Then we can explicitly construct the candidate eigenmaps in \S\ref{C2eigen}, and they have a very satisfying description: they are
built from the unit maps of certain Frobenius algebra objects. We explain the connection between Frobenius algebra objects, distinguished involutions, and eigenmaps in the
following sections. The eventual proof is found in \S\ref{frobunitandeigen}, see Corollary \ref{cor:mainthmproven}.

\begin{rmk} We remark on the difficulty of this theorem. As mentioned, the work in preparation \cite{EHFullTwistDihedral} proves the characteristic zero version of Theorem
\ref{thm:C2char2diag}, also by explicitly computing the full twist (for any dihedral group). These characteristic zero full twists have a great deal of interesting structure; while we do not showcase that structure here, we give a hint by displaying an interesting Koszul-like complex in \S\ref{subsec:koszul}.

Meanwhile, the computation in finite characteristic is an order of magnitude more difficult, and lacks many useful tools available in the characteristic zero setting. While our
computations for the half twist in type $C_2$ involved only a handful of Gaussian eliminations (6 pages of work), the computation of the full twist was enormous (60 pages of work)
and quite involved. At the moment there is no organizing structure, only a mass of computations, though we hope one day this will be rectified. \end{rmk}

So, where do the extra eigenmaps come from? The example of type $C_2$ sheds light on this question.

\begin{ex} The eigenmap for the $0$-cell $\la_{\bug}$ lifts over $\Z$ in a fairly straightforward way, and descends to an eigenmap for the $p$-cell $\la_{\pbug}$. Meanwhile, the extra
$p$-cell $\la_s$ has eigenvalue $\one(0)[4]$, a homological degree unmatched by any $0$-cell, so its eigenmap is not an adaptation of any characteristic zero eigenmap. In characteristic
zero there is a one-dimensional space of chain maps $\one(0)[4] \to \FT_W$ modulo homotopy, and the most obvious generator has a lift over $\Z$ which we denote $\phi$.

Let $\HOM$ denote the bigraded space of morphisms (of all graded and homological degrees) in the homotopy category. In type $C_2$ over $\Z$, one can compute that $\HOM(\one, \FT_W)$ is
free\footnote{The space of chain maps from a shift of $\one$ to $\FT_W$ is free over $\Z$, but it is not obvious that the same should be true of chain maps modulo homotopy. In type $A$,
Hom spaces in the homotopy category were computed in \cite{ElHog16a}, using spectral sequences which degenerated due to parity. One hopes that similar parity arguments will also apply
over $\Z$, allowing one to prove freeness and other properties of full twists.} over $\Z$. The subspace of morphisms in the ideal $\IC_{< \la_s}$ (which one can define over $\Z$) form a
sublattice. In homological degree $4$, this sublattice has index $2$. In fact, $\phi \in \IC_{< \la_s}$. The new eigenmap $\al_{\la_s}$ (in characteristic $2$) does lift over $\Z$, and $2 \al_{\la_s}$ is homotopic to $\phi$. See \S\ref{eigentorsion} for details.

In summary, the new eigenmap does not come from a torsion element in $\HOM(\one, \FT_W)$, but from an element which is torsion modulo the sublattice $\IC_{< \la_s}$. This is in contrast
to Example \ref{ex:2torsion}, where $\al_{\la_s} \ot B_s$ produces a $2$-torsion element in $\HOM(B_s, \HT_W \ot B_s)$. Another takeaway is that, when working over $\Z$, $\HOM(\one,
\FT_W)$ is bigger than the experts familiar with characteristic zero may expect. While the map $\phi$ is the obvious generator in characteristic zero, it only generates an index $2$
sublattice over $\Z$. \end{ex}

\begin{rmk} The beautiful conjectures of Gorsky-Negut-Rasmussen \cite{GNR} (now mostly addressed in works of Oblomkov-Rozansky) have placed the study of the full twist for the
symmetric group $S_n$ in a new geometric perspective. They posit that the Drinfeld center of the Hecke category is equivalent to the category of equivariant coherent sheaves on (a
particular version of) the flag Hilbert scheme of $n$ points on the plane, with the full twist corresponding to $\OC(1)$. Thus maps from $\one$ to
$\FT_W$ correspond to sections of $\OC(1)$. The eigenmaps form a special set of linear sections on this quasi-projective variety. One expects there to be a similarly
interesting quasi-projective variety appearing for other finite Coxeter groups. Our observations in this paper suggest that the coordinate rings of these varieties have an interesting
integral form whose study could shed light on the geometry of $p$-cells. \end{rmk}

\subsection{Outline and acknowledgments}

In \S\ref{sec:intro}, which perhaps you already read, we stated the main conjectures and theorems in this paper. In \S\ref{sec:C2diag} we provide background on the diagrammatic
Hecke category in type $C_2$, using the thick calculus established in \cite{ECathedral}. In \S\ref{sec:C2} we provide minimal versions of various complexes (the half twist, the
full twist, and their action on various indecomposables) in type $C_2$ in both characteristic $0$ and $2$. We provide the eigenmaps in \S\ref{C2eigen}. In \S\ref{frobalgobj}
through \S\ref{frobunitandeigen} we discuss Frobenius algebra objects associated to distinguished involutions, and how eigenmaps can optimistically built from their unit maps.
Since this does occur in type $C_2$, we can prove Theorem \ref{thm:C2char2diag} in \S\ref{frobunitandeigen}. Afterwards, we discuss the extra eigenmap in \S\ref{eigentorsion}, and
discuss our methods of computation in \S\ref{C2methods}. The bulk of the computations can be found explicitly in the supplemental document \cite{C2supplement}.

In \S\ref{sec:C2} we assume the reader is familiar with the technique known as Gaussian elimination for complexes. Everything we use about this technique can be found in \cite[\S 5.4]{EGaitsgory}.

Because there are few references on dihedral Hecke categories, especially in finite characteristic, we have tried to make \S\ref{sec:C2diag} a useful resource. We have included
everything needed to understand the complexes of \S\ref{sec:C2}, but also a few extra topics of interest, such as the Koszul-like complex in \S\ref{subsec:koszul}.

In \S\ref{sec:C3} we give a high-level discussion of one interesting numerical situation which arises in type $C_3$, and hypothesize on the categorical explanation. Though mostly
guesswork this chapter contains some interesting food for thought. One can skip from \S\ref{sec:intro} directly to \S\ref{sec:C3}.

The appendix \S\ref{appendix}, written by Joel Gibson, and based on calculations originally done by Geordie Williamson and the second author, describes the differences between the Kazhdan-Lusztig cells and the $p$-cells for Weyl groups in rank $\le 6$. Graphs detail the places where $(\xb,\cb)$ and $(\xb^p,\cb^p)$ disagree in rank $\le 4$. For several Weyl groups, this is the first time that the $p$-cell partial order has appeared in print.

{\bf Acknowledgments} The authors would like to especially thank Joel Gibson for his assistance in verifying the main conjecture, for making the data so easy to parse, for writing
the appendix, and for keeping the authors from making incorrect conjectures (see e.g. Remark \ref{rmk:butnotx}). It was very helpful, when pursuing an idea, to know promptly
whether it was true in all known cases or not! We would also like to thank Geordie Williamson for making these computations possible, and for countless useful conversations. We wish to thank George Lusztig for his interest and helpful comments. Finally, we would like to thank the anonymous referee for many helpful suggestions.


This paper began with a visit by the second author to University of Oregon, a trip supported by NSF grant DMS-1553032. During this paper's long journey, the first author was
supported by NSF grants DMS-1553032 and DMS-1800498 and DMS-2039316. This paper was partially written while the authors were visiting the IAS, a visit supported by NSF grant
DMS-1926686.

\section{Type $C_2$ diagrammatics} \label{sec:C2diag}

In an effort to make our computations more accessible to anyone attempting a deep study, we provide here a primer on morphisms in the Hecke category in type $C_2$ (over $\Z$). The reader interested in seeing the results (the half twist and full twist complexes) but without needing to understand the differentials should skip this chapter entirely.

For background on the diagrammatic Hecke category, see \cite{EMTW} or \cite{EWGr4sb}.

\subsection{Setup and notation}

\begin{notation} Let $\Bbbk$ be a commutative domain, which we call the \emph{base ring}. In this paper, we will typically set $\Bbbk$ to be $\Z$ or a field.
 For technical reasons, we fix an element $\kappa \in \Bbbk$. \end{notation}

A \emph{realization} of a Coxeter system $(W,S)$ over $\Bbbk$ is, roughly speaking, a free $\Bbbk$-module $V$ equipped with a choice of roots in $V$ and coroots in $V^* :=
\Hom_{\Bbbk}(V,\Bbbk)$, making $V$ into a ``reflection representation'' of $W$. See \cite[\S 5.7]{EMTW} for information on realizations. The non-expert reader need not concern
themselves with the definition or properties of realizations in general. In this paper we work with a suitably universal realization $V$ of type $C_2$.

\begin{defn} Let $W$ be the Coxeter group of type $C_2$ with simple reflections $\{s,t\}$. Define $V$ to be the free $\Bbbk$-module with basis $\{\alpha_s, \alpha_t, \varpi\}$, where $\alpha_s$ and $\alpha_t$ are the \emph{simple roots}. Inside $V^*$, define the \emph{simple coroots} $\alpha_s^\vee$ and $\alpha_t^\vee$ so that
\begin{equation} \langle \alpha_s^\vee, \alpha_s \rangle = \langle \alpha_t^\vee, \alpha_t \rangle = 2, \qquad \langle
\alpha_s^\vee, \alpha_t \rangle = -2, \qquad \langle \alpha_t^\vee, \alpha_s \rangle = -1, \end{equation}
and
\begin{equation} \langle \alpha_s^\vee, \varpi \rangle = 1, \qquad \langle \alpha_t^\vee, \varpi \rangle = \kappa. \end{equation}

\begin{remark} \label{rmk:whyV} One technical property one might desire of a realization is \emph{Demazure surjectivity}, which states that the map $\alpha_s^\vee : V \to \Bbbk$
is surjective (and similarly for $\alpha_t^\vee$). This property holds for $V$ since $\alpha_s^\vee$ pairs against $\varpi$ to be $1$. This is the reason to
introduce $\varpi$; if $V$ were spanned by $\alpha_s$ and $\alpha_t$ instead, then Demazure surjectivity would fail over $\Z$ or in characteristic $2$. We need to introduce the scalar $\kappa$ to encode the value of $\langle \alpha_t^\vee, \varpi \rangle$, but $\kappa$ plays absolutely no role in the rest of this paper. Our realization $V$ can be
specialized to any other realization of $C_2$ satisfying Demazure surjectivity (if $\kappa$ is specialized appropriately). We recall the importance of Demazure surjectivity in
Remark \ref{rmk:DemSurj} below. \end{remark}

These pairings of simple coroots with simple roots conform with the standard Cartan matrix for $C_2$, with short root $\alpha_s$ and long root $\alpha_t$. Below $s$ is red and $t$ is blue.
\[
	\dynkin[edge/.style={-},text style/.style={scale=1},*/.style={red},o/.style={blue},edge-length=1cm,root-radius=0.10cm,labels={s, t}]B{*o}
\]
We sometimes refer to the elements of $\{s,t\}$, which often appear as indices in this paper, as \emph{colors}. Below, $s$ will always be red, and $t$ will always be blue.

The group $W$ acts on $V$ by $\Bbbk$-linear transformations, via the formula
\begin{equation} s(v) = v - \langle \alpha_s^\vee, v \rangle  \alpha_s, \qquad t(v) = v - \langle \alpha_t^\vee, v \rangle \alpha_t. \end{equation}\end{defn}


\begin{notation}
Henceforth, let
\begin{equation} \alpha = \alpha_s, \quad \beta = \alpha_t, \quad \gamma = \alpha_s + \alpha_t, \quad \delta = \alpha_t + 2 \alpha_s \end{equation}
	\[ \begin{tikzpicture}[descr/.style={fill=white}, scale = 1]
		\node (origin) at (0,0) {};
		\node (alpha) at (1,0) {$\alpha$};
		\node (beta) at (-1,1) {$\beta$};
		\node (gamma) at (0,1) {$\gamma$};
		\node (delta) at (1,1) {$\delta$};
		\path (origin.center) edge (alpha);
		\path (origin.center) edge (beta);
		\path (origin.center) edge (gamma);
		\path (origin.center) edge (delta);
	\end{tikzpicture} \]
be the four positive roots in $V$. It is worth remembering that $\gamma = t(\alpha_s)$ is $s$-invariant and $\delta = s(\alpha_t)$ is $t$-invariant. These facts will be used tacitly in all sorts of polynomial forcing relations below. We also use
\begin{equation} \chi := -s(\varpi) = \alpha - \varpi. \end{equation}
\end{notation}

\begin{notation} Let $R$ be the polynomial ring of $V$, with degrees doubled so that $\deg V = 2$. For $\xi \in R$ we write $\xi_r$ for the operator of right-multiplication by
$\xi$, and $\xi_l$ for left multiplication, acting on any given $(R,R)$-bimodule. \end{notation}

Typical examples of $(R,R)$-bimodules will be morphism spaces in the Hecke category. For example, the monoidal identity in the Hecke category is denoted $\one$, and its
endomorphism ring is $R$.

\begin{notation} The operator $\alpha_s^\vee$ extends by a twisted Leibniz rule to a map $\pa_s \co R \to R$, and similarly for $\pa_t$, see \cite[\S 4.3, Exercise 5.46]{EMTW} for
details. In this paper, $\pa_s$ is primarily applied to elements of $V \subset R$, where it agrees with $\alpha_s^\vee$. \end{notation}

\subsection{The Hecke category}

We assume the reader is familiar with diagrammatics for monoidal categories, see \cite[\S 7]{EMTW} for an introduction. We give here an ad hoc definition of the diagrammatic Hecke
category, tailored for type $C_2$. In type $C_2$, the original construction was given in \cite{ECathedral}. For a full definition in general, see \cite[\S 10.2]{EMTW}. We do not cite original sources; more detailed background and history of these results can be found in the textbook \cite{EMTW}.

\begin{defn} Let $R$ denote the polynomial ring of $V$ over $\Bbbk$. The \emph{(diagrammatic) Bott-Samelson category} $\HC_{\BS} = \HC_{\BS}(W,S,\Bbbk,V)$ is the (strict) monoidal
category defined as follows. The objects are monoidally generated by objects called $B_s$ and $B_t$, whose identity maps are drawn as red and blue lines, respectively. Objects are
therefore sequences in the set $\{s,t\}$, which we sometimes call \emph{Bott-Samelson objects}. As usual for diagrammatic categories, the morphism space between any two
Bott-Samelson objects is the $\Bbbk$-span of diagrams built from the generators in \eqref{subeq:generators}, modulo certain relations which we discuss below. Diagrams have a
degree and relations are homogeneous, making these morphism spaces into graded $\Bbbk$-modules. \end{defn}

The generating morphisms in $\HC_{\BS}$ are trivalent vertices and univalent vertices (called \emph{dots}) of each color $s$ and $t$, $8$-valent vertices, and polynomials.
\begin{subequations} \label{subeq:generators}
\begin{equation} \startdotred, \quad \startdotblue, \quad \finaldotred, \quad \finaldotblue, \quad \splitred, \quad \splitblue, \quad \mergered, \quad \mergeblue, \end{equation}
\begin{equation} \ig{.5}{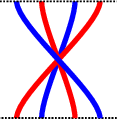}, \quad \ig{.5}{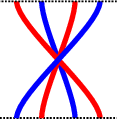}, \quad \bigpoly{f} \text{ for } f \in R. \end{equation}
\end{subequations} 

Because the generating diagrams include polynomials in $R$ acting on the monoidal identity $\one$, morphism spaces are naturally graded $(R,R)$-bimodules, and pre- or
post-composition by any morphism is a bimodule map.

\begin{defn} \label{defn:heckecategory} The \emph{(diagrammatic) Hecke category} $\HC = \HC(W,S,\Bbbk,V)$ is the additive graded Karoubi envelope of $\HC_{\BS}(W,S,\Bbbk,V)$. The objects are formal direct sums of grading shifts of direct summands (i.e. formal images of idempotents) of Bott-Samelson objects, and the morphisms are induced from those in $\HC_{\BS}$. \end{defn}

\begin{notation} We write $\HC^{\Z}$ for the Hecke category when $\Bbbk = \Z$, and $\HC^0$ or $\HC^p$ for the Hecke category when $\Bbbk$ is a field of the corresponding characteristic. \end{notation}

The relations can be found in \cite[\S 10.2.2]{EMTW}, and we discuss them here only at a high level. One family of relations are the \emph{isotopy relations}, stating that isotopic diagrams with the
same boundary represent the same morphism. A second family of relations are the \emph{Frobenius relations}, stating that $B_s$ is a Frobenius algebra object in $\HC_{\BS}$. For example, the trivalent vertex satisfies an associativity relation
\begin{equation} \label{eq:associativity} \ig{1}{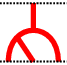} = \ig{1}{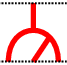}, \end{equation}
and the univalent and trivalent vertices satisfy a unit relation
\begin{equation} \label{eq:unitaxiom} \ig{1}{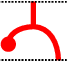} = \ig{1}{linered}. \end{equation}
A third family of relations involve the $8$-valent vertex. In this paper we will entirely avoid the use of the $8$-valent vertex and its relations, by way of the thick calculus developed in \cite{ECathedral}, which we recall in the next section. A final relation is the \emph{polynomial forcing relation}
\begin{equation} \label{eq:polyforce} \bigpoly{f} \ig{1}{linered} = \ig{1}{linered} \bigpoly{s(f)} + \ig{1}{brokenred} \bigpoly{\pa_s(f)} \quad. \end{equation}

Suppose $f \in R$. Using \eqref{eq:polyforce}, one can show that $f_l - f_r$ acts by zero on the identity map of $B_s$ when $f$ is $s$-invariant. In particular, it acts by zero on any morphism factoring through the object $B_s$ (since pre- and post-composition are $(R,R)$-bimodule maps). If $f$ is $W$-invariant, then $f_l - f_r$ acts by zero on the identity map of any Bott-Samelson object, and therefore on any morphism in $\HC$. We record these statements for future reference.

\begin{equation} \label{sinvariantkillsBs}\text{For } f \in R^s, f_l - f_r \text{ acts by zero on any morphism factoring through } B_s. \end{equation}

\begin{equation} \label{Winvariantkillsall} \text{For } f \in R^W, f_l - f_r \text{ acts by zero on any morphism in } \HC. \end{equation}

Now we recall the key properties of $\HC$, all of which are discussed at length in \cite[\S 10 and \S 11]{EMTW}. Over any base ring $\Bbbk$, morphism spaces are free as left
$R$-modules and as right $R$-modules. The size of morphism spaces is controlled by the Soergel Hom Formula, see \cite[Theorem 5.27]{EMTW}, and an explicit \emph{double leaves}
basis for morphisms is defined in \cite[\S 10]{EMTW}. When $\Bbbk$ is a field, the category $\HC$ is Karoubian and satisfies the Krull-Schmidt property, so its Grothendieck group
$[\HC]$ has basis given by the symbols of indecomposable objects. One can prove that the $\Z[v,v^{-1}]$-algebra map sending $b_s \mapsto [B_s]$ and $v \mapsto [\one(1)]$ is an
isomorphism from the Hecke algebra $\HB(W)$ to $[\HC]$, for any field $\Bbbk$.

When $\Bbbk$ is a field, the indecomposable objects in $\HC$ are parametrized, up to isomorphism and grading shift, by the elements of $W$. More precisely, the Bott-Samelson object
associated to a reduced expression of $w \in W$ will have a unique direct summand which can not be found as a summand of any shorter expression, and this summand is denoted ${}^p
B_w$. When the characteristic is understood, this is shortened to $\pB_w$, and in characteristic zero we may also use $B_w$. The object $\pB_w$ is independent of the choice of reduced expression, up
to (non-canonical) isomorphism. However, the behavior of this indecomposable summand is very dependent on the choice of $\Bbbk$! For example, in characteristic $2$, $B_s B_t B_s$
is indecomposable, so this Bott-Samelson object is $\pB_{sts}$. In any other characteristic, $B_s$ is a direct summand of $B_s B_t B_s$, and $\pB_{sts}$ is the complementary
summand.

\begin{defn} \label{defn:pcanbasis} Let $\Bbbk$ be a field of characteristic $p$. The \emph{$p$-canonical basis} $\{\pb_w\}$ is the basis of $\HB(W)$ corresponding to $\{[\pB_w]\}$
under the aforementioned isomorphism $\HB(w) \to [\HC]$. \end{defn}

It was proven in \cite{JenWilpCan} that this basis does not depend on the precise choice of field, only on the characteristic of that field. Many other basic properties of the $p$-canonical basis can be found in \cite{JenWilpCan}.

%

When we work over $\Bbbk = \Z$, the Krull-Schmidt property will not hold. It still makes sense to discuss indecomposable objects in the Hecke category $\HC^{\Z}$, but their
endomorphism rings will not be local. Statements about the classification of indecomposable objects in $\HC^{\Z}$ and the uniqueness of direct sum decompositions are not in the
literature to our knowledge. Regardless, type $C_2$ is sufficiently small that indecomposable objects can be constructed by brute force. For each indecomposable object which is
not already a Bott-Samelson object, we will explicitly write down the idempotent (inside the endomorphism ring of some Bott-Samelson object) which that indecomposable is the image
of. We will also explicitly write down inclusion and projection maps for all direct sum decompositions we use.

\begin{notation} Recall that over $\Z$ the Bott-Samelson object $B_s B_t B_s$ is indecomposable, and this stays true in characteristic $2$, so that $B_s B_t B_s = \pB_{sts}$ for
$p=2$. For clarity, we will never use the notation $\pB_{sts}$ below, always refering to this object as $B_s B_t B_s$. For any base ring where $2$ is invertible, $B_s B_t B_s \cong
B_{sts} \oplus B_s$ where $B_{sts}$ is the image of the idempotent $\frac{e_{sts}}{2}$ defined in the next section. We only use the notation $B_{sts}$ when $2$ is invertible.
\end{notation}

Finally, we note the relation between the diagrammatic category $\HC$ and the category of Soergel bimodules. There is a monoidal functor $\FC$ sending $\HC$ to the
category of graded $(R,R)$-bimodules, for which $\FC(B_s) = R \ot_{R^s} R(1)$. The grading shift places $1 \ot 1$ in degree $-1$. The indecomposable object $\pB_{w_0}$ has the
same size as $B_{w_0}$ in any characteristic (i.e. $\pb_{w_0} = b_{w_0}$), and is sent by $\FC$ to the bimodule $R \ot_{R^W} R(4)$. For the realization we are using, $\FC$ is fully faithful, and induces an $(R,R)$-bimodule isomorphism on morphism spaces. The essential image of $\FC$ is known as the category of Soergel bimodules.

\begin{rmk} In type $C_2$ the functor $\FC$ is fully faithful, so diagrams can be viewed as a computational tool to study Soergel bimodules. In affine type in finite
characteristic, it is extremely important to distinguish between the diagrammatic category $\HC$ and the category of Soergel bimodules. This is because the reflection
representation of an infinite Coxeter group can not be faithful in finite characteristic, leading to additional morphisms between Soergel bimodules which do not exist in $\HC$. In
other words, the functor $\FC$ will no longer be full. \end{rmk}

\subsection{Basics of thick calculus}

Because $\pa_t(\al_s) = -1$ is invertible, $B_t$ is a direct summand of $B_t B_s B_t$. The complementary summand $B_{tst}$ is the image of the idempotent $e_{tst}$.
\begin{equation} \label{eq:3} e_{tst} = \ethree := \longblue \longred \longblue + \doublebluepitch \end{equation}
By design $e_{tst}$ satisfies \emph{death by pitchfork}:
\begin{equation} \label{3pitch} {
\labellist
\tiny\hair 2pt
 \pinlabel {$3$} [ ] at 20 20
\endlabellist
\centering
\ig{.5}{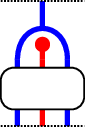}
} = 0. \end{equation}
We also have the \emph{leapfrog move}
\begin{equation} \label{3leapfrog} {
\labellist
\tiny\hair 2pt
 \pinlabel {$3$} [ ] at 20 20
\endlabellist
\centering
\ig{.5}{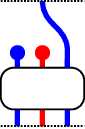}
} = {
\labellist
\tiny\hair 2pt
 \pinlabel {$3$} [ ] at 20 20
\endlabellist
\centering
\ig{.5}{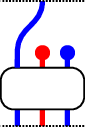}
}. \end{equation}
The easiest way to prove \eqref{3leapfrog} is to combine death by pitchfork with the \emph{three-way dot force}
\begin{equation} \label{movingdot} \ig{.5}{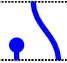} = \ig{.5}{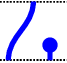} + \ig{.5}{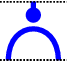} - \mergeblue \barbblue. \end{equation}

Meanwhile, $B_s B_t B_s$ is equipped with a quasi-idempotent $e_{sts}$
\begin{equation} e_{sts} := 2 \longred \longblue \longred + \doubleredpitch. \end{equation}
This endomorphism satisfies death by pitchfork, and consequently $e_{sts}^2 = 2 e_{sts}$. After inverting the number $2$ we can define the true idempotent $\frac{e_{sts}}{2}$, which projects to a summand named $B_{sts}$. We denote this idempotent by
\begin{equation} {
\labellist
\tiny\hair 2pt
 \pinlabel {$3'$} [ ] at 19 23
\endlabellist
\centering
\ig{.5}{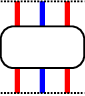}
} := \frac{e_{sts}}{2}. \end{equation}

The \emph{Jones-Wenzl morphism} is the following degree $+2$ rotation-invariant sum of diagrams on the planar disk:
\begin{equation} \label{JW} {
\labellist
\small\hair 2pt
 \pinlabel {$\JW$} [ ] at 22 24
\endlabellist
\centering
\ig{1}{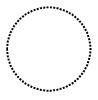}
} := \ig{1}{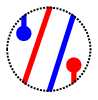} + \ig{1}{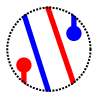} + \ig{1}{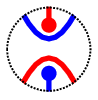} + \ig{1}{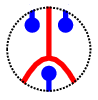} + 2 \ig{1}{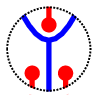}. \end{equation}
It satisfies death by pitchfork from all angles, and is uniquely specified by this property up to scalar. Note the coefficient $2$ on the final diagram.

The Jones-Wenzl morphism gives rise to an idempotent inside both $B_s B_t B_s B_t$ and $B_t B_s B_t B_s$,
\begin{equation} \label{JWidemp} {
\labellist
\small\hair 2pt
 \pinlabel {$\JW$} [ ] at 35 30
\endlabellist
\centering
\ig{.7}{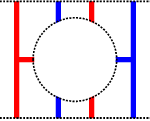}
} \qquad {
\labellist
\small\hair 2pt
 \pinlabel {$\JW$} [ ] at 35 30
\endlabellist
\centering
\ig{.7}{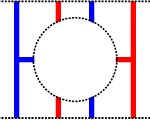}
}, \end{equation} and the images of these two idempotents are isomorphic, see Remark \ref{8valent}. We introduce an abstract object $B_{w_0}$ isomorphic to these images, which we denote with a thicker purple strand, following the thick calculus of \cite[\S 6.3]{ECathedral}. When interacting with polynomials, $B_{w_0}$ behaves like its Soergel bimodule counterpart $R \ot_{R^W} R(4)$.  There are inclusion and projection maps, denoted as vertices which ``split'' the purple strand into red and blue strands.
\begin{equation} \ig{.5}{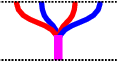}, \qquad \ig{.5}{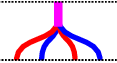}, \qquad \ig{.5}{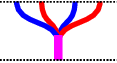}, \qquad \ig{.5}{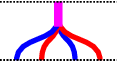}. \end{equation}
Then, tautologically, we have the following relations and their color swap:
\begin{equation} \label{splitmerge} \ig{.5}{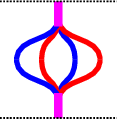} = \ig{.5}{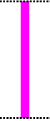}, \qquad \ig{.5}{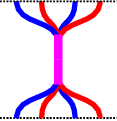} = {
\labellist
\small\hair 2pt
 \pinlabel {$\JW$} [ ] at 35 30
\endlabellist
\centering
\ig{.5}{idemptsts}
}. \end{equation}

\begin{rmk} \label{8valent} The composition $B_s B_t B_s B_t \to B_{w_0} \to B_t B_s B_t B_s$ is the $8$-valent vertex in the ordinary Soergel calculus. The $8$-valent vertex induces an isomorphism between the images of the two Jones-Wenzl idempotents.
\begin{equation} \label{8valentcomposition} \ig{.7}{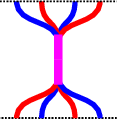} = \ig{.7}{8valentversion1}. \end{equation}
Because of \eqref{8valentcomposition}, we can entirely avoid using $8$-valent vertices, and for the rest of this paper we only use diagrams with splitting vertices instead. \end{rmk}

Using \eqref{splitmerge}, we deduce that the splitters satisfy death by pitchfork. As a consequence, they absorb $e_{tst}$:
\begin{equation} \label{absorb} {
\labellist
\tiny\hair 2pt
 \pinlabel {$3$} [ ] at 32 8
\endlabellist
\centering
\ig{1}{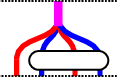}
} = \ig{1}{ststtow}. \end{equation}
By the uniqueness of the Jones-Wenzl morphism, it is straightforward to deduce
\begin{equation} \label{doton8valent} \ig{.7}{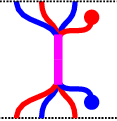} = {
\labellist
\small\hair 2pt
 \pinlabel {$\JW$} [ ] at 28 28
\endlabellist
\centering
\ig{.7}{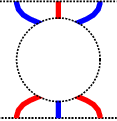}
}, \qquad \ig{.7}{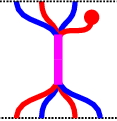} = {
\labellist
\small\hair 2pt
 \pinlabel {$\JW$} [ ] at 28 28
\endlabellist
\centering
\ig{.7}{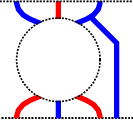}
}. \end{equation}

Two other relations allow one to move dots around on splitters. These relations also hold upside-down, or with the colors swapped. The \emph{leapfrog move} can be proven similarly to \eqref{3leapfrog} above.
\begin{equation} \label{leapfrog} \ig{.5}{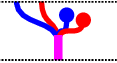} = \ig{.5}{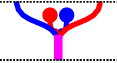} = \ig{.5}{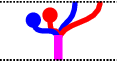}. \end{equation}
The \emph{dot migration} move can be proven with \eqref{splitmerge} and \eqref{doton8valent}.
\begin{equation} \label{dotmigrate} \ig{.5}{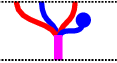} = \ig{.5}{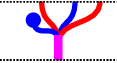}. \end{equation}

In particular, there is a one-dimensional space of degree $+1$ morphisms from $B_{w_0}$ to $B_s B_t B_s$ (or to $B_t B_s B_t$), spanned by the morphism in \eqref{dotmigrate}. We often shorten this morphism to a \emph{partial splitter} as follows:
\begin{equation} \label{partialsplitterdef} \wtosts := \ig{.5}{wtoDOTsts} = \ig{.5}{wtostsDOT}. \end{equation}
While the splitter has degree $0$, this partial splitter has degree $+1$. Note that these partial splitters satisfy death by pitchfork, and thus the partial splitter to $B_t B_s B_t$ absorbs $e_{tst}$.

Similarly, there is a one-dimensional space of degree $4 - \ell(u)$ morphisms from $B_{w_0}$ to $B_u$ for any $u \in W$, allowing us to define other partial splitters:
\begin{equation} \label{morepartialsplitters} \wtots := \ig{.5}{wtotsDOTDOT}, \qquad \wtot := \ig{.5}{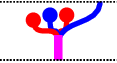}. \end{equation}
All partial splitters are just the splitter composed with the appropriate number of dots, whose placement is irrelevant. The partial splitters in \eqref{morepartialsplitters} have degree $2$ and $3$ respectively. Our conventions for partial splitters are invariant under flipping diagrams horizontally or vertically, and under color swap. The ``partial splitter'' when $u = 1$ is usually just denoted as a purple dot, c.f. \cite[(6.20)]{ECathedral}:
\begin{equation} \label{purple dot} \wtoid. \end{equation}

By placing dots on the various relations above, one can compute a number of useful relations. For clarity we place all polynomials on the right, and use the subscript $r$ to indicate right multiplication.
\begin{subequations} \label{useful}
\begin{equation} \ig{.5}{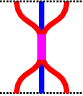} = \linered\lineblue\linered \beta_r + \doubleredpitch \delta_r + \finaldotred\finaldotblue\redpitchforkup + \startdotred\startdotblue\redpitchforkdown + 2 \linered \lineblue \brokenred - 2 \ig{.5}{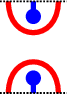}, \end{equation}
\begin{equation} {
\labellist
\tiny\hair 2pt
 \pinlabel {$3$} [ ] at 20 88
 \pinlabel {$3$} [ ] at 21 18
\endlabellist
\centering
\ig{.5}{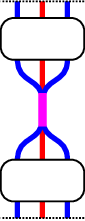}
} = \ethree \alpha_r + \threetsthree. \end{equation}
\begin{equation} {
\labellist
\tiny\hair 2pt
 \pinlabel {$3$} [ ] at 20 18
\endlabellist
\centering
\ig{.5}{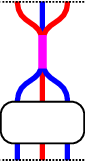}
} = \threetotstosts + \threetosttosts + \threetostosts. \end{equation}
\end{subequations}

\subsection{Direct sum decompositions}

Whenever we use the direct sum decomposition
\begin{subequations} \label{BtBtip}
\begin{equation} B_t B_t \cong B_t(1) \oplus B_t(-1), \end{equation}
we use the following inclusion and projection maps. The subscript indicates which summand is the source or target, i.e. $p_{-1}$ is a map to $B_t(-1)$ so it has degree $-1$, while $i_{-1}$ is a map from $B_t(-1)$ so it has degree $+1$.
\begin{equation} p_{-1} = \ig{.7}{mergeblue}, \qquad i_{-1} = {
\labellist
\tiny\hair 2pt
 \pinlabel {$-\alpha$} [ ] at 15 24
\endlabellist
\centering
\ig{.7}{splitblue}
}. \end{equation}
\begin{equation} p_{+1} = {
\labellist
\tiny\hair 2pt
 \pinlabel {$\gamma$} [ ] at 15 8
\endlabellist
\centering
\ig{.7}{mergeblue}
}, \qquad i_{+1} = \ig{.7}{splitblue}. \end{equation}
\end{subequations}
Crucial to this direct sum decomposition is the fact that $\pa_t(-\alpha) = 1$, and $t(\alpha) = \gamma$.

Whenever we use the direct sum decomposition
\begin{subequations} \label{BsBsip}
\begin{equation} B_s B_s \cong B_s(1) \oplus B_s(-1), \end{equation}
we use the following inclusion and projection maps.
\begin{equation} p_{-1} = \ig{.7}{mergered}, \qquad i_{-1} = {
\labellist
\tiny\hair 2pt
 \pinlabel {$\varpi$} [ ] at 15 24
\endlabellist
\centering
\ig{.7}{splitred}
}. \end{equation}
\begin{equation} p_{+1} = {
\labellist
\tiny\hair 2pt
 \pinlabel {$\chi$} [ ] at 15 8
\endlabellist
\centering
\ig{.7}{mergered}
}, \qquad i_{+1} = \ig{.7}{splitred}. \end{equation}
\end{subequations}
Crucial to this direct sum decomposition is the fact that $\pa_s(\varpi) = 1$ and $-s(\varpi) = \chi$.

\begin{remark} \label{rmk:DemSurj} The reason to include $\varpi$ in the realization $V$, and more generally the reason for Demazure surjectivity, is to ensure a decomposition as in \eqref{BsBsip}. \end{remark}

Whenever we use the direct sum decomposition
\begin{subequations} \label{wsip}
\begin{equation} B_{w_0} B_s \cong B_{w_0}(1) \oplus B_{w_0}(-1), \end{equation}
we use projection and inclusion maps similar to \eqref{BsBsip}, except that every appearance of $\splitred$ is replaced with
\begin{equation} \label{thicktri} \ig{.7}{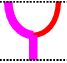}, \end{equation}
and similarly upside-down. The same polynomials appear in the same places in $p_{+1}$ and $i_{-1}$. The \emph{thick trivalent vertex} \eqref{thicktri} is defined as in \cite[(6.18)]{ECathedral}. It may help to note that
\begin{equation} \ig{.7}{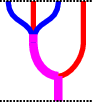} = \ig{.7}{wtotsts}, \end{equation}
a variant on the unit axiom (it follows easily from \eqref{splitmerge} and one-color relations).
\end{subequations}

Similarly, for the direct sum decomposition
\begin{subequations} \label{wtip}
\begin{equation} B_{w_0} B_t \cong B_{w_0}(1) \oplus B_{w_0}(-1), \end{equation}
we use projection maps similar to \eqref{BtBtip} except with a blue thick trivalent vertex replacing $\splitblue$. We do the same for
\begin{equation} B_{tst} B_t \cong B_{tst}(1) \oplus B_{tst}(-1), \end{equation}
except using a different ``trivalent'' vertex
\begin{equation} {
\labellist
\tiny\hair 2pt
 \pinlabel {$3$} [ ] at 20 25
\endlabellist
\centering
\ig{.7}{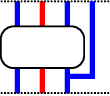}
} = {
\labellist
\tiny\hair 2pt
 \pinlabel {$3$} [ ] at 20 25
\endlabellist
\centering
\ig{.7}{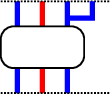}
}. \end{equation}
\end{subequations}

One has a direct sum decomposition
\begin{subequations} \label{BtstBsip}
\begin{equation} B_{tst} B_s \cong B_{w_0} \oplus B_{ts}. \end{equation}
We use inclusion and projection maps
\begin{equation} p_{w_0} = {
\labellist
\tiny\hair 2pt
 \pinlabel {$3$} [ ] at 21 9
\endlabellist
\centering
\ig{.7}{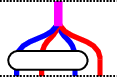}
}, \qquad i_{w_0} = {
\labellist
\tiny\hair 2pt
 \pinlabel {$3$} [ ] at 21 30
\endlabellist
\centering
\ig{.7}{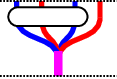}
}, \end{equation}
\begin{equation} p_{ts} = {
\labellist
\tiny\hair 2pt
 \pinlabel {$3$} [ ] at 21 9
\endlabellist
\centering
\ig{.7}{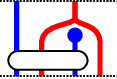}
}, \qquad i_{ts} = -{
\labellist
\tiny\hair 2pt
 \pinlabel {$3$} [ ] at 21 30
\endlabellist
\centering
\ig{.7}{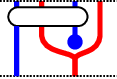}
}. \end{equation}
It is not hard to verify that
\begin{equation} p_{ts} \circ i_{ts} = -(\pa_s(\beta) + 1) \id_{ts} = \id_{ts}, \end{equation}
\begin{equation} p_{w_0} \circ i_{w_0} = \id_{w_0}, \end{equation}
\begin{equation} \label{fooblah} i_{w_0} \circ p_{w_0} + i_{ts} \circ p_{ts} = \id_{tst} \ot \id_{s}. \end{equation}
For \eqref{fooblah}, when resolving $i_{w_0} p_{w_0}$ with \eqref{splitmerge} and \eqref{JW}, death by pitchfork eliminates all but two terms: the identity, and a term cancelling $i_{ts} p_{ts}$.
\end{subequations}
Note that this direct sum decomposition works over $\Z$.

Finally we consider the decomposition of the Bott-Samelson object $(B_s B_t B_s) B_t$. There is a two-dimensional space of degree zero morphisms to $B_{st}$, spanned by
\begin{subequations} \label{CstsBtip}
\begin{equation} \linered \bluepitchforkdown, \qquad \redpitchforkdown \lineblue \end{equation}
The two-dimensional space of morphisms from $B_{st}$ is spanned by the vertical flip of these maps. Composing these to get a $2 \times 2$ matrix of endomorphisms of $B_{st}$ (the so-called \emph{local intersection form}), the resulting morphisms are all multiples of the identity, and the coefficients are
\begin{equation} \left( \begin{array}{cc} -1 & 1 \\ 1 & -2 \end{array} \right). \end{equation}
Since this matrix has determinant $1$, $B_{st}$ appears as a summand inside $B_s B_t B_s B_t$ with multiplicity $2$ (even over $\Z$). Specifically, we chould choose projection and inclusion maps as follows.
\begin{equation} p_1 = \linered \bluepitchforkdown + \redpitchforkdown \lineblue, \qquad i_1 = -\redpitchforkup \lineblue , \end{equation}
\begin{equation} p_2 = 2 \linered \bluepitchforkdown + \redpitchforkdown \lineblue, \qquad i_2 = -\linered \bluepitchforkup. \end{equation}
It is easy to verify that $\id_{B_s B_t B_s B_t} - i_1 p_1 - i_2 p_2$ is the Jones-Wenzl idempotent from \eqref{JWidemp}. Thus we deduce that
\begin{equation} B_s B_t B_s B_t \cong B_{w_0} \oplus B_{st} \oplus B_{st} \end{equation}
where the final projection and inclusion maps are given by splitters
\begin{equation} p_{w_0} = \ig{.5}{wtostst}, \qquad i_{w_0} = \ig{.5}{ststtow}. \end{equation}
\end{subequations}

\subsection{Endomorphisms of $B_{w_0}$}

As an $R$-bimodule, $\End(B_{w_0}) \cong R \ot_{R^W} R$. This isomorphism equips $\End(B_{w_0})$ with a presentation, where it is generated by left and right multiplication by
polynomials, and the relations come from the tensor product. Note that $\End(B_{w_0})$ is a commutative ring.

We will return to left and right multiplication by polynomials in \S\ref{subsec:koszul}. First, we explore a different description of $\End(B_{w_0})$.

As a right $R$-module, $\End(B_{w_0})$ is free of graded rank $1 + 2v^2 + 2v^4 + 2v^6 + v^8$. It has a basis $\{\phi_u\}$ for $u \in W$, where $\phi_u$ is defined as a composition
of partial splitters $B_{w_0} \to B_u \to B_{w_0}$, and has degree $8 - 2 \ell(u)$. The identity map is $\phi_{w_0}$.

\begin{rmk} All statements about the graded rank of morphism spaces follow from Soergel's Hom formula, and its analogue for the diagrammatic Hecke category (see \cite[Theorem 11.1(5)]{EMTW}), which still applies even when working over $\Z$. \end{rmk}


Let
\begin{equation} x := \phi_{tst} = \ig{.5}{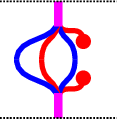}, \qquad y := \phi_{sts} = \ig{.5}{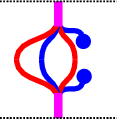}. \end{equation}
Then $\End(B_{w_0})$ is generated over the right action of $R$ by $x$ and $y$. The action of $x$ and $y$ on the right $R$-basis $\{\phi_u\}$ can be computed fairly easily using \eqref{useful}. We encode it in the following graph, where a blue (resp. red) arrow indicates the action of $x$ (resp. $y$).
\begin{equation} \label{endograph}
\begin{tikzpicture}
\node (w0) at (0,8) {$\phi_{w_0}$};
\node (tst) at (-2,6) {$\phi_{tst}$};
\node (sts) at (2,6) {$\phi_{sts}$};
\node (ts) at (-2,4) {$\phi_{ts}$};
\node (st) at (2,4) {$\phi_{st}$};
\node (t) at (-2,2) {$\phi_t$};
\node (s) at (2,2) {$\phi_s$};
\node (1) at (0,0) {$\phi_1$};
\path[blue, ->]
	(w0) edge (tst)
	(tst.255) edge (ts.105)
	(sts.255) edge (st.105)
	(tst) edge [loop left] node {$\alpha$} (tst)
	(sts) edge (ts)
	(st) edge [loop left] node {$\alpha$} (st)
	(ts) edge [loop left] node {$\gamma$} (ts)
	(ts.255) edge (t.105)
	(st.255) edge (s.105)
	(st) edge (t)
	(s) edge [loop left] node {$\gamma$} (s)
	(t) edge [loop left] node {$\delta$} (t)
	(s) edge (1)
	(1) edge [loop left] node {$\delta$} (1);
\path[red, ->]
	(w0) edge (sts)
	(tst.285) edge (ts.75)
	(sts.285) edge node[right] {$2$} (st.75)
	(sts) edge [loop right] node {$\beta$} (sts)
	(tst) edge (st)
	(st) edge [loop right] node {$\delta$} (st)
	(ts) edge [loop right] node {$\beta$} (ts)
	(ts.285) edge node[right] {$2$} (t.75)
	(st.285) edge (s.75)
	(ts) edge (s)
	(s) edge [loop right] node {$2 \gamma$} (s)
	(t) edge [loop right] node {$\delta$} (t)
	(t) edge (1)
	(1) edge [loop right] node {$2 \gamma$} (1);
\end{tikzpicture} \end{equation}
For example, $\phi_{ts} x = \phi_t + \gamma_r \phi_{ts}$ and $\phi_{ts} y = \phi_s + 2 \phi_t + \beta_r \phi_{ts}$. Remember that $\xi_r$ denotes right multiplication by the polynomial $\xi \in R$.
As we will see in the next section, this graph is also useful for studying morphisms from $B_{w_0}$ to other objects.

\subsection{Morphisms involving $B_{w_0}$}

Now let $z \in W$ be any element except for $sts$; one can also allow $z = sts$ if $2$ is inverted. The indecomposable object $B_z$ categorifies the Kazhdan-Lusztig basis element, and this basis element is smooth, meaning that all of its Kazhdan-Lusztig polynomials are trivial. As a consequence of the Soergel Hom Formula, for any $u \le z$ there is a one-dimensional space of
morphisms from $B_z$ to $B_u$ in degree $\ell(z) - \ell(u)$, which can be obtained by including $B_z$ into a Bott-Samelson, applying the appropriate number of dots (whose placement
is irrelevant), and then projecting to $B_u$. This map is a generalization of the partial splitters above. We might call it a \emph{canonical map}.

\begin{rmk} When there is a functor from diagrammatics to bimodules, the inclusion and projection maps of top summands are normalized so that they preserve certain vectors known as 1-tensors. The 1-tensors are also preserved by $\finaldotred$, so they are preserved by canonical maps. \end{rmk}

For any $z \in W$ and $u \le z$ we can define $\phi_u^z \co B_{w_0} \to B_z$, defined as the composition $B_{w_0} \to B_u \to B_z$ of canonical maps. It has degree $4 + \ell(z) - 2
\ell(u)$. Then $\{ \phi_u^z\}_{u \le z}$ is a right $R$-basis for $\Hom(B_{w_0},B_z)$. There is an action of $\End(B_{w_0})$ on
$\Hom(B_{w_0}, B_z)$ by precomposition, which one might encode with a graph as in \eqref{endograph}. In fact, this graph embeds inside \eqref{endograph}, sending $\phi_u^z \mapsto
\phi_u$ (the proof by direct computation is straightforward). For example, \[ \phi^{tst}_{ts} \circ x = \phi^{tst}_t + \gamma_r \phi^{tst}_{ts}.\]

Similarly, we can define $\barphi_u^z \co B_z \to B_{w_0}$, as the composition $B_z \to B_u \to B_{w_0}$. In other words, $\barphi_u^z$ is the vertical flip of the diagram for
$\phi_u^z$. The action of $\End(B_{w_0})$ by postcomposition on $\Hom(B_z, B_{w_0})$ obeys exactly the same rules as \eqref{endograph} again, since the generators $x$ and $y$ of
$\End(B_{w_0})$ are invariant under the vertical flip.


It remains to examine $\Hom(B_{w_0}, B_s B_t B_s)$. By flipping everything upside-down, we obtain similar statements for $\Hom(B_s B_t B_s, B_{w_0})$. A basis for morphisms is
given by the following eight elements, sorted by degree.
\begin{subequations}
\begin{equation} \phi_{sts}^{sts} := \wtosts, \quad \phi_{st}^{sts} := \wtost \startdotred, \quad \phi_{ts}^{sts} := \startdotred \wtots, \quad \rho_s^{sts} = \ig{.5}{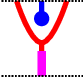}. \end{equation}
\begin{equation} \phi_s^{sts}:= \wtos \startdotblue \startdotred, \quad \phi_t^{sts} = \startdotred \wtot \startdotred, \quad \rho_1^{sts} = \wtoid \pitchcupred, \quad \phi_1^{sts} = \wtoid \startdotred \startdotblue\startdotred. \end{equation}
\end{subequations}
The action of $x$ and $y$ on this basis is given as follows. The action on $\rho^{sts}_s$ and $\rho^{sts}_1$ mimic the action on $\phi_s$ and $\phi_1$ from \eqref{endograph}. In the quotient by the submodule spanned by $\rho^{sts}_s$ and $\rho^{sts}_1$, the action on $\phi_u^{sts}$ embeds inside \eqref{endograph}. Here is the full graph.
\begin{equation} \label{endograph2}
\begin{tikzpicture}
\node (sts) at (2,6) {$\phi^{sts}_{sts}$};
\node (ts) at (-2,4) {$\phi^{sts}_{ts}$};
\node (st) at (2,4) {$\phi^{sts}_{st}$};
\node (t) at (-2,2) {$\phi^{sts}_t$};
\node (s) at (2,2) {$\phi^{sts}_s$};
\node (1) at (0,0) {$\phi^{sts}_1$};
\node (rhos) at (-6,4) {$\rho^{sts}_s$};
\node (rho1) at (-6,2) {$\rho^{sts}_1$};
\path[blue, ->]
	(sts.255) edge (st.105)
	(sts) edge (ts)
	(st) edge [loop left] node {$\alpha$} (st)
	(ts) edge [loop left] node {$\gamma$} (ts)
	(ts.255) edge (t.105)
	(st.255) edge (s.105)
	(st) edge (t)
	(s) edge [loop left] node {$\gamma$} (s)
	(t) edge [loop left] node {$\delta$} (t)
	(s) edge (1)
	(1) edge [loop left] node {$\delta$} (1)
	(sts.180) edge (rhos.90)
	(rhos) edge (rho1)
	(rhos) edge [loop left] node {$\gamma$} (rhos)
	(rho1) edge [loop left] node {$\delta$} (rho1);
\path[red, ->]
	(sts.285) edge node[right] {$2$} (st.75)
	(sts) edge [loop right] node {$\beta$} (sts)
	(st) edge [loop right] node {$\delta$} (st)
	(ts) edge [loop right] node {$\beta$} (ts)
	(ts.285) edge node[right] {$2$} (t.75)
	(st.285) edge (s.75)
	(ts) edge (s)
	(s) edge [loop right] node {$2 \gamma$} (s)
	(t) edge [loop right] node {$\delta$} (t)
	(t) edge (1)
	(1) edge [loop right] node {$2 \gamma$} (1)
	(rhos) edge [loop right] node {$2 \gamma$} (rhos)
	(rho1) edge [loop right] node {$2 \gamma$} (rho1)
	(sts.195) edge (rhos.30)
	(ts) edge (rho1)
	(ts) edge[bend left] node[above] {$-\alpha$} (rhos);
\end{tikzpicture} \end{equation}



\subsection{Left multiplication versus right multiplication} \label{subsec:koszul}

We return to the study of $\End(B_{w_0})$. We have the following formulas which compute left multiplication by polynomials in terms of this right $R$-basis.
\begin{subequations} \label{leftmult}
\begin{equation} \alpha_l = 2x - y - \alpha_r. \end{equation}
\begin{equation} \beta_l = 2y - 2x - \beta_r. \end{equation}
\begin{equation} \gamma_l = y - \gamma_r. \end{equation}
\begin{equation} \delta_l = 2x - \delta_r. \end{equation}
\end{subequations}
In particular, there are nice formulas for the sum $f_l + f_r$ for various roots $f$.
\begin{subequations} \label{sums}
\begin{equation} \alpha_l + \alpha_r = 2x - y. \end{equation}
\begin{equation} \beta_l  + \beta_r = 2y - 2x. \end{equation}
\begin{equation} \gamma_l + \gamma_r = y. \end{equation}
\begin{equation} \delta_l + \delta_r = 2x. \end{equation}
\end{subequations}
Some sums are divisible by $2$, so that when we write $\frac{\beta_l + \beta_r}{2}$ we refer to the morphism $y-x$, a genuine morphism over $\Z$. One also obtains formulas for differences $f_l - f_r$, which appear frequently in differentials in Rouquier complexes. For example
\begin{equation} \frac{\delta_l - \delta_r}{2} = x - \delta_r, \qquad \gamma_l - \gamma_r = y - 2 \gamma_r. \end{equation}

The difference $\gamma_l - \gamma_r$ is an endomorphism of $B_{w_0}$ which will annihilate any morphism $B_{w_0} \to B_s$ under precomposition, or any morphism $B_s
\to B_{w_0}$ under postcomposition. This is because $\gamma$ is $s$-invariant, so $\gamma_l - \gamma_r$ acts by zero on the identity map of $B_s$ by \eqref{eq:polyforce}. Since pre-composition is an $(R,R)$-bimodule map, $\gamma_l - \gamma_r$ kills any morphism out of $B_s$. It should not be surprising then that $\gamma_l - \gamma_r = y - 2 \gamma_r$ kills $\phi_s$, as is evident from \eqref{endograph}.
Similarly, $\delta_l - \delta_r$ and $x - \delta_r$ annihilate any morphism between $B_{w_0}$ and $B_t$. For any polynomial $f$, $f_l - f_r$ annihilates any morphism between $B_{w_0}$ and $R$.

Let $q = 2 \alpha^2 + 2 \alpha \beta + \beta^2$. This is the \emph{quadratic form} induced by the Cartan matrix, and it is $W$-invariant. Therefore $q_l - q_r = 0$ acting on any morphism space in the Hecke category (since the bimodule action of $R$ factors through $R \ot_{R^W} R$).

Recall that $\gamma = \beta + \alpha$ and $\delta = \beta + 2\alpha$. This implies that
\begin{equation} \label{eq:qis} q = \alpha \delta + \beta \gamma. \end{equation}
Now consider the action of the following operator on an $(R,R)$-bimodule.
\begin{equation} (\delta_l - \delta_r)(\alpha_l + \alpha_r) + (\gamma_l - \gamma_r)(\beta_l + \beta_r) = (\alpha_l \delta_l + \gamma_l \beta_l) - (\alpha_r \delta_r + \gamma_r \beta_r) + (\delta_l \alpha_r - \delta_r \alpha_l + \gamma_l \beta_r - \gamma_r \beta_l). \end{equation}
The first term involves all left actions, and equals $q_l$ by \eqref{eq:qis}. The second term equals $-q_r$. The third term mixes the left and right action. However, expanding $\gamma$ and $\delta$ into $\alpha$ and $\beta$, the third term equals zero:
\begin{equation} \delta_l \alpha_r - \delta_r \alpha_l + \gamma_l \beta_r - \gamma_r \beta_l = 2 \alpha_l \alpha_r + \beta_l \alpha_r - 2 \alpha_r \alpha_l - \beta_r \alpha_l + \beta_l \beta_r + \alpha_l \beta_r - \beta_r \beta_l - \alpha_r \beta_l = 0. \end{equation}
Thus we conclude that
\begin{equation} (\delta_l - \delta_r)(\alpha_l + \alpha_r) + (\gamma_l - \gamma_r)(\beta_l + \beta_r) = q_l - q_r, \end{equation}
which acts by zero on any morphism space in the Hecke category.

This permits us to build the 2-periodic ``Koszul complex''
\begin{equation} \label{Koszul} \ldots \to B_{w_0}^{\oplus 2} \to B_{w_0}^{\oplus 2} \to B_{w_0} \to R \to 0. \end{equation}
The final differential in this complex is $\wtoid$. The penultimate differential is
\[ K_0 = \left[ \begin{array}{cc} \frac{\delta_l - \delta_r}{2} & \gamma_l - \gamma_r \end{array} \right], \]
which annihilates any map to $R$. The differential before that is
\[ K_1 = \left[ \begin{array}{cc} \gamma_l - \gamma_r & \alpha_l + \alpha_r \\ - \frac{\delta_l - \delta_r}{2} & \frac{\beta_l + \beta_r}{2} \end{array} \right] \]
and the differential before that is
\[ K_2 = \left[ \begin{array}{cc}  \frac{\beta_l + \beta_r}{2} & -(\alpha_l + \alpha_r) \\  \frac{\delta_l - \delta_r}{2} & \gamma_l - \gamma_r \end{array} \right]. \]
After that the differentials alternate between $K_1$ and $K_2$, as $K_1 \circ K_2 = K_2 \circ K_1 = 0$.
We don't use this Koszul complex in this paper, but pieces of it appear in the full twist and its eigenmaps, and the computation that $K^2 = 0$ is a useful stepping stone to the more complicated computations below.

Note also that
\begin{equation} (\alpha_l - \alpha_r)(\delta_l + \delta_r) + (\beta_l - \beta_r)(\gamma_l + \gamma_r) = q_l - q_r = 0. \end{equation}
The corresponding Koszul complex is isomorphic to the one above. To see this, think of each instance of $B_{w_0}^{\oplus 2}$ as being $B_{w_0} \boxtimes V$ (or more precisely, since $V$ is three-dimensional, one should tensor with the two-dimensional span of the roots inside $V$). The multiplicity space $V$ admits a rotation which sends $\alpha \mapsto \gamma$ and $\beta \mapsto -\delta$ and also sends $\gamma \mapsto - \alpha$ and $\delta \mapsto \beta$. Applying this change of basis to each $B_{w_0}^{\oplus 2}$ yields the desired isomorphism up to sign.

\section{Type $C_2$ results} \label{sec:C2}

We describe the half twist in \S\ref{C2half}, the action of the half twist on indecomposables in \S\ref{C2halfacts}, and the full twist in \S\ref{C2full}, all in type $C_2$. We provide complexes defined over $\Z$ which descend to the minimal complex in characteristic $2$, and also the minimal complexes obtained after inverting $2$. Also, in \S\ref{C2eigen} we provide and discuss the eigenmaps for the full twist.

The lengthy computations which justify these complexes are performed in the supplemental document \cite{C2supplement}, but we explain our methods in \S\ref{C2methods}. The supplement \cite{C2supplement} also contains minimal complexes for other Rouquier complexes.

An underline indicates the term in homological degree zero.

\subsection{The half twist}\label{C2half}

Over any field of characteristic $\ne 2$ one has
\begin{subequations} \label{HTcharboth}
\begin{equation} \label{HTchar0} \HT \cong \left(
\begin{tikzpicture}
\node (w0) at (0,0) {$\un{B}_{w_0}(0)$};
\node (sts) at (2.5,.5) {$B_{sts}(1)$};
\node at (2.5,0) {$\oplus$};
\node (tst) at (2.5,-.5) {$B_{tst}(1)$};
\node (st) at (5,.5) {$B_{st}(2)$};
\node at (5,0) {$\oplus$};
\node (ts) at (5,-.5) {$B_{ts}(2)$};
\node (s) at (7.5,.5) {$B_{s}(3)$};
\node at (7.5,0) {$\oplus$};
\node (t) at (7.5,-.5) {$B_{t}(3)$};
\node (id) at (10,0) {$\one(4)$};
\path
	(w0) edge node[above] {$\wtootherthree$} (sts)
	(w0) edge node[below] {$-\wtothree$} (tst)
	(sts) edge node[above] {$d_1$} (st)
	(sts) edge (ts)
	(tst) edge (st)
	(tst) edge (ts)
	(st) edge node[above] {$d_2$} (s)
	(st) edge (t)
	(ts) edge (s)
	(ts) edge (t)
	(s) edge node[above] {$\finaldotred$} (id)
	(t) edge node[below] {$\finaldotblue$} (id);
\end{tikzpicture} \right). \end{equation}
One has
\begin{equation} \tilde{d}_1 = \left[ \begin{array}{cc} \linered \lineblue \finaldotred & \finaldotblue \linered \lineblue \\ \; & \; \\ \finaldotred \lineblue \linered & \lineblue \linered \finaldotblue \end{array} \right], \quad d_1 = \tilde{d}_1 \circ \left[ \begin{array}{c} {
\labellist
\tiny\hair 2pt
 \pinlabel {$3'$} [ ] at 19 23
\endlabellist
\centering
\ig{.5}{otherthree}
} \\ \; \\ \ethree \end{array} \right], \qquad d_2 = \left[ \begin{array}{cc} -\linered \finaldotblue & \finaldotblue \linered \\ \; & \; \\ \finaldotred \lineblue & - \lineblue \finaldotred \end{array} \right]. \end{equation}

Let us witness the phenomenon discussed in Remark \ref{rmk:indegree2c}. We have $\cb(\la_{\bug}) = 1$, and objects in cell $\la_{\bug}$ only appear in $\HT$ in degrees $\ge 1$. The distinguished involutions in this cell are $s$ and $t$, and in degree exactly $1$, we find $sts = \Schu(s)$ and $tst = \Schu(t)$. Similar statements hold for $\la_0$ and $\la_1$.

In contrast, over $\Z$ or in characteristic 2, one has
\begin{equation} \label{HTchar2} \HT \cong \left(
\begin{tikzpicture}
\node (w0) at (0,0) {$\un{B}_{w_0}(0)$};
\node (sts) at (2.5,.5) {$B_s B_t B_s(1)$};
\node at (2.5,0) {$\oplus$};
\node (tst) at (2.5,-.5) {$B_{tst}(1)$};
\node (snew) at (5,2.5) {$B_s(1)$};
\node at (5,1.5) {$\oplus$};
\node (st) at (5,.5) {$B_{st}(2)$};
\node at (5,0) {$\oplus$};
\node (ts) at (5,-.5) {$B_{ts}(2)$};
\node (s) at (7.5,.5) {$B_{s}(3)$};
\node at (7.5,0) {$\oplus$};
\node (t) at (7.5,-.5) {$B_{t}(3)$};
\node (id) at (10,0) {$\one(4)$};
\path
	(w0) edge node[above] {$\wtosts$} (sts)
	(w0) edge node[below] {$-\wtothree$} (tst)
	(sts) edge node[above] {$d_1'$} (st)
	(sts) edge (ts)
	(tst) edge (st)
	(tst) edge (ts)
	(sts) edge node[above] {$\redpitchforkdown$} (snew)
	(st) edge node[above] {$d_2$} (s)
	(st) edge (t)
	(ts) edge (s)
	(ts) edge (t)
	(snew) edge node[above] {$\trianglered$} (s)
	(s) edge node[above] {$\finaldotred$} (id)
	(t) edge node[below] {$\finaldotblue$} (id);
\end{tikzpicture} \right), \end{equation}
and no further Gaussian elimination can be applied. Note that the differential has zero component from $B_{tst}(1)$ to $B_s(1)$ or from $B_s(1)$ to $B_t(3)$. Here
\begin{equation} d_1' = \tilde{d}_1 \circ \left[ \begin{array}{c} \linered \lineblue \linered \\ \; \\ \ethree \end{array} \right]. \end{equation}
\end{subequations}
Also, the degree $+2$ \emph{triangle map} generates the kernel of a dot on top or on bottom:
\begin{equation} \label{triangledef} \trianglered := - \brokenred + \linered \barbred. \end{equation}
To verify that $d^2 = 0$ from $B_s B_t B_s(1)$ to $B_s(3)$, apply the three-way dot force \eqref{movingdot}.

Now consider the same phenomenon (see Remark \ref{rmk:indegree2credux}) for $p=2$. We have $\cb^p(\la_{\pbug}) = 1$, and objects in cell $\la_{\pbug}$ first appear in homological degree $1$. The distinguished involutions in this $p$-cell are $sts$ and $t$, and in homological degree $1$ we witness $sts = \pSchu(sts)$ and $tst = \pSchu(t)$. Similarly, $\cb^p(\la_s) = 2$, $\pSchu(s) = s$, and $B_s$ first appears in homological degree $2$.

After inverting $2$, the component of the differential $\redpitchforkdown$ in \eqref{HTchar2} becomes the projection map to a direct summand, and it is easily verified that
Gaussian elimination on \eqref{HTchar2} yields \eqref{HTchar0}.

\subsection{Action of the half twist} \label{C2halfacts}

Over $\Z$ or any field, one has
\begin{subequations}
\begin{equation} \HT \ot B_t \cong \left(
\begin{tikzpicture}
\node (a) at (0,0) {$\un{B}_{w_0}(-1)$};
\node (b) at (2.5,0) {$B_{tst}(0)$};
\node at (1,-1) {};
\path
	(a) edge node[above] {$\wtothree$} (b);
\end{tikzpicture} \right). \end{equation}

\begin{equation} \HT \ot B_{st} \cong \left(
\begin{tikzpicture}
\node (a) at (0,0) {$\un{B}_{w_0}(-2)$};
\node (b) at (2.5,0) {$B_{st}(0)$};
\node at (1,-.5) {};
\path
	(a) edge node[above] {$\wtost$} (b);
\end{tikzpicture} \right). \end{equation}

\begin{equation} \HT \ot B_{ts} \cong \left(
\begin{tikzpicture}
\node (a) at (0,0) {$\un{B}_{w_0}(-2)$};
\node (b) at (2.5,0) {$B_{ts}(0)$};
\node at (1,-.5) {};
\path
	(a) edge node[above] {$\wtots$} (b);
\end{tikzpicture} \right). \end{equation}

\begin{equation} \HT \ot B_{tst} \cong \left(
\begin{tikzpicture}
\node (a) at (0,0) {$\un{B}_{w_0}(-3)$};
\node (b) at (2.5,0) {$B_{t}(0)$};
\node at (1,-.5) {};
\path
	(a) edge node[above] {$\wtot$} (b);
\end{tikzpicture} \right). \end{equation}
\end{subequations}

Over any field of characteristic $\ne 2$ one has
\begin{subequations}
\begin{equation} \label{HTBschar0} \HT \ot B_s \cong \left(
\begin{tikzpicture}
\node (a) at (0,0) {$\un{B}_{w_0}(-1)$};
\node (b) at (2.5,0) {$B_{sts}(0)$};
\node at (1,1) {};
\node at (1,-1) {};
\path
	(a) edge node[above] {$\wtootherthree$} (b);
\end{tikzpicture} \right). \end{equation}
In contrast, over $\Z$, one has
\begin{equation} \label{HTBschar2} \HT \ot B_s \cong \left(
\begin{tikzpicture}
\node (a) at (0,0) {$\un{B}_{w_0}(-1)$};
\node (b) at (2.5,0) {$B_s B_t B_s(0)$};
\node (c) at (5,0) {$B_s(0)$};
\node at (2,-.5) {};
\path
	(a) edge node[above] {$\wtosts$} (b)
	(b) edge node[above] {$\redpitchforkdown$} (c);
\end{tikzpicture} \right). \end{equation}
\end{subequations}

Over any field of characteristic $\ne 2$ one has
\begin{subequations}
\begin{equation} \label{HTBsts} \HT \ot B_{sts} \cong \left(
\begin{tikzpicture}
\node (a) at (0,0) {$\un{B}_{w_0}(-3)$};
\node (b) at (2.5,0) {$B_{s}(0)$};
\node at (1,-.5) {};
\path
	(a) edge node[above] {$\wtos$} (b);
\end{tikzpicture} \right). \end{equation}
Over $\Z$ one has
\begin{equation} \label{HTBsBtBschar2} \HT \ot B_s B_t B_s \cong \left(
\begin{tikzpicture}
\node (a) at (0,0) {$\un{B}_{w_0} \un{B}_s(-2)$};
\node (b) at (3,0) {$B_s B_t B_s(0)$};
\node at (1,-.5) {};
\path
	(a) edge node[above] {$\wtost \linered$} (b);
\end{tikzpicture} \right) \cong \left(
\begin{tikzpicture}
\node (z) at (0,0) {$\oplus$};
\node (a) at (0,-.5) {$\un{B}_{w_0}(-1)$};
\node (b) at (3.5,0) {$B_s B_t B_s(0)$};
\node (c) at (0,.5) {$\un{B}_{w_0}(-3)$};
\path
	(a) edge node[below] {$\ig{1.1}{wtosts}$} (b)
	(c) edge node[above] {${
\labellist
\small\hair 2pt
 \pinlabel {$\chi$} [ ] at 25 24
\endlabellist
\centering
\ig{1.1}{wtosts}
}
$} (b);
\end{tikzpicture} \right). \end{equation}
\end{subequations}
Here $\chi = \alpha - \varpi$. The final isomorphism uses the decomposition given in \eqref{wsip}.

After base change to a field of characteristic $\ne 2$, $B_s B_t B_s \cong B_{sts} \oplus B_s$, and \eqref{HTBsBtBschar2} splits as a direct sum of \eqref{HTBsts} and \eqref{HTBschar0}. Over $\Z$ or in characteristic $2$, the complex \eqref{HTBsBtBschar2} is indecomposable.

Finally, over $\Z$ or any field one has
\begin{equation} \HT \ot B_{w_0} \cong \un{B}_{w_0}(-4). \end{equation}

\subsection{The full twist} \label{C2full}

After inverting $2$, the full twist has the following minimal complex.
\begin{equation} \label{FTchar0} \FT \cong \left(
\begin{tikzpicture}
\node (w00) at (0,0) {$\un{B}_{w_0}(-4)$};
\node (w01a) at (2.5,.5) {$B_{w_0}(-2)$};
\node (w01b) at (2.5,-.5) {$B_{w_0}(-2)$};
\node (deg1) at (2.5,0) {$\oplus$};
\node (sts2) at (5,1.5) {$B_s(1)$};
\node at (5,1) {$\oplus$};
\node (t2) at (5,.5) {$B_t(1)$};
\node (w02a) at (5,-.5) {$B_{w_0}(0)$};
\node at (5,-1) {$\oplus$};
\node (w02b) at (5,-1.5) {$B_{w_0}(0)$};
\node (deg2) at (5,0) {$\oplus$};
\node (st3) at (7.5,1.5) {$B_{st}(2)$};
\node at (7.5,1) {$\oplus$};
\node (ts3) at (7.5,0.5) {$B_{ts}(2)$};
\node (deg3) at (7.5,0) {$\oplus$};
\node (w03a) at (7.5,-.5) {$B_{w_0}(2)$};
\node at (7.5,-1) {$\oplus$};
\node (w03b) at (7.5,-1.5) {$B_{w_0}(2)$};
\node (sts4) at (10,1) {$B_{sts}(3)$};
\node (tst4) at (10,0) {$B_{tst}(3)$};
\node at (10,-.5) {$\oplus$};
\node (w04) at (10,-1) {$B_{w_0}(4)$};
\node (deg4) at (10,.5) {$\oplus$};

\node (sts5) at (2.5,-3.5) {$B_{sts}(5)$};
\node (deg5) at (2.5,-4) {$\oplus$};
\node (tst5) at (2.5,-4.5) {$B_{tst}(5)$};
\node (st6) at (5,-3.5) {$B_{st}(6)$};
\node (deg6) at (5,-4) {$\oplus$};
\node (ts6) at (5,-4.5) {$B_{ts}(6)$};
\node (s7) at (7.5,-3.5) {$B_s(7)$};
\node (t7) at (7.5,-4.5) {$B_t(7)$};
\node (deg7) at (7.5,-4) {$\oplus$};
\node (id) at (10,-4) {$\one(8)$};
\path
	(w00) edge node[above] {$d_0$} (deg1)
	(deg1) edge node[above] {$d_1$} (deg2)
	(deg2) edge node[above] {$d_2$} (deg3)
	(deg3) edge node[above] {$d_3$} (tst4)
	(tst4) edge[out=355, in=175, /tikz/overlay] node[below] {$d_4$} (deg5)
	(deg5) edge node[above] {$d_5$} (deg6)
	(deg6) edge node[above] {$d_6$} (deg7)
	(deg7) edge node[above] {$d_7$} (id);
\end{tikzpicture}
\qquad
\right). \end{equation}

We have
\begin{subequations}
\renewcommand{\arraystretch}{3}
\begin{equation} d_0 = \left[ \begin{array}{c} y - 2 \gamma \\ \hline 2 \delta - 2x \end{array} \right], \qquad d_1 = \left[ \begin{array}{c|c}
\wtos & 0 \\
\hline
0 & \wtot \\
\hline
y - \delta & y - x - \gamma \\
\hline
2x - y - \delta & x - \gamma
\end{array} \right],
\end{equation}
\begin{equation} d_2 = \left[ \begin{array}{c|c|c|c}
- \linered \startdotblue & \startdotred \lineblue &  \wtost & 0 \\
\hline \startdotblue \linered & - \lineblue \startdotred & 0 & \wtots \\
\hline 0 & 0 & \beta - y & y - 2x - \beta \\
\hline 0 & 0 & y - x + \alpha & x - \alpha
\end{array} \right], \end{equation}
\begin{equation} d_3 = \left[ \begin{array}{c|c|c|c}
2 \; \sttoother & 2 \; \tstoother & \wtootherthree & 0 \\
\hline \sttothree & \tstothree & 0 & -\wtothree \\
\hline 2 \sttow & 2 \tstow & x - \alpha & \beta - y
\end{array} \right], \end{equation}
\begin{equation} d_4 = \left[ \begin{array}{c|c|c}
\othertriother - \othertsother & -2 \threestother & \wtootherthree \\
\hline \othertsthree & \threestthree - \threetrithree & - \wtothree
 \end{array} \right], \end{equation}
\begin{equation} d_5 = \left[ \begin{array}{c|c}
\othertost & \threetost \\
\hline \othertots & \threetots
 \end{array} \right], \qquad d_6 = \left[ \begin{array}{c|c}
- \lineblue \finaldotblue & \finaldotblue\linered \\
\hline \finaldotred \lineblue & - \lineblue \finaldotred
\end{array} \right], \qquad d_7 = \left[ \begin{array}{c|c} \finaldotred & \finaldotblue \end{array} \right], \end{equation}
\end{subequations}

The full twist in type $C_2$ over $\Z$ has the following minimal complex.

\begin{equation} \label{FTchar2} \FT \cong \left(
\begin{tikzpicture}
\node (w00) at (0,0) {$\un{B}_{w_0}(-4)$};
\node (w01a) at (2.5,.5) {$B_{w_0}(-2)$};
\node (w01b) at (2.5,-.5) {$B_{w_0}(-2)$};
\node (deg1) at (2.5,0) {$\oplus$};
\node (sts2) at (5,1.5) {$B_s B_t B_s(1)$};
\node at (5,1) {$\oplus$};
\node (t2) at (5,.5) {$B_t(1)$};
\node (w02a) at (5,-.5) {$B_{w_0}(0)$};
\node at (5,-1) {$\oplus$};
\node (w02b) at (5,-1.5) {$B_{w_0}(0)$};
\node (deg2) at (5,0) {$\oplus$};
\node (sts3) at (7.5,2) {$B_s B_t B_s(1)$};
\node at (7.5,1.5) {$\oplus$};
\node (st3) at (7.5,1) {$B_{st}(2)$};
\node at (7.5,.5) {$\oplus$};
\node (ts3) at (7.5,0) {$B_{ts}(2)$};
\node at (7.5,-.5) {$\oplus$};
\node (w03a) at (7.5,-1) {$B_{w_0}(2)$};
\node at (7.5,-1.5) {$\oplus$};
\node (w03b) at (7.5,-2) {$B_{w_0}(2)$};
\node (s4) at (10,1.5) {$B_s(1)$};
\node at (10,1) {$\oplus$};
\node (sts4) at (10,.5) {$B_sB_t B_s(3)$};
\node (tst4) at (10,-.5) {$B_{tst}(3)$};
\node at (10,-1) {$\oplus$};
\node (w04) at (10,-1.5) {$B_{w_0}(4)$};
\node (deg4) at (10,0) {$\oplus$};

\node (s5) at (2.5,-3) {$B_s(3)$};
\node at (2.5,-3.5) {$\oplus$};
\node (sts5) at (2.5,-4) {$B_sB_tB_s(5)$};
\node at (2.5,-4.5) {$\oplus$};
\node (tst5) at (2.5,-5) {$B_{tst}(5)$};
\node (s6) at (5,-3) {$B_s(5)$};
\node at (5,-3.5) {$\oplus$};
\node (st6) at (5,-4) {$B_{st}(6)$};
\node at (5,-4.5) {$\oplus$};
\node (ts6) at (5,-5) {$B_{ts}(6)$};
\node (s7) at (7.5,-3.5) {$B_s(7)$};
\node (t7) at (7.5,-4.5) {$B_t(7)$};
\node (deg7) at (7.5,-4) {$\oplus$};
\node (id) at (10,-4) {$\one(8)$};
\path
	(w00) edge node[above] {$d_0$} (deg1)
	(deg1) edge node[above] {$d_1$} (deg2)
	(deg2) edge node[above] {$d_2$} (ts3)
	(ts3) edge node[above] {$d_3$} (deg4)
	(deg4) edge[out=355, in=175, /tikz/overlay] node[below] {$d_4$} (sts5)
	(sts5) edge node[above] {$d_5$} (st6)
	(st6) edge node[above] {$d_6$} (deg7)
	(deg7) edge node[above] {$d_7$} (id);
\end{tikzpicture}
\right). \end{equation}

We have
\begin{subequations}
\renewcommand{\arraystretch}{3}
\begin{equation} d_0 = \left[ \begin{array}{c} y - 2 \gamma \\ \hline x - \delta \end{array} \right], \qquad d_1 = \left[ \begin{array}{c|c}
- \wtostswtri - \ig{.5}{wtopitchred} & \wtostswtri - \wtost \startdotred \\
\hline
0 & \wtot \\
\hline
y - 2x + \delta & 2(x - \gamma) \\
\hline
x - y & y - 2x
\end{array} \right],
\end{equation}
\begin{equation} d_2 = \left[ \begin{array}{c|c|c|c}
2 \linered \lineblue \linered + \doubleredpitch & 0 & \wtosts & 0 \\
\hline -\linered \lineblue \finaldotred & \startdotred \lineblue & 0 & \wtost \\
\hline - \finaldotred \lineblue \linered - \startdotblue \redpitchforkdown & - \lineblue \startdotred & 0 & 0 \\
\hline 2 \ststow & 0 & y-x & \beta -y \\
\hline -\ststow & 0 & 0 & y - x + \alpha
\end{array} \right], \end{equation}
\begin{equation} d_3 = \left[ \begin{array}{c|c|c|c|c}
\redpitchforkdown & 0 & 0 & 0 & 0 \\
\hline - \linered \lineblue \linered \beta + \brokenred \lineblue \linered - \linered \lineblue \brokenred + \redpitchforkdown \startdotblue \startdotred & 2 \linered \lineblue \startdotred + \redpitchforkup \finaldotblue & 2 \startdotred \lineblue \linered + \finaldotblue \redpitchforkup & \wtosts & 0 \\
\hline 0 & \sttothree & \tstothree & 0 & -\wtothree \\
\hline \ststowwtri & 2 \sttow & 2 \tstow & x - \alpha & \beta - y
\end{array} \right], \end{equation}
\begin{equation} d_4 = \left[ \begin{array}{c|c|c|c}
\trianglered & \redpitchforkdown & 0 & 0 \\
\hline 0 & \linered\lineblue\trianglered - \brokenred\lineblue\linered - \finaldotred \finaldotblue \redpitchforkup & -2 \threetosttosts - \threetostosts & \wtosts \\
\hline 0 & \ststotstothree & \threestthree - \threetrithree & - \wtothree
 \end{array} \right], \end{equation}
\begin{equation} d_5 = \left[ \begin{array}{c|c|c}
-\brokenred & \redpitchforkdown & 0 \\
\hline 0 & \linered\lineblue\finaldotred & \threetost \\
\hline 0 & \finaldotred \lineblue\linered & \threetots
 \end{array} \right], \qquad d_6 = \left[ \begin{array}{c|c|c}
\trianglered & - \lineblue \finaldotblue & \finaldotblue\linered \\
\hline 0 & \finaldotred \lineblue & - \lineblue \finaldotred
\end{array} \right], \qquad d_7 = \left[ \begin{array}{c|c} \finaldotred & \finaldotblue \end{array} \right], \end{equation}
\end{subequations}

\subsection{Eigenmaps} \label{C2eigen}

The candidate eigenmaps over $\Z$ are described as follows. Checking that they are chain maps is an easy exercise.

\begin{subequations} \label{eigenmaps}
\begin{equation} \al_0 \co \one(-8)[0] \to \FT \qquad \qquad \left[ \idtow \right] \end{equation}
\begin{equation} \al_{\pbug} \co \one(0)[2] \to \FT \qquad \qquad \left[ \begin{array}{c} \pitchcupred \\ 0 \\ \startdotblue \\ 0 \\ 0 \end{array} \right] \end{equation}
\begin{equation} \al_{s} \co \one(0)[4] \to \FT \qquad \qquad \left[ \begin{array}{c} \startdotred \\ 0 \\ 0 \\ 0 \end{array} \right] \end{equation}
\begin{equation} \al_{1} \co \one(8)[8] \to \FT \qquad \qquad \left[ \begin{array}{c} 1 \end{array} \right]. \end{equation}
\end{subequations}

When $2$ is inverted, we instead have
\begin{equation} \al_{\bug} \co \one(0)[2] \to \FT \qquad \qquad \left[ \begin{array}{c} \startdotred \\ \startdotblue \\ 0 \\ 0 \end{array} \right]. \end{equation}

To prove that they are eigenmaps, we make a digression over several sections to discuss Frobenius algebra objects. Many of the ideas below owe their conception to conversations
between the first author and Matt Hogancamp, and are implicit in \cite{EHDiag2}.

\subsection{Frobenius algebra objects and distinguished involutions} \label{frobalgobj}

See \cite[\S 8]{EMTW} for additional background on Frobenius algebra objects in a monoidal category. The reader may wish to recall Notation \ref{notation:val}.

The object $B_s$ is a (graded) Frobenius algebra object of degree $1$: it has a unit $\startdotred$ of degree $1$, a multiplication map $\mergered$ of degree $-1$, and a counit
and comultiplication map given by flipping these diagrams upside-down. The relations among diagrams which involve only one color are effectively just the relations stating that
$B_s$ is a Frobenius algebra object. The Frobenius algebra structure is related to the decomposition of $B_s \ot B_s$. For example, $\startdotred \linered$ is the inclusion map of
a summand $B_s(-1)$ inside $B_s B_s$, with projection map $\mergered$; they compose to the identity by the unit axiom \eqref{eq:unitaxiom}.

Similarly, $B_{w_0}$ is also a Frobenius algebra object of degree $\ell(w_0) = 4$, since $R^{W} \subset R$ is a Frobenius extension. The unit is $\idtow$. The
multiplication map is the projection $B_{w_0} \ot B_{w_0} \to B_{w_0}(-4)$ to the unique negative-most degree summand. See \cite[\S 6.3]{ECathedral} for more details on the
multiplication map.

\begin{prop} \label{prop:onlydistcanbefrob} In $\HC^0$, if the indecomposable object $B_w$ is a Frobenius algebra object of degree $a$, then $w$ is a distinguished involution, and $a = \ab(w)$. Similarly, in $\HC^p$, if the indecomposable object $\pB_w$ is a Frobenius algebra object of degree $a$, then $w$ is a $p$-distinguished involution in the sense of Conjecture \ref{conj:distinv}, and $-\val(\mu^w_{w,w}) \ge a \ge \val(h^1_w)$. \end{prop}

We prove Proposition \ref{prop:onlydistcanbefrob} later in this section. These two examples in type $C_2$ are completely illustrative of the general principles behind this
proposition.

\begin{ex} The indecomposable object $B_{tst}$ is not a graded Frobenius algebra object of any degree. This can be seen directly from the graded dimensions of Hom spaces. Any map
$\one \to B_{tst}$ has degree at least $+3$, so the unit has such a degree. However, there are no morphisms $B_{tst} \ot B_{tst} \to B_{tst}$ of degree $\le -3$ to serve as the
multiplication map. \end{ex}

\begin{ex} The indecomposable objects $B_{st}$ is not graded Frobenius algebra object, because it is not self-biadjoint. \end{ex}

What is interesting in type $C_2$ for $p=2$ is that there is a new $p$-distinguished involution, namely $\pB_{sts} = B_s B_t B_s$ in cell $\la_{\pbug}$. The Bott-Samelson objects $B_s B_t B_s$ and $B_t B_s B_t$ are both Frobenius algebra objects of degree $1$ (in any characteristic). For $B_s B_t B_s$ the structure maps are
\begin{equation} \text{unit} = \pitchcupred, \qquad \text{multiplication} = \ig{.5}{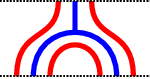}. \end{equation} For the same reasons as
above, $\pitchcupred \linered \lineblue \linered$ is the inclusion map of a direct summand $B_s B_t B_s(-1)$ inside $(B_s B_t B_s) \ot (B_s B_t B_s)$.

\begin{remark} We are unaware of any literature on Frobenius algebra objects possessing direct summands which are also Frobenius algebra objects. There should be some relationship
between the Frobenius algebra structures on $B_t B_s B_t$ and its direct summand $B_t$ (and similarly for $B_s B_t B_s$ and $B_s$ when $2$ is inverted), but we do not know how to
make this relationship precise. The unit map for $B_t B_s B_t$ is a composition of the unit map for $B_t$ with the inclusion map $B_t \to B_t B_s B_t$, and is therefore orthogonal
to the complementary summand $B_{tst}$. The multiplication map is not orthogonal to $B_{tst}$ (nor can it be, for the unit axiom to hold). \end{remark}

Here is the postponed proof of Proposition \ref{prop:onlydistcanbefrob}. 

\begin{proof} Let $B$ be a Frobenius algebra object. By composing the counit with the multiplication, one obtains a degree zero map $B \ot B \to \one$ which is typically called
the cap. Similarly, one can define the cup $\one \to B \ot B$. The cap and cup satisfy isotopy relations, implying that $B \ot (-)$ is self-biadjoint (the cup and cap are the unit
and counit of adjunction). The monoidal adjoint of $B_w$ is $B_{w^{-1}}$, so $B_w$ is self-biadjoint if and only if $w$ is an involution. The same statement can be made for
$\pB_w$.

Suppose $w$ is an involution and $B_w$ is a Frobenius algebra object in $\HC^0$. The minimal degree of any morphism $\one \to B_w$ is $\val(h^1_w)$, so $a \ge \val(h^1_w)$. The
minimal degree of any morphism $B_w \ot B_w \to B_w$ is $\val(\mu^w_{w,w})$, so $-a \ge \val(\mu^w_{w,w}) \ge -\ab(w)$. Thus $\ab(w) \ge a \ge \val(h^1_w)$. Since $\val(h^1_w) \ge
\ab(w)$, with equality if and only if $w$ is distinguished, we deduce that $a = \ab(w)$ and $w$ is distinguished.

The same initial argument can be made for $\pB_w$ in $\HC^p$, to deduce that $-\val(\mu^w_{w,w}) \ge a \ge \val(h^1_w)$. According to Conjecture \ref{conj:distinv}, there is a
unique involution in each left or right $p$-cell for which $-\val(\mu^w_{w,w}) \ge \val(h^1_w)$, and these are precisely the $p$-distinguished involutions, so $w$ must be
$p$-distinguished. However, this inequality can be strict, so we can not pin down the value of $a$. \end{proof}

\begin{remark} For a $p$-distinguished involution $d$ in type $C_2$, $\val(h^1_d) = -\val(\mu^d_{d,d})$, and both agree with $\ab(d)$ for the usual $\ab$-function. Thus the degree
of any Frobenius algebra object $\pB_d$ must be $\ab(d)$. \end{remark}

Let us discuss the converse. In type $C_2$ in characteristic $2$, the $p$-distinguished involutions are $\{1, s, t, sts, w_0\}$, and $\pB_d$ is a Frobenius algebra object for each
one. For an arbitrary Coxeter group, in characteristic zero, it is conjectured \cite[Conjecture 4.40]{EHDiag2} and \cite[\S 5.2]{Klein} that $B_d$ is a graded Frobenius algebra
object in $\HC^0$ whenever $d$ is a distinguished involution. A weaker statement, that the image of $B_d$ is a graded Frobenius algebra object in the associated graded of the cell
filtration, was proven in \cite[\S 4.4]{MMMTZ}. In characteristic $p$, we do not yet have enough evidence to state a conjecture confidently. Instead we state a dream, which we do
not necessarily believe is true. See Conjecture \ref{conj:distinv} for notation.

\begin{dream} For Weyl groups, in any characteristic, the indecomposable object $\pB_d$ associated to a $p$-distinguished involution $d$ is a graded Frobenius algebra object of
degree $\val(h^1_d)$. \end{dream}

\begin{remark} A weaker dream, which we expect to be true, is that $\pB_d$ is a graded Frobenius algebra object when $d$ is a $p$-distinguished involution and $\val(h^1_d) = -\val(\mu^d_{d,d})$. \end{remark}
	
\subsection{Frobenius modules, distinguished involutions, and the $J$-ring} \label{frobmodule}

Let $B$ be a graded Frobenius algebra object of degree $a$, with unit map $\eta_B \co \one(-a) \to B$. Several times in the previous section we emphasized the property that $\eta_B \ot \id_B \co
B(-a) \to B \ot B$ is the inclusion of a direct summand, with projection map given by multiplication. The unit axiom implies that the composition of projection and inclusion is
$\id_B$, as required. However, the unit map $\eta_B$ gives rise to many more direct sum decompositions, as we now explain in the example $B = B_s$.

Let $w$ be any element with $sw < w$. Then $B_s B_w \cong B_w(-1) \oplus B_w(1)$, and there is a projection map $p_w \co B_s B_w \to B_w(-1)$. An interesting example is in type
$C_2$ with $w = w_0$, where $p_w$ is a thick trivalent vertex (the 180 degree rotation of \eqref{thicktri}). More generally, one can construct $p_w$ by including $B_w$ into a
Bott-Samelson object for a reduced expression of $w$ which starts with $s$, applying the ordinary trivalent vertex $B_s B_s \to B_s(-1)$, and then projecting again from the
Bott-Samelson to $B_w$. A consequence of this description is that
\begin{equation} p_w \circ (\eta_s \ot \id_w) = \id_w. \end{equation}
When $w = s$, one just recovers the unit axiom; this is a generalization.

\begin{defn} Let $B$ be a graded Frobenius algebra object of degree $a$ in a graded monoidal category, with unit map $\eta_B$ of degree $a$. A \emph{Frobenius module} $M$ is just
a module object for $B$ (when viewed as an algebra object), where the multiplication map $p_M \co B \ot M \to M$ has degree $-a$. \end{defn}

We call $M$ a Frobenius module (rather than just a module) to emphasize some of the other structures it possesses. For example, using the self-adjunction of $B$ one can construct a comultiplication map $M \to B \ot M$, and one can verify coassociativity and the counit axiom. Note that the unit axiom states that
\begin{equation} p_M \circ (\eta_B \ot \id_M) = \id_M. \end{equation}
	
Clearly $B$ is a Frobenius module over itself. The example above indicates that $B_w$ is a Frobenius module over $B_s$ whenever $sw<w$. The unit axiom implies in general that
$M(-a)$ is a direct summand of $B \ot M$, with projection map $p_M$ and inclusion map $\eta_B \ot \id_M$.

Now we explain why one expects many Frobenius modules over a distinguished involution. One key property of distinguished involutions is that they function as the
identity element of Lusztig's $J$-ring. Let us unravel this statement for the non-expert. In characteristic zero, multiplicities in direct sum decompositions are determined by the
behavior of the KL basis in the Hecke algebra. Whenever $w, x, y$ are all in cell $\la$, the minimal integer $k$ for which $B_w(k)$ is a direct summand of $B_x \ot B_y$ satisfies
$k \ge -\ab(\la)$. Moreover, this minimal shift can be attained using distinguished involutions. If $d$ is the distinguished involution in the same right cell as $w$, then
\begin{equation} \label{identityofJring} B_w(-\ab(\la)) \sumset B_d \ot B_w, \end{equation} that is, $B_w(-\ab(\la))$ appears as a direct summand (with multiplicity 1) inside $B_d
\ot B_w$. Similarly, if $d'$ is the distinguished involution in the same left cell as $w$, then $B_w(-\ab(\la))$ appears as a direct summand (with multiplicity 1) inside $B_w \ot
B_d'$.

Now thinking with morphisms, the existence of this special direct summand \eqref{identityofJring} implies the existence of a projection map $p_w \co B_d \ot B_w \to B_w$ of degree
$-\ab(\la)$. The lack of other summands in this degree implies that the morphism space in which $p_w$ lives is one-dimensional modulo $\IC_{< \la}$. Assume that $B_w$ is a
Frobenius algebra object. The composition $p_w \circ (\eta_d \ot \id_w)$ is some scalar multiple of $\id_w$, and if this scalar is nonzero we can rescale $p_w$ to assume that the
unit axiom holds. Roughly speaking, this implies $B_w$ is a Frobenius module over $B_d$ (see the next remark). If $p_w \circ (\eta_d \ot \id_w) = 0$, we would be out of luck; we conjecture below that this never happens.

\begin{remark} The morphism space $B_d \ot B_d \ot B_w \to B_w$ in degree $-2\ab(\la)$ is also one-dimensional modulo $\IC_{< \la}$, so that associativity must hold up to scalar,
modulo $\IC_{< \la}$. If the unit axiom holds, then by composing both sides of the associativity relation with the unit $\eta_d \ot \id_d \ot \id_w$, we see that the scalar must
be $1$. So the unit axiom implies the associativity axiom modulo $\IC_{< \la}$, and $B_w$ is a Frobenius module over $B_d$ in this quotient category. In all examples we have computed, associativity holds on the nose. \end{remark}

The following is only a slight elaboration upon \cite[Conjecture 4.40]{EHDiag2}.

\begin{conj} \label{conj:iotadoesit} Let $W$ be a finite Coxeter group, and $d$ a distinguished involution in cell $\la$. Not only is $B_d$ a graded Frobenius algebra object of
degree $\ab(\la)$ in $\HC^0$, but the unit map $\eta_d$ satisfies the property that $\eta_d \ot \id_w$ is the inclusion map of a direct summand whenever $w$ is in the same right
cell as $d$. Moreover, $B_w$ is a Frobenius module over $B_d$. \end{conj}
	
\begin{dream} \label{dream:iotadoesit} Let $W$ be a Weyl group. Let $d$ be a $p$-distinguished involution (for some characteristic $p$, possibly zero) for which $\val(h^1_d) = -\val(\mu^d_{d,d}) =: a$. Not only is $\pB_d$ a graded Frobenius algebra object of degree $a$, but the unit map $\eta_d$ satisfies the property that $\eta_d \ot \id_w$
is the inclusion map of a direct summand whenever $w$ is in the same right $p$-cell as $d$. Moreover, $\pB_w$ is a Frobenius module over $\pB_d$. \end{dream}

By direct inspection, one can verify Dream \ref{dream:iotadoesit} in type $C_2$. We leave this to the reader, the only interesting remaining case being $d = sts$ and $w = st$ for
$p=2$. A statement closely related to Conjecture \ref{conj:iotadoesit} is proven in \cite[Proposition 4.38]{EHDiag2} in type $A$.

\subsection{Frobenius units and eigenmaps} \label{frobunitandeigen}

Finally, let us look again at the eigenmaps constructed in \eqref{eigenmaps}. In this discussion, $p$ is arbitrary (possibly zero). For each cell $\la$, the objects appearing in
the full twist in homological degree $2 \pcb(\la)$ and cell $\la$ are precisely the $p$-distinguished involutions with a particular degree shift, though there are other summands
in lower cells. Consequently, any chain map from $\one(2\pxb(\la))[2\pcb(\la)]$ to $\FT$ will consist of some multiple of the Frobenius unit $\eta_d$ mapping to $\pB_d$, as well
as unspecified maps to the summands in lower cells. Crucially, we can hope that the scalar on each Frobenius unit is nonzero.

\begin{dream} \label{dream:eigenmapsarealmostunits} Fix a $p$-cell $\la$, and assume for each $p$-distinguished involution therein that $\val(h^1_d) = -\val(\mu^d_{d,d}) =: a$,
and that $\pB_d$ is a Frobenius algebra object of degree $a$. Consider the homological degree $2\pcb(\la)$ inside the minimal complex of the full twist. The objects which appear
in this degree in cell $\la$ are precisely the $p$-distinguished involutions with grading shift $2\pxb(\la) + a$. Moreover, there is some chain map from
$\one(2\pxb(\la))[2\pcb(\la)]$ to $\FT$ which agrees modulo lower cells with a sum of Frobenius units multiplied by nonzero scalars. We denote this chain map $\al_{\la}$. \end{dream}

Amazingly, the maps $\al_{\la}$ constructed in \eqref{eigenmaps} above do not involve any terms going to lower cells (or scalars on the Frobenius units, though this is a matter of
normalization). Instead, it seems that $\sum \eta_d$ is actually a chain map, landing in the kernel of the differential on $\FT$! We know of no reason why the sum of
Frobenius units should be a chain map, but this happens in all cases we have computed. Again, we do not yet have enough evidence to state the below as a conjecture, though it is not necessary for any further arguments.

\begin{dream} \label{dream:eigenmapsareunits} Continuing Dream \ref{dream:eigenmapsarealmostunits}, the map from $\one(2\pxb(\la))[2\pcb(\la)]$ to $\FT$ which is the sum of the
Frobenius units is actually a chain map. \end{dream}

Now we have a useful technical result for establishing that a candidate is an eigenmap. Though we require Dream \ref{dream:iotadoesit}, we do not require the part of that Dream which states that $\pB_w$ is a Frobenius module.

\begin{prop} \label{prop:dreamsmakeeigen} Assume that Dream \ref{dream:eigenmapsarealmostunits} and Dream \ref{dream:iotadoesit} and Conjecture \ref{conj:htcat} all hold. Then $\al_{\la}$ is a $\la$-eigenmap. \end{prop}

\begin{proof} We need to show that, for any $w \in \la$, $\al_{\la} \ot \pB_w$ is an isomorphism modulo $\IC^p_{< \la}$. For this purpose we work in the quotient $\HC^p/\IC^p_{<
\la}$. In this quotient, the minimal complex of the full twist is supported in degrees $\ge \cb^p(\la)$, so we are examining the leftmost nonzero degree. 

Let ${}^i \FT$ denote the $i$-th chain object in the minimal complex of $\FT$. For the sake of our argument we wish to distinguish between several complexes:
\begin{itemize}
\item $X = \FT \ot \pB_w$, where the $i$-th chain object is described as ${}^i \FT \ot \pB_w$.
\item $Y = \FT \ot \pB_w$, where the $i$-th chain object is described as a direct sum of indecomposable objects, and
\item $Z$, the minimal complex of $\FT \ot \pB_w$. By Conjecture \ref{conj:htcat}, $Z$ is just a single copy of $\pB_w(2\pxb(\la))$ in homological degree $2 \pcb(\la)$.
\end{itemize}
Of course $X$ is isomorphic to $Y$ and homotopy equivalent to $Z$, but it helps to keep track of these isomorphisms and homotopy equivalences explicitly. The map $\rho \co X \to Y$ is a sum of projection maps for some choice of decomposition. The map $\chi \co Y \to Z$ is the map induced by Gaussian elimination of contractible summands.

\begin{ex} Consider the cell $\la_{\bug}$ in characteristic zero. Taking the quotient by lower cells, $\FT$ is supported in degrees $\ge 2$, see \eqref{FTchar0}. Let $w = st$. Then
\begin{equation} {}^2 X = B_s B_{st}(1) \oplus B_t B_{st}(1), \qquad {}^2 Y = B_{st}(0) \oplus B_{st}(2) \oplus B_{tst}(1) \oplus B_t(1), \qquad {}^2 Z = B_{st}(0). \end{equation}
 \end{ex}

The crucial fact about Gaussian elimination that we use is the following (see \cite[Exercise 5.17]{EGaitsgory} for slightly more detail on this well-known tool). A single application of Gaussian elimination might replace
\[ Y = \left( \dots \to C \oplus D \to C \oplus E \to \dots \right) \]
(where the differential induces an isomorphism $C \to C$) with
\[ Z = \left( \dots \to D \to E \to \dots \right). \]
In the homotopy equivalence $Y \to Z$, the map from $C \oplus E \to E$ is not the naive projection $\left( \begin{array}{cc} 0 & \id_E \end{array} \right)$, but involves a nontrivial map $C \to E$. However, the map $C \oplus D \to D$ is the naive projection $\left( \begin{array}{cc} 0 & \id_D \end{array} \right)$! In particular, in the leftmost nonzero degree of any complex, Gaussian elimination will induce the naive projection map. 

The chain map $\chi \circ \rho \circ (\al_{\la} \ot \id_w)$ is supported in a single homological degree, where it is the composition
\begin{equation} \pB_w(2 \pxb(\la)) \to \bigoplus_{d'} \pB_{d'} \ot \pB_w(2 \pxb(\la) + a) \to \pB_w(2 \pxb(\la)) \oplus \bigoplus \text{other terms} \to \pB_w(2 \pxb(\la)). \end{equation}
The first sum is over all distinguished involutions $d'$ in the two-sided cell $\la$. The final map is the naive projection to one summand, and that summand only appears inside $\pB_d \ot \pB_w$ when $d$ is the unique distinguished involution in the same right cell as $w$. So $\rho$ and $\chi$ compose to projection to this one summand inside $\pB_d \ot \pB_w$. Hence the entire composition $\chi \circ \rho \circ (\al_{\la} \ot \id_w)$ is equal in this homological degree to
\begin{equation} \label{foocomp} \pB_w(2 \pxb(\la)) \to \pB_d \ot \pB_w(2 \pxb(\la) + a) \to \pB_w(2 \pxb(\la)). \end{equation}
The first map is a nonzero scalar multiple of $\eta_d \ot \id_w$, the component of $\al_{\la}$ mapping to $\pB_d \ot \pB_w$. The second map is the projection $p_w$, which is uniquely-defined up to a scalar modulo $\IC^p_{< \la}$, regardless of the choice of decomposition (as noted previously, this morphism space is one-dimensional in the quotient).  By  Dream \ref{dream:iotadoesit}, the composition \eqref{foocomp} is a nonzero multiple of $\id_w$. Thus $\al_{\la} \ot \pB_w$ is an isomorphism modulo $\IC^p_{< \la}$, as desired.
\end{proof}

\begin{rmk} The assumption that Dream \ref{dream:iotadoesit} holds is motivational, but could probably be removed by the following method. One would argue that the property that
$\al_{\la} \ot \pB_w$ is an isomorphism (modulo $\IC^p_{< \la}$) is constant over right $p$-cells. Suppose that $\al_{\la} \ot \id_{B_W} \co \pB_w \to \FT \ot \pB_w$ is an
isomorphism (modulo lower terms, ignoring shifts). Now tensor on the right with $\pB_x$, to obtain an isomorphism $\pB_w \pB_x \to \FT \ot \pB_w \pB_x$. Taking direct summands of
both sides (using an idempotent which commutes with $\al_{\la} \ot \id_{\pB_w \pB_x}$), $\pB_y \to \FT \ot \pB_y$ is an isomorphism too. In this fashion, one could reduce the proof that $\al_{\la}$ is an eigenmap to the case when $w$ is a $p$-distinguished involution, where it follows from the unit axiom as in the proof above.

This argument can also apply to left $p$-cells, provided one could prove that the full twist is in the Drinfeld center, and the eigenmaps are morphisms in the Drinfeld center.
Then one can left multiply $\FT \ot \pB_w$ by $\pB_x$, and obtain something canonically isomorphic to $\FT \ot \pB_x \pB_w$, compatibly with the map $\al_{\la}$. \end{rmk}

\begin{cor} \label{cor:mainthmproven} Theorem \ref{thm:C2char2diag}, stating the existence of eigenmaps in type $C_2$ for $p=2$, holds true. \end{cor}

\begin{proof} By construction, the chain maps of \eqref{eigenmaps} satisfy Dream \ref{dream:eigenmapsareunits}. The other dreams and conjectures have been verified in type $C_2$ earlier. Thus Proposition \ref{prop:dreamsmakeeigen} finishes the proof. \end{proof}

\subsection{The $s$-eigenmap and 2-torsion} \label{eigentorsion}

Let us consider the new eigenmap $\al_s$ in more detail. One can compute that the space of all chain maps $\one(0)[4] \to \FT$ modulo
homotopy is the $\Z$-span of $\al_s$. There is a homotopy, a non-chain map $\one(0)[3] \to \FT$ of the form \begin{equation} h = \left[ \begin{array}{c} \pitchcupred \\ \startdotred\startdotblue \\ 0 \\ 0 \\ 0
\end{array} \right], \qquad \qquad d_3 \circ h = \left[ \begin{array}{c} -2 \startdotred \\ 0 \\ \idtothree \\ \idtow \end{array} \right]. \end{equation} In particular, $2 \al_s$ is
homotopic to \begin{equation} \phi := \left[ \begin{array}{c} 0 \\ 0 \\ \idtothree \\ \idtow \end{array} \right]. \end{equation} Since $\phi$ is built from morphisms factoring through
objects in $< \la_s$, $\phi$ is inside $\IC^p_{< \la_s}$. However, $\al_s$ is nonzero in the quotient by $\IC^p_{< \la_s}$, since it is a $\la_s$-eigenmap. Thus the subspace of morphisms in $\IC^p_{< \la_s}$ has index $2$.

Now, when $2$ is inverted, we can work instead with the version of $\FT$ found in \eqref{FTchar0}. One can confirm that the space of chain maps modulo homotopy $\one(0)[4] \to \FT$ is generated by the image of $\phi$ composed with the Gaussian elimination homotopy equivalence, which we will also abusively call $\phi$:
\[ \phi = \left[ \begin{array}{c} 0 \\ \idtothree \\ \idtow \end{array} \right]. \]
From the perspective of \eqref{FTchar0}, it is surprising that $\frac{\phi}{2}$ has intrinsic meaning.

\subsection{Methods of computation} \label{C2methods}

We typically combined two methods to perform computations most effectively.

The first method is direct Gaussian elimination. For the computation of $\HT$, suppose we have computed already that (c.f. \eqref{HTcharboth} for the differentials)
\begin{subequations}
\begin{equation} F_{tst} = F_t F_s F_t \cong \left(
\begin{tikzpicture}
\node (tst) at (2.5,0) {$B_{tst}(0)$};
\node (st) at (5,.5) {$B_{st}(1)$};
\node at (5,0) {$\oplus$};
\node (ts) at (5,-.5) {$B_{ts}(1)$};
\node (s) at (7.5,.5) {$B_{s}(2)$};
\node at (7.5,0) {$\oplus$};
\node (t) at (7.5,-.5) {$B_{t}(2)$};
\node (id) at (10,0) {$\one(3)$};
\path
	(tst) edge (st)
	(tst) edge (ts)
	(st) edge node[above] {$d_2$} (s)
	(st) edge (t)
	(ts) edge (s)
	(ts) edge (t)
	(s) edge node[above] {$\finaldotred$} (id)
	(t) edge node[below] {$\finaldotblue$} (id);
\end{tikzpicture} \right). \end{equation}
Note that, like all Rouquier complexes for reduced expressions, this appears up to a shift as a subcomplex of $\HT$ itself.

Now we compute $F_{tst} \ot B_s$ by Gaussian elimination, postponing the question of why until after the computation.
\begin{equation} F_{tst} \ot B_s \cong \left(
\begin{tikzpicture}
\node (tst) at (2.5,0) {$B_{tst}B_s(0)$};
\node (st) at (5,.5) {$B_s B_t B_s(1)$};
\node at (5,0) {$\oplus$};
\node (ts) at (5,-.5) {$B_{ts} B_s(1)$};
\node (s) at (7.5,.5) {$B_{s} B_s(2)$};
\node at (7.5,0) {$\oplus$};
\node (t) at (7.5,-.5) {$B_{ts}(2)$};
\node (id) at (10,0) {$B_s(3)$};
\path
	(tst) edge (st)
	(tst) edge (ts)
	(st) edge node[above] {$d_2 \linered$} (s)
	(st) edge (t)
	(ts) edge (s)
	(ts) edge (t)
	(s) edge node[above] {$\finaldotred \linered$} (id)
	(t) edge node[below] {$\finaldotblue \linered$} (id);
\end{tikzpicture} \right). \end{equation}
The map $\splitred$ gives a splitting of the map $\finaldotred \linered$ in the last differential, allowing us to eliminate $B_s(3)$ with one summand of $B_s B_s(2)$. Having done this, the remaining differential from $B_s B_t B_s(1)$ to the other summand $B_s(1)$ is the original differential $-\linered \finaldotblue \linered$ to $B_s B_s(2)$, composed with the projection map $\mergered$ from $B_s B_s(2)$ to $B_s(1)$, see \eqref{BsBsip}.
\begin{equation} \left(
\begin{tikzpicture}
\node (tst) at (2.5,0) {$B_{tst}B_s(0)$};
\node (st) at (5,.5) {$B_s B_t B_s(1)$};
\node at (5,0) {$\oplus$};
\node (ts) at (5,-.5) {$B_{ts} B_s(1)$};
\node (s) at (7.5,.5) {$B_{s}(1)$};
\node at (7.5,0) {$\oplus$};
\node (t) at (7.5,-.5) {$B_{ts}(2)$};
\path
	(tst) edge (st)
	(tst) edge (ts)
	(st) edge node[above] {$-\redpitchforkdown$} (s)
	(st) edge (t)
	(ts) edge (s)
	(ts) edge (t);
\end{tikzpicture} \right). \end{equation}
The sign on the pitchfork can be removed later (up to isomorphism of complexes) by multiplying $B_s(1)$ by $-1$.

Similarly, using $\lineblue \splitred$ we can eliminate $B_{ts}(2)$ with one summand of $B_{ts} B_s(1)$. In theory this could modify the surviving differential from $B_s B_t B_s(1)$ to $B_s(1)$ by subtracting a \emph{zigzag term}, the (differential-splitting-differential) composition $B_s B_t B_s(1) \to B_{ts}(2) \to B_{ts}(2) \to B_s(1)$. One can compute that this zigzag term is
\[ \finaldotred \finaldotblue \ig{.5}{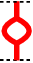} = 0.\]

To cancel the summand $B_{ts}(0)$ inside $B_{ts} B_s(1)$ with the summand $B_{ts}(0)$ inside $B_{tst} B_s(0)$, we can use the splitting given by the map $i_{ts}$ from \eqref{BtstBsip}. Again, one can compute that no zigzag term will affect the differential $B_{w_0}(0) \to B_s B_t B_s(1)$ that remains. So we compute the minimal complex
\begin{equation} \label{FtstBs} F_{tst} \ot B_s \cong \left(
\begin{tikzpicture}
\node (a) at (0,0) {$\un{B}_{w_0}(0)$};
\node (b) at (2.5,0) {$B_s B_t B_s(1)$};
\node (c) at (5,0) {$B_s(1)$};
\node at (2,-.5) {};
\path
	(a) edge node[above] {$\wtosts$} (b)
	(b) edge node[above] {$\redpitchforkdown$} (c);
\end{tikzpicture} \right). \end{equation}
\end{subequations}

On the other hand, suppose we do the computation more lazily, only keeping track of enough information to ensure that the terms we expect to Gaussian eliminate will be connected by isomorphisms (rather than zero, see the warning in \cite[Exercise 19.14]{EMTW}). We will end up with a complex of the form
\[ \left( \un{B}_{w_0}(0) \to B_s B_t B_s(1) \to B_s(1) \right), \]
but where we do not know the differentials. However, we can easily rederive the differentials from first principles! Both $\Hom(B_{w_0},B_s B_t B_s(1))$ and $\Hom(B_s B_t B_s(1), B_s(1))$ are free $\Z$-modules of rank $1$. Thus our differentials agree with those of \eqref{FtstBs} up to scalars $a_i \in \Z$.
\begin{equation} F_{tst} \ot B_s \cong \left(
\begin{tikzpicture}
\node (a) at (0,0) {$\un{B}_{w_0}(0)$};
\node (b) at (2.5,0) {$B_s B_t B_s(1)$};
\node (c) at (5,0) {$B_s(1)$};
\node at (2,-.5) {};
\path
	(a) edge node[above] {$a_1 \wtosts$} (b)
	(b) edge node[above] {$a_2 \redpitchforkdown$} (c);
\end{tikzpicture} \right). \end{equation}
Suppose e.g. that $a_2 = 0$. Then the complex $F_{tst} B_s$ would be decomposable in the homotopy category. However, $B_s$ is indecomposable and $F_{tst} \ot (-)$ is an invertible functor on the homotopy category, so $F_{tst} B_s$ must be indecomposable, a contradiction. Similarly, suppose that $a_2 \ne \pm 1$. Then after specialization to some finite characteristic, $F_{tst} B_s$ would be decomposable. This is again a contradiction, since $F_{tst}$ is invertible over any field. Thus $a_1 = \pm 1$ and $a_2 = \pm 1$. Using isomorphisms which scale the chain objects, we can assume $a_1 = a_2 = 1$. This method also has the added bonus of proving a uniqueness statement: \eqref{FtstBs} is the unique absolutely indecomposable complex with the same underlying chain objects, up to isomorphism.

What we wished to compute was $\HT = F_{tst} F_s$. Note that $F_s$ is the cone of a map $B_s \to \one(1)$, so $\HT$ is the cone of a chain map $F_{tst} B_s \to F_{tst}(1)$. We
have already computed $F_{tst} B_s$ and $F_{tst}(1)$ above. Instead of carefully computing what this chain map is, we take the ``lazy'' approach of computing all possible chain
maps. One can compute directly that there is a three-dimensional space (rather, a free rank 3 $\Z$-module) of chain
maps $F_{tst} B_s \to F_{tst}(1)$, of which a two-dimensional subspace is nulhomotopic. Homotopic chain maps give rise to homotopy-equivalent cones, which have isomorphic minimal complexes. A quick way to compute the
space of possible chain maps modulo homotopy is to first use the homotopies to assert that certain coefficients of the chain map are zero. Then one computes the remaining
constraints on the chain map.

For example, let us compute how chain maps from $F_{tst} B_s$ to $F_{tst}(1)$ behave in homological degree $2$, where one has a map from $B_s(1)$ to $B_s(3) \oplus B_t(3)$. The space
of maps $B_s(1) \to B_t(3)$ is one-dimensional, spanned by $\redtoblue$. The homotopy $\startdotblue \linered$ produces a nulhomotopic chain map, whose coefficient of $\redtoblue$ in
homological degree $2$ is $1$. By adding a multiple of this nulhomotopic chain map to any chain map, we can assert that the map $B_s(1) \to B_t(3)$ is zero. What remains in degree $2$ is some morphism $B_s(1) \to B_s(3)$, which must be killed by the differential $\finaldotred$ on the target to be a chain map. There is only a one-dimensional space of such
morphisms, spanned by the triangle map $\trianglered$.

Using these methods one computes that any chain map $F_{tst} B_s \to F_{tst}(1)$ is, up to homotopy, a scalar $a \in \Z$ times the chain map whose cone yields the complex in
\eqref{HTchar2}. Once again, $a = \pm 1$ is the only choice for which the cone of this map is absolutely indecomposable, which $\HT$ must be. Up to isomorphism of complexes, we can
assume $a = 1$.

Let us call the method above the \emph{uniqueness method}. Alternatively, we could have used the \emph{(explicit) Gaussian elimination method} to compute $F_{tst} \ot F_s$ directly.
There would be exactly the same three cancellation moves as in $F_{tst} \ot B_s$: one term $B_s(3)$ would cancel a summand of $B_s B_s(2)$, etcetera. However, each cancellation would
produce many more zigzag terms, involving differentials going from the $F_{tst} B_s$ half of the complex to the $F_{tst}(1)$ half of the complex. For example, the cancellation of the
summands $B_s(3)$ produced no zigzag terms in $F_{tst} B_s$ for trivial reasons, because there were no surviving terms in homological degree $3$. The analogous cancellation in $F_{tst}
F_s$ does produce a zigzag term, and keeping track of this might be inconvenient.

To emphasize the difference between these two methods, the uniqueness method seems like it has to do more work because it computes the space of all possible chain maps, but it does so
from a small complex to a large complex. The Gaussian elimination method will simplify a large complex to a small complex, while keeping track of extra zigzags to another large
complex, and this becomes quite painful. The only real advantage of the Gaussian elimination method is that it does not require first computing $F_{tst} \ot B_s$.


However, computing $F_{tst} B_s$ was not a waste of time. For example, we can now immediately compute $\HT \ot B_s$, since $\HT \cong F_t F_s F_t F_s$ and $F_s B_s \cong B_s(-1)$. Thus
$\HT B_s \cong F_{tst} B_s(-1)$ is a three-term complex. Tensoring this with $B_t$, we can now apply Gaussian elimination to cancel a copy of $B_{st}(0)$ in homological degrees $1$ and
$2$, and a copy of $B_{w_0}(0)$ in degrees $0$ and $1$. This can be done explicitly using the decompositions in \eqref{wtip} and \eqref{CstsBtip}, see \cite{C2supplement} for details.
The result is a two-term complex \[ \HT \ot B_{st} \cong \left( \un{B}_{w_0}(-2) \to B_{st}(0) \right). \] From here it is easy to compute the complex $\HT \ot (B_s B_t B_s)$ as well.

Our preferred method of computing $\HT \ot B_t$ is to first compute $F_{sts}$ and then $F_{sts} B_t$, another big calculation with Gaussian elimination like the computation of
$F_{tst} B_s$ was. All that remains at the end is a simple two-term complex $\left( \un{B}_{w_0}(-1) \to B_{tst}(0) \right)$. Because almost everything eventually is eliminated,
one need not keep careful track of most differentials (one must just ensure that one does not make an erroneous cancellation), and can use an easy uniqueness argument at the end to
determine the differential. From here it is straightforward to compute $\HT \ot B_{ts}$ and $\HT \ot B_{tst}$ as well.

Finally, we need to compute the full twist, which we accomplish by computing $\HT \ot F_w$ starting with $w = 1$ (already done), and steadily increasing the length of $w$. To begin,
$\HT \ot F_t$ is the cone of a map $\HT B_t \to \HT (1)$, from a simple two-term complex to a complicated complex. Computing all the chain maps (modulo homotopy, but there are no
homotopies in this example) is relatively easy, using the morphisms and relations given in \S\ref{sec:C2diag}, such as \eqref{useful}. For the computation of $\HT \ot F_s$,
finding all the chain maps from $\HT \ot B_s$ to $\HT(1)$ is somewhat more of an ordeal, but the alternative (careful Gaussian elimination with an extra space of differentials to
$\HT(1)$) is worse.

Now suppose we wish to compute $\HT \ot F_{ts}$. We know this is the cone of a chain map from $\HT \ot B_{ts}$ to $\HT \ot \left( B_s(1) \oplus B_t(1) \to R(2)\right)$. This latter
tensor product we already know by our computations above, because it has both $\HT \ot F_s$ and $\HT \ot F_t$ as subcomplexes. Now we compute the space of all chain maps from the
two-term complex $\HT \ot B_{ts}$ to the large complex $\HT \ot \left( B_s(1) \oplus B_t(1) \to R(2)\right)$, which is not as terrible as it seems. The complex $\HT \ot \left( B_s(1) \oplus B_t(1) \to R(2)\right)$ may be quite large, but only homological degrees between $0$ and $2$ are relevant for computing chain maps from the two-term complex $\HT \ot B_{ts}$.

Continuing, we compute $\HT \ot F_w$ for all $w \in W$, until we finally find our formula for the full twist. At each step, the result was unique up to homotopy equivalence, assuming
the indecomposability of the result over any field. The computation of $\HT \ot B_s B_t B_s$ is by far the nastiest part of this process, as the nontrivial automorphisms of $B_s B_t
B_s$ lead to an extra degree of freedom (this nastiness is avoided in characteristic zero). Finally, one additional trick was used: that $\HT \ot F_w \ot B_{w_0}$ is indecomposable,
which requires the differential from the lowest degree to have enough nonzero components to be able to cancel all extra copies of $B_{w_0}$.

Details on these computations can be found in \cite{C2supplement}.

\section{Type $C_3$} \label{sec:C3}

Let us discuss the case where $W$ has type $C_3$ and $p=2$. The simple reflections are denoted $\{s,t,u\}$, with $m_{st} = 3$ and $m_{tu} = 4$. The longest element is denoted $w_0$ as usual.

\subsection{$p$-cells and their eigenvalues}

The $p$-cells in this example can be found\footnote{Technically, \cite[p16]{JensenABC} describes the right $p$-cells. Using symmetry one can determine the left $p$-cells and then
the two-sided $p$-cells. There are a couple typos to be found on \cite[p16]{JensenABC}. The elements $21$ and $23$ should be swapped in the partition of $C_2$ into $p$-cells. The
last element in $C_7$ should read $21213212$.} in \cite[p16]{JensenABC}, and also are close at hand in the appendix, under labelling $(s, t, u) = (1, 2, 3)$. Let $\mu$ denote the (two-sided\footnote{The $0$-cell $\mu$ is the union of the right
$0$-cells denoted $C_6$, $C_9$ and $C_{11}$ on \cite[p16]{JensenABC}, and $\nu$ is the union of $C_7$, $C_{10}$, and $C_{12}$.}) $0$-cell containing $sts$, and $\nu$ the $0$-cell
containing $stsuts$. Although there are eleven elements $w \in W$ for which $b_w \ne \pb_w$, every $0$-cell aside from $\mu$ and $\nu$ is also a $p$-cell, with $\xb$ and $\cb$ and
$\Schu_L$ unchanged. Something quite fascinating occurs with $\mu$ and $\nu$.

For the computation below, the following elements of the $p$-canonical basis play a role.
\begin{subequations} \label{pcanC3}
\begin{equation} \pb_{sts} = b_{sts}. \end{equation}
\begin{equation} \pb_{stsuts} = b_{stsuts} + (v+v^{-1}) b_{sts}. \end{equation}
\begin{equation} \pb_{w_0 u} = b_{w_0 u} + b_{stsuts}. \end{equation}
\begin{equation} \pb_{w_0} = b_{w_0}. \end{equation}
\end{subequations}
Note that $\pb_{stsuts}$ is the ``first'' example of a \emph{non-perverse} $p$-canonical basis element, meaning that its coefficients in the ordinary KL basis are Laurent polynomials rather than integers.

The two $0$-cells $\mu$ and $\nu$ divide into three $p$-cells as follows.
\begin{equation} \la_{sts} = \{sts\}, \qquad \mu' = \mu \setminus \{sts\} \cup \{stsuts\}, \qquad \nu' = \nu \setminus \{stsuts\}. \end{equation}
Thus the $p$-cells do not form a partition of the $0$-cells.

We compute\footnote{The following computations were done by computer, with thanks to Joel Gibson and Geordie Williamson.} the action of the half twist. On the ordinary KL basis we have
\begin{subequations}
\begin{equation} \label{halftimesstscanon}\half_W \cdot b_{sts} = v^{-3} b_{w_0} - v^{-2} b _{w_0 u} + v^{-1}  b_{sts}, \end{equation}
\begin{equation} \half_W \cdot b_{stsuts} = v^{-6} b_{w_0} -v^{-3} b_{stsuts}, \end{equation}
from which one easily computes that
\begin{equation} \label{halftimesstspcanon} \half_W \cdot \pb_{sts} = v^{-3} \pb_{w_0} - v^{-2} \pb_{w_0 u} + v^{-2} \pb_{stsuts} - v^{-3} \pb_{sts}, \end{equation}
\begin{equation} \label{halftimesstsutspcanon} \half_W \cdot \pb_{stsuts} = (v^{-2} + v^{-4} + v^{-6}) \pb_{w_0} - (v^{-1} + v^{-3}) \pb_{w_0 u} + v^{-1} \pb_{stsuts}. \end{equation}
\end{subequations}
As a consequence we have
\begin{subequations}
\begin{equation} \half_W \cdot b_{sts} \equiv v^{-1} b_{sts} + I_{< \mu}, \end{equation}
\begin{equation} \half_W \cdot \pb_{sts} \equiv -v^{-3} \pb_{sts} + I_{< \la_{sts}}, \end{equation}
and
\begin{equation} \half_W \cdot b_{stsuts} \equiv -v^{-3} b_{stsuts} + I_{< \nu}, \end{equation}
\begin{equation} \half_W \cdot \pb_{stsuts} \equiv v^{-1} \pb_{stsuts} + I_{< \mu'}. \end{equation}
\end{subequations}
As you can see, the eigenvalues of $sts$ and $stsuts$ were swapped!

Conjecture \ref{conj:htdecat} holds, consistent with the values
\begin{equation} \label{C3eigenvalues} \pxb(\la_{sts}) = -3, \pcb(\la_{sts}) = 3, \qquad \pxb(\mu') = -1, \pcb(\mu') = 2, \qquad \pxb(\nu') = -3, \pcb(\nu') = 1. \end{equation}
Note that $\pSchu_L$ and $\Schu_L$ agree, and both are the identity map outside of $\nu'$.

\subsection{Musings on the categorification: guessing the minimal complexes}

Now we discuss how we expect things to play out in the categorification, with an attempt to understand why $sts$ and $stsuts$ swapped eigenvalues. Everything below should be considered as an educated guess, or as the ravings of a soothsayer. In characteristic zero, Conjecture \ref{conj:htcatchar0} has still not been proven in type $C_3$, but we assume it for sanity later in this discussion.

Our arguments will compare characteristic zero and characteristic $p$, by working with the Hecke category over $\Z$. By \cite{EWGr4sb, EWLocalized}, morphism spaces in the Hecke
category are free $\Z$-modules, whose graded rank is determined by the Soergel hom formula. All assertions about the size of Hom spaces stem from this fact. The only prime for
which the $p$-canonical basis differs from the KL basis in type $C_3$ is $p=2$. Thus we expect the decomposition of bimodules over $\Z$ to mimic that in characteristic $2$, with no
additional subtleties. We write $\pB_x$ for an indecomposable bimodule over $\Z$, the top summand of its Bott-Samelson bimodule, which categorifies $\pb_x$.

The fact that $\pb_{w_0 u} = b_{w_0 u} + b_{stsuts}$ should be interpreted as follows.  From the Soergel hom formula there are maps
\[ q \co \pB_{w_0 u} \to \pB_{stsuts}, \qquad j \co \pB_{stsuts} \to \pB_{w_0 u}\]
which span their respective Hom spaces in degree zero.
If $e := j \circ q$ then $e^2 = 2e$. After inverting $2$, $\frac{e}{2}$ is the idempotent projecting to the common summand $B_{stsuts}$. The complementary summand in $\pB_{w_0 u}$ is $B_{w_0 u}$. Meanwhile, in characteristic $2$, $e$ is nilpotent and $\pB_{w_0 u}$ is indecomposable, with ``the extra copy of $B_{stsuts}$ still stuck on.'' The maps $q$ and $j$ are now in the Jacobson radical of the category.

In similar fashion, $\pb_{stsuts} = b_{stsuts} + (v+v^{-1}) b_{sts}$ indicates the existence of maps $p_{\pm 1} \co \pB_{stsuts} \to \pB_{sts}$ of degree $\pm 1$, and maps $i_{\pm
1}$ in the opposite direction, which project to the common summands $B_{sts}(\pm 1)$ up to multiplication by $2$.

In characteristic not equal to $2$ we expect the following minimal complex:
\begin{subequations} \label{HTstsC3eqns}
\begin{equation} \label{HTstsC3char0} \HT \ot B_{sts} \cong \left( \un{B}_{w_0}(-3) \to B_{w_0 u}(-2) \to B_{sts}(-1) \right). \end{equation}
In characteristic $2$ or over $\Z$ we expect
\begin{equation} \label{HTstsC3char2} \HT \ot \pB_{sts} \cong \left( \un{\pB}_{w_0}(-3) \to \pB_{w_0 u}(-2) \to \pB_{stsuts}(-2) \to \pB_{sts}(-3) \right).\end{equation}
\end{subequations}
We explain these expectations below, but first we discuss the differentials. All differentials live in a Hom space which has rank $1$ over $\Z$. Choosing a generator of each Hom space as a $\Z$-module, we obtain the complex above. Note that $d^2 = 0$ simply because it lives in a zero Hom space (e.g. there are no morphisms of degree $2$ from $B_{w_0}$ to $B_{sts}$). The second differential in \eqref{HTstsC3char2} is the map $q$ discussed above, and the third differential is $p_{-1}$. After specialization to characteristic $2$, the
differentials live in the Jacobson radical.

That $\HT \ot B_{sts}$ should be built from $B_{w_0}(-3)$ and $B_{sts}(-1)$ in even homological degree, and $B_{w_0 u}(-2)$ in odd homological degree, is expected from
\eqref{halftimesstscanon}. One might worry that the minimal complex of $\HT \ot B_{sts}$ might have additional chain objects which are invisible in the Grothendieck group. For
example, some $B_w(k)$ might appear in two different homological degrees, one even and one odd. We call this a \emph{cancellation phenomenon}. The cancellation phenomenon
\emph{does} occur in finite characteristic, see \eqref{FTchar2}, and is a legitimate worry. In characteristic zero it was proven \cite{EWHodge} that there is no cancellation in Rouquier
complexes which are positive lifts (like $\HT$), because they are perverse. No analogous result has been proven for complexes like $\HT \ot B_w$, but it is still expected that the
chain objects in the minimal complex have no cancellation in characteristic zero.

\begin{lemma} In the absence of any cancellation phenomenon, $\HT \ot B_{sts}$ must have the form given in \eqref{HTstsC3char0}, with the differentials stated (up to isomorphism).
\end{lemma}

\begin{proof} First we claim that the choice of differentials in a complex of the form \eqref{HTstsC3char0} is unique up to isomorphism, given that $\HT \ot B_{sts}$ is absolutely
indecomposable. In previous chapters we explained the reasoning, which we briefly recall. The complex $\HT$ is invertible so it preserves indecomposable objects in the homotopy
category. This remains true after any specialization. If a differential did not generate its Hom space, then it becomes zero in some specialization, and the complex splits
nontrivially. There is only one generator of a free rank $1$ module up to sign, and we can rescale the chain objects by signs to modify the generators at will.

The only non-zero morphism spaces between the objects $B_{w_0}(-3)$ and $B_{w_0 u}(-2)$ and $B_{sts}(-1)$ are those indicated by the arrows in \eqref{HTstsC3char0}. Had one tried
to build a complex with these objects in a different order, then some differential is zero. Thus \eqref{HTstsC3char0} is the unique way to glue these three objects together into an
indecomposable complex, up to an overall homological shift. The overall homological shift is determined by considering the homology of this complex in the category of
$(R,R)$-bimodules. The homology of $\HT$ is known to be concentrated in degree zero, and thus the same holds for $\HT \ot B_{sts}$. \end{proof}

With \eqref{HTstsC3char0} as a given, let us try to find a complex well-defined over $\Z$ which is homotopy equivalent to \eqref{HTstsC3char0} after inverting $2$. In homological
degree $1$ we must have $\pB_{w_0 u}(-2)$ instead of $B_{w_0 u}(-2)$, because the latter is not defined over $\Z$. This contributes an extra summand $B_{stsuts}(-2)$ after
inverting $2$, which must be cancelled by some copy of $B_{stsuts}(-2)$ in homological degree $2$ or $0$. Over $\Z$, this cancelling copy can come from either $\pB_{stsuts}(-2)$ or
$\pB_{w_0 u}(-2)$. Having another copy of $\pB_{w_0 u}(-2)$ is unlikely: in characteristic zero it would produce another copy of $B_{w_0 u}(-2)$, needing to be cancelled by yet
another copy of $\pB_{w_0 u}(-2)$, which just repeats the same problem again. So we expect some copy of $\pB_{stsuts}(-2)$ in homological degree $0$ or $2$. We do not expect it to
appear in homological degree zero, as this would contribute additional copies of $B_{sts}$ in degree zero which would need to be cancelled out. The cases we found unlikely above
all lead to a cancellation phenomenon. These arguments are slightly optimistic, and we discuss this optimism in several remarks below.

A copy of $\pB_{stsuts}(-2)$ in homological degree $2$ would account already for the term $B_{sts}(-1)$ in homological degree $2$, but would also contribute a copy of $B_{sts}(-3)$. Optimism suggests that it is cancelled in degree $3$ rather than degree $1$, yielding our guess \eqref{HTstsC3char2}. Note that the differentials are unique up to isomorphism, given the absolute indecomposability of $\HT \ot \pB_{sts}$.

Here are some additional comments on this line of reasoning.

\begin{rmk} In the examples we know where the cancellation phenomenon occurs, it was not possible to construct an indecomposable complex without these extra terms; one can not
arrange for $d^2 = 0$. The complex in \eqref{HTstsC3char2} is well-defined without the need for extra terms. \end{rmk}

\begin{rmk} In Remarks \ref{rmk:indegree2c} and \ref{rmk:indegree2credux} we discuss properties of the half and full twist which follow from our conjectures. Suppose that the
minimal complex of $\HT$ does not contain $\pB_w$ in any homological degree less than $\pcb(w)$. For example, only $\pB_{w_0}$ appears in homological degree zero, and only terms in
the bottom two cells $\la_0$ and $\nu'$ can appear in homological degree $1$. Since $p$-cells correspond to monoidal ideals, the same must be true in $\HT \ot \pB_{sts}$. This
rules out the appearance of $\pB_{stsuts}$ in degree zero, or $\pB_{sts}$ in degree $1$, etcetera. \end{rmk}

Now we consider the other element $stsuts$. In characteristic zero we expect
\begin{subequations}
\begin{equation} \label{HTstsutsC3char0} \HT \ot B_{stsuts} \cong \left( \un{B}_{w_0}(-6) \to B_{stsuts}(-3) \right). \end{equation}
Remembering that $\pB_{stsuts}$ splits in characteristic zero as $B_{stsuts} \oplus B_{sts}(-1) \oplus B_{sts}(+1)$, we see that $\HT \ot \pB_{stsuts}$ should have the same size as \eqref{HTstsutsC3char0} plus two copies of \eqref{HTstsC3char0}. Using quantum numbers for grading shifts to save space, we expect
\begin{equation} \label{HTstsutsC3char2} \HT \ot \pB_{stsuts} \cong \left( [3]\un{\pB}_{w_0}(-4) \to [2]\pB_{w_0 u}(-2) \to \pB_{stsuts}(-1) \right),\end{equation}
\end{subequations}
matching \eqref{halftimesstsutspcanon}.
Note that $\pB_{stsuts}(-1)$ contains the $[2]$ copies of $B_{sts}(-1)$ expected to appear in homological degree $2$. No copies of $\pB_{sts}$ may appear, because the monoidal ideal generated by $\pB_{stsuts}$ does not contain $\pB_{sts}$ (which is in a higher $p$-cell). The justifications for guessing \eqref{HTstsutsC3char0} and \eqref{HTstsutsC3char2} are similar to the arguments above. One difference is that there is no a priori reason why $d^2 = 0$ in \eqref{HTstsutsC3char2}, and it would not be surprising if there were additional cancelling copies of $\pB_{w_0}(-4)$ in degrees $0$ and $1$, needed to force $d^2 = 0$. This would not affect the discussion below.

\subsection{Musings on the categorification: eigenvalue swap}

So why did $\pb_{stsuts}$ inherit the eigenvalue of $b_{sts}$? Let us pose the question more generally. Suppose that $w, x \in W$ (think $x = sts$ and $w = stsuts$) satisfy the property that $w$ is in a strictly lower $0$-cell and a lower $p$-cell than $x$, and that
\begin{equation} \label{wxpair} \pb_w = b_w + f b_x \end{equation}
for some nonzero $f \in \N[v,v^{-1}]$. We claim that $\pB_w$ should inherit the eigenvalue of $B_x$. For sake of simplicity we assume in the argument below that $\Schu_L$ is the identity operator.

Assuming our conjectures hold, $\HT \ot B_w$ is concentrated in homological degrees $\le \cb(w)$. However, $\pB_w$ has an extra $f$ copies of $B_x$, and $\HT \ot B_x$ has terms appearing in homological degree $\cb(x)$ as well. The statistic $\cb$ is monotone in the cell order, so $\cb(w) < \cb(x)$. Thus we expect that, after passage to characteristic zero, $\HT \ot \pB_w$ should have $v^{\xb(x)} f$ copies of $B_x$ in degree $\cb(x)$. However, $\HT \ot \pB_w$ can only have terms in cells less than that of $w$. Thus in the lift over $\Z$, these copies of $B_x$ can not come from copies of $\pB_x$. Though other options are theoretically possible, the simplest solution to obtaining $v^{\xb(x)} f$ copies of $B_x$ is from one copy of $v^{\xb(x)} \pB_w$. Thus, we expect that $\HT \ot \pB_w$ has one copy of $\pB_w$ in homological degree $\cb(x)$ and with grading shift $\xb(x)$, giving it the eigenvalue of $B_x$.

\begin{rmk} This heuristic argument did not depend on the value of $f$. In particular, it did not matter that $\pb_{stsuts}$ was not perverse. \end{rmk}

\begin{rmk} There are no such pairs $(w,x)$ in any dihedral type. \end{rmk}

Now we ask: should $\pB_x$ inherit the eigenvalue of $B_w$? The answer is no: in our example, $\pxb(sts) = \xb(stsuts)$ but $\pcb(sts) \ne \cb(stsuts)$, they only agree in parity. The better question is: should we expect $\pxb(x) = \xb(w)$?  We do not have a convincing reason this equality should hold. However, it is easier to explain why, in our example, the shift $\xb(sts) - \pxb(sts) = 2$ occurred. This seems to be related to the difference between the two summands $(v+v^{-1}) b_{sts}$ in the non-perverse basis element $\pb_{stsuts}$.

Suppose one has a pair $(w,x)$ as in \eqref{wxpair} and that the valuation of $f$ is $k$, i.e. $f = v^k + \ldots + v^{-k}$ for some $k > 0$. We do not require the coefficient of $v^k$
is one, but we will assume this for simplicity. Let $q_{-k}$ be the map\footnote{This matches $p^{-1}$ in the previous section, but $p$ was changed to $q$ because we also invoke the name of the prime.} of degree $-k$ which projects from $\pB_w$ to $B_x$ up to scalar (it splits after inverting $p$).

Assuming our conjectures hold, $\HT \ot \pB_x$ should have, in its final homological degree $\pcb(x)$, a copy of $\pB_x(\pxb(x))$. This is also the unique copy of $\pB_x$ in the
minimal complex of $\HT \ot \pB_x$. Suppose one posits that $q_{-k}$ appears in the differential which maps to this final degree (discussion below). The source of $q_{-k}$ would be a
copy of $\pB_w(\pxb(x)+k)$ in degree $\pcb(x)-1$. After inverting $p$, $q_{-k}$ splits, but $\pB_w(\pxb(x)+k)$ has another summand $B_x(\pxb(x)+2k)$. Were this term to survive to the
minimal complex after $p$ is inverted (discussion below), then since $\HT \ot B_x$ has a unique copy of $B_x$ in its minimal complex, it must be this one. This situation leads to
$\cb(x) = \pcb(x)-1$ and $\xb(x) = \pxb(x)+2k$.

The same heuristic argument can be used to predict in type $C_2$ that $\cb(s) = \pcb(s) -1$ and $\xb(s) = \pxb(s)$. This time, we expect a differential from $\pB_{sts}(\xb(sts))$
in degree $\pcb(s)-1$ to $\pB_s(\pxb(sts))$ in degree $\pcb(s)$. The other summand $B_{sts}(\pxb(sts))$ inside $\pB_{sts}(\pxb(sts))$ survives, and this is the Sch\"{u}tzenberger
dual of $s$.

Why should $q_{-k}$, or other maps which fail to split over $\Z$, appear as differentials in $\HT \ot \pB_x$? Before applying any Gaussian elimination, the tensor product of complexes
$\HT \ot \pB_x$ has many extra terms. In order for Gaussian elimination to remove these terms in characteristic zero, many projection maps have to appear as differentials, such as
$q_{-k}$. These happen not to split over $\Z$, so they survive as differentials.

Why should $B_x(\pxb(x)+2k)$ in degree $\pcb(x)-1$ survive to the minimal complex in $\HT \ot \pB_x$ after $p$ is inverted? It need not.

\begin{ex} In type $C_4$, number the simple reflections so that $\{s_1, s_2, s_3\}$ generate a parabolic subgroup of type $A_3$. Let $x = s_1 s_2 s_1 s_3 s_2 s_1$ be the longest
element of this parabolic subgroup, and let $w = x s_4 s_3 s_2 s_1$. We have $\pb_w = b_w + (v^2 + 1 + v^{-2}) b_x$ and $\pb_x = b_x$. We have $\pxb(x) = \xb(w) = -8$ and $\pxb(w) =
\xb(x) = -4$, another eigenvalue swap. Again, $\xb(x) - \pxb(x) = 4$ is the difference between the largest and smallest shifts of $b_x$ appearing in $\pb_w$. We expect $\pcb(x) =
\cb(x) + 2$, so the heuristic explanation above will not suffice. \end{ex}

In $C_4$, what might have happened to the copy of $B_x(\pxb(x)+2k)$ in degree $\pcb(x)-1$? No more copies of $\pB_x$ appear to Gaussian eliminate against it. One option is that $\pB_x$
has summands of the form $B_y$, and thus $\HT \ot \pB_x$ has summands of the form $\HT \ot B_y$; when $x$ is in a lower cell than $y$, this can easily explain copies of $B_x$ in degree
$\pcb(x)-1$. Another option is that additional copies of $\pB_w$ or some other $\pB_{w'}$ contain summands $B_x$. In the $C_3$ example it is hard to produce such a situation without a
massive and unlikely cancellation effect.

In truth, the situation is still a mystery, and additional exploration is needed.

\section{Appendix: $p$-cells and their eigenvalues} \label{appendix}

In this appendix we present explicitly the differences visible in the cell structure between the canonical and $p$-canonical bases for the Cartan types of rank at most 4. The
computations have been performed up to rank 6, but we only include a summary of ranks 5 and 6. These calculations are possible due to work of Geordie Williamson and the second author,
who have written an algorithm to calculate $p$-canonical bases based on the diagrammatic Hecke category. Tables of $p$-canonical bases for all the types described below, as well as the
Magma package \texttt{IHecke} used to work with the Hecke algebra and calculate the cell decompositions, are available on GitHub \cite{HeckeCode}. The algorithm which calculates the
bases themselves is described in \cite{GJW}.

Let us call a prime $p$ \textit{interesting} for a Cartan type if the $p$-canonical basis differs from the canonical basis.
Up to rank four, the interesting types are $C_2, C_3, C_4$ with $p = 2$, $B_2, B_3, B_4$ with $p = 2$, $D_4$ with $p = 2$, $F_4$ with $p = 2, 3$, and $G_2$ with $p = 2, 3$.
In all of these interesting types, we have verified the following statement.
Let $\lambda$ denote a left $p$-cell, and $w \in \lambda$ an element of that cell. Then
\begin{equation}
	h_{w_0} \cdot {^p b_w} \equiv (-1)^{\bf{c}(w)} v^{\bf{x}(w)} \cdot {^p b}_{\operatorname{Schu}_L^p(w)} \pmod {I_{< \lambda}}
\end{equation}
for some functions $\mathbf{c} \colon W \to \{0, 1\}$ and $\mathbf{x} \colon W \to \mathbb{Z}$ which are constant on left cells, and an involution $\operatorname{Schu}_L^p \colon W \to W$.
Magma code which verifies this conjecture is available as one of the examples in the \texttt{IHecke} package mentioned above.
Note that since the Hecke algebra only sees the parity of the actual function $\mathbf{c} \colon W \to \mathbb{Z}$, throughout this appendix we refer to the eigenvalue pairs $(\mathbf{x}, \pm)$ rather than $(\mathbf{x}, \mathbf{c})$, and really all we can verify is that the function $(-1)^{\mathbf{c}}$ is constant on left cells.

The following table collects some high-level statistics about the $p$-cells in the types listed above.
The rows of this table are (in order): the number of left cells, two-sided cells, unique $(\mathbf{x}, \pm)$-eigenvalues, $\operatorname{Schu}_L^p$-fixed points, and $\operatorname{Schu}_L^p$-moving points.

\begin{table}[h!]
\label{table:type-statistics}
\centering
\makebox[\textwidth][c]{\begin{tabular}{r | r r | r r r | r r r | r r r | r r | r r r}
Type & $C_2$ & $C_2$ & $G_2$ & $G_2$ & $G_2$ & $C_3$ & $C_3$ & $B_3$ & $C_4$ & $C_4$ & $B_4$ & $D_4$ & $D_4$ & $F_4$ & $F_4$ & $F_4$ \\
$p$ & 0 & 2 & 0 & 2 & 3 & 0 & 2 & 2 & 0 & 2 & 2 & 0 & 2 & 0 & 2 & 3 \\
\hline
\# Left & 4 & 5 & 4 & 6 & 5 & 14 & 17 & 17 & 50 & 63 & 63 & 36 & 38 & 72 & 106 & 78 \\
\# Two-sided & 3 & 4 & 3 & 4 & 4 & 6 & 8 & 8 & 10 & 15 & 15 & 11 & 12 & 11 & 17 & 12 \\
\# $(\mathbf{x}, \pm)$ & 3 & 4 & 3 & 4 & 3 & 6 & 7 & 7 & 9 & 12 & 12 & 7 & 8 & 9 & 12 & 9 \\
\# $\operatorname{Fix} \operatorname{Schu}_L^p$ & 4 & 6 & 4 & 12 & 4 & 32 & 40 & 40 & 280 & 332 & 332 & 112 & 120 & 544 & 1136 & 544 \\
\# $\operatorname{Unfix} \operatorname{Schu}_L^p$ & 4 & 2 & 8 & 0 & 8 & 16 & 8 & 8 & 104 & 52 & 52 & 80 & 72 & 608 & 16 & 608 \\
\end{tabular}
}
\caption{Statistics on $0$-cells vs $p$-cells for the Cartan types of rank at most 4.}
\end{table}

Similar tables for ranks 5 and 6 appear in \S\ref{subsec:5and6}. Notice that the statistics listed in the tables for types $C_n$ and $B_n$ are identical for $n \le 6$, but the cells
themselves are certainly not. It is unclear whether this is a general phenomenon, or just a low-rank coincidence.

Our labeling of the simple reflections follows the Dynkin diagram numbering which is built-in to MAGMA, and reproduced below.
\begin{table}[h!]
\centering
\begin{tabular}{r l r l r l r l r l}
$B_n$ & \dynkin[edge/.style={-},indefinite edge/.style={-,densely dotted}, edge-length=0.75cm,root-radius=0.08cm,labels={1,n-1,n}]B{*...**} &
$C_n$ & \dynkin[edge/.style={-},indefinite edge/.style={-,densely dotted},edge-length=0.75cm,root-radius=0.08cm,labels={1,n-1,n}]C{*...**} &
$D_4$ & \dynkin[edge/.style={-},edge-length=0.75cm,root-radius=0.08cm,labels={1,2,3,4}]D4 &
$G_2$ & \dynkin[backwards,edge/.style={-},edge-length=0.75cm,root-radius=0.08cm,labels={2,1}]G2 &
$F_4$ & \dynkin[edge/.style={-},edge-length=0.75cm,root-radius=0.08cm,labels={1,2,3,4}]F4
\end{tabular}
\caption{Dynkin diagram numbering convention used throughout the appendix.}
\end{table}
We use the standard Cartan matrix\footnote{For another programmer wanting to double-check conventions: look at the coefficients involved in the change-of-basis from the $2$-canonical bases of $C_3$ and $B_3$ to the canonical basis. All coefficients should be integers, except for one $(v^{-1} + v)$ in the table for $C_3$.} for this Dynkin diagram when defining the realization from which the diagrammatic Hecke category is constructed (see \cite[\S 3.1]{EWGr4sb}). Namely, if a double edge points from $i$ to $j$, then $\partial_i(\alpha_j) = -1$ and $\partial_j(\alpha_i) = -2$.

{\bf Acknowledgments} The author of this appendix learned everything he knows about calculating $p$-canonical bases from Thorge Jensen, and would like to thank him.


\subsection{Rank 2} \label{app:rank2}


The only interesting primes in rank $2$ are $p=2$ for $C_2$, and $p = 2, 3$ for $G_2$. These groups are small enough that we may explicitly draw the elements of the cells directly on to a Hasse diagram. We group elements by parentheses when they are related by the Sch{\"u}tzenberger involution. We also use the $y$-coordinates to label the $(\mathbf{x}, \pm)$-eigenvalues on each cell.

\begin{figure}[h]
\label{fig:c2-left-cells}
\centering
\begin{tikzpicture}[every node/.style={scale=0.85}]
\node [align=center] (c2p02) at (1.10, 1.80) {$(2, 212)$ \\ $12$};
\node [align=center] (c2p0id) at (0.00, 0.00) {$id$};
\node [align=center] (c2p01) at (-1.10, 1.80) {$(1, 121)$ \\ $21$};
\node [align=center] (c2p01212) at (0.00, 2.70) {$1212$};
\draw [-] (c2p01) -- (c2p01212);
\draw [-] (c2p02) -- (c2p01212);
\draw [-] (c2p0id) -- (c2p02);
\draw [-] (c2p0id) -- (c2p01);
\node [align=center] (c2p0two1) at (1.98, 1.80) {$\bullet$};
\node [align=center] (c2p0two1212) at (1.98, 2.70) {$\bullet$};
\node [align=center] (c2p0twoid) at (1.98, 0.00) {$\bullet$};
\draw [-] (c2p0two1) -- (c2p0two1212);
\draw [-] (c2p0twoid) -- (c2p0two1);
\node [align=center] (c2p2id) at (5.50, 0.00) {$id$};
\node [align=center] (c2p21) at (4.40, 0.90) {$1$};
\node [align=center] (c2p22) at (6.60, 1.80) {$(2, 212)$ \\ $12$};
\node [align=center] (c2p21212) at (5.50, 2.70) {$1212$};
\node [align=center] (c2p221) at (4.40, 1.80) {$21$ \\ $121$};
\draw [-] (c2p221) -- (c2p21212);
\draw [-] (c2p21) -- (c2p221);
\draw [-] (c2p22) -- (c2p21212);
\draw [-] (c2p2id) -- (c2p22);
\draw [-] (c2p2id) -- (c2p21);
\node [align=center] (c2p2two2) at (7.48, 1.80) {$\bullet$};
\node [align=center] (c2p2two1212) at (7.48, 2.70) {$\bullet$};
\node [align=center] (c2p2two1) at (7.48, 0.90) {$\bullet$};
\node [align=center] (c2p2twoid) at (7.48, 0.00) {$\bullet$};
\draw [-] (c2p2two2) -- (c2p2two1212);
\draw [-] (c2p2two1) -- (c2p2two2);
\draw [-] (c2p2twoid) -- (c2p2two1);
\node at (-3.0250000000000004, 0.0) {$(4, +)$};
\node at (-3.0250000000000004, 1.8) {$(0, -)$};
\node at (-3.0250000000000004, 2.7) {$(-4, +)$};
\node at (-3.0250000000000004, 0.9) {$(0, +)$};
\node at (-3.0250000000000004, 3.15) {$(\mathbf{x}, (-1)^{\mathbf{c}})$};
\node at (0.0, 3.15) {$C_2$, $p = 0$};
\node at (5.5, 3.15) {$C_2$, $p = 2$};
\end{tikzpicture}
\caption{Left and two-sided $0$-cells and $2$-cells in type $C_2$.}
\end{figure}
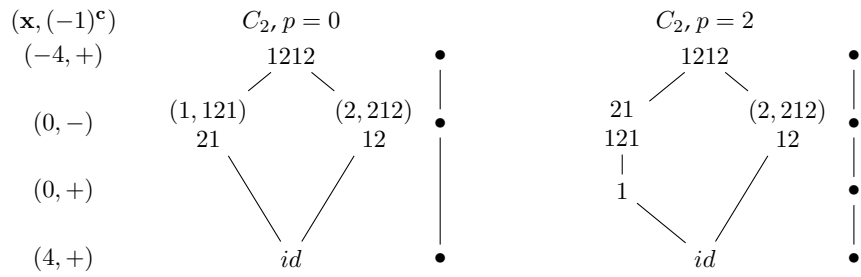

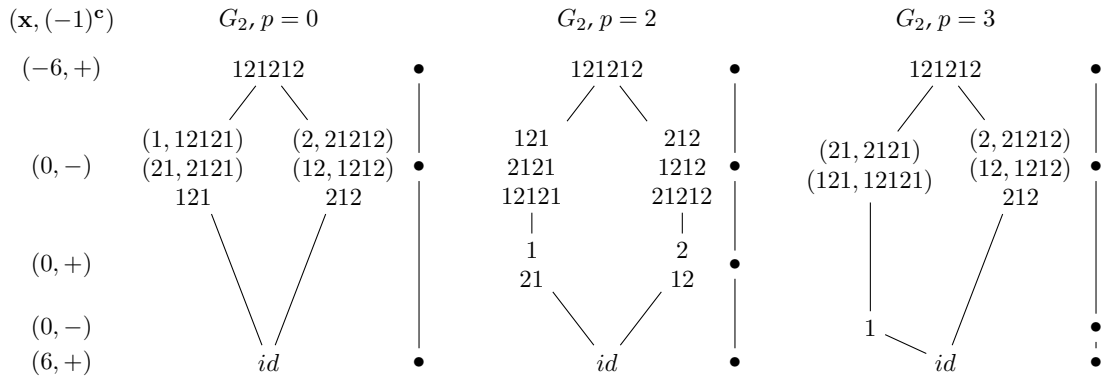
\begin{figure}
\label{fig:g2-left-cells}
\centering
\begin{tikzpicture}[every node/.style={scale=0.85}]
\node [align=center] (g2p01) at (-1.00, 3.90) {$(1, 12121)$ \\ $(21, 2121)$ \\ $121$};
\node [align=center] (g2p0id) at (0.00, 1.30) {$id$};
\node [align=center] (g2p0121212) at (0.00, 5.20) {$121212$};
\node [align=center] (g2p02) at (1.00, 3.90) {$(2, 21212)$ \\ $(12, 1212)$ \\ $212$};
\draw [-] (g2p01) -- (g2p0121212);
\draw [-] (g2p02) -- (g2p0121212);
\draw [-] (g2p0id) -- (g2p01);
\draw [-] (g2p0id) -- (g2p02);
\node [align=center] (g2p0two1) at (2.00, 3.90) {$\bullet$};
\node [align=center] (g2p0two121212) at (2.00, 5.20) {$\bullet$};
\node [align=center] (g2p0twoid) at (2.00, 1.30) {$\bullet$};
\draw [-] (g2p0two1) -- (g2p0two121212);
\draw [-] (g2p0twoid) -- (g2p0two1);
\node [align=center] (g2p2212) at (5.50, 3.90) {$212$ \\ $1212$ \\ $21212$};
\node [align=center] (g2p21) at (3.50, 2.60) {$1$ \\ $21$};
\node [align=center] (g2p2id) at (4.50, 1.30) {$id$};
\node [align=center] (g2p22) at (5.50, 2.60) {$2$ \\ $12$};
\node [align=center] (g2p2121) at (3.50, 3.90) {$121$ \\ $2121$ \\ $12121$};
\node [align=center] (g2p2121212) at (4.50, 5.20) {$121212$};
\draw [-] (g2p2121) -- (g2p2121212);
\draw [-] (g2p21) -- (g2p2121);
\draw [-] (g2p2212) -- (g2p2121212);
\draw [-] (g2p22) -- (g2p2212);
\draw [-] (g2p2id) -- (g2p21);
\draw [-] (g2p2id) -- (g2p22);
\node [align=center] (g2p2two121) at (6.20, 3.90) {$\bullet$};
\node [align=center] (g2p2two121212) at (6.20, 5.20) {$\bullet$};
\node [align=center] (g2p2two1) at (6.20, 2.60) {$\bullet$};
\node [align=center] (g2p2twoid) at (6.20, 1.30) {$\bullet$};
\draw [-] (g2p2two121) -- (g2p2two121212);
\draw [-] (g2p2two1) -- (g2p2two121);
\draw [-] (g2p2twoid) -- (g2p2two1);
\node [align=center] (g2p3id) at (9.00, 1.30) {$id$};
\node [align=center] (g2p321) at (8.00, 3.90) {$(21, 2121)$ \\ $(121, 12121)$};
\node [align=center] (g2p31) at (8.00, 1.76) {$1$};
\node [align=center] (g2p3121212) at (9.00, 5.20) {$121212$};
\node [align=center] (g2p32) at (10.00, 3.90) {$(2, 21212)$ \\ $(12, 1212)$ \\ $212$};
\draw [-] (g2p321) -- (g2p3121212);
\draw [-] (g2p31) -- (g2p321);
\draw [-] (g2p32) -- (g2p3121212);
\draw [-] (g2p3id) -- (g2p32);
\draw [-] (g2p3id) -- (g2p31);
\node [align=center] (g2p3two2) at (11.00, 3.90) {$\bullet$};
\node [align=center] (g2p3two121212) at (11.00, 5.20) {$\bullet$};
\node [align=center] (g2p3two1) at (11.00, 1.76) {$\bullet$};
\node [align=center] (g2p3twoid) at (11.00, 1.30) {$\bullet$};
\draw [-] (g2p3two2) -- (g2p3two121212);
\draw [-] (g2p3two1) -- (g2p3two2);
\draw [-] (g2p3twoid) -- (g2p3two1);
\node at (-2.75, 1.3) {$(6, +)$};
\node at (-2.75, 3.9000000000000004) {$(0, -)$};
\node at (-2.75, 5.2) {$(-6, +)$};
\node at (-2.75, 2.6) {$(0, +)$};
\node at (-2.75, 1.7550000000000001) {$(0, -)$};
\node at (-2.75, 5.8500000000000005) {$(\mathbf{x}, (-1)^{\mathbf{c}})$};
\node at (0, 5.8500000000000005) {$G_2$, $p = 0$};
\node at (4.5, 5.8500000000000005) {$G_2$, $p = 2$};
\node at (9, 5.8500000000000005) {$G_2$, $p = 3$};
\end{tikzpicture}
\caption{Left and two-sided $0$, $2$, and $3$-cells in type $C_2$.}
\end{figure}

\begin{landscape}
\subsection{Rank 3}

The only interesting types in rank 3 are $B_3$ and $C_3$, each with $p = 2$ an interesting prime. Since their underlying Coxeter systems are isomorphic, their canonical bases and hence
$0$-cells are identical. Below we show the Hasse diagrams for the left $0$-cells and $2$-cells for each type.

Each left $0$-cell is written as $[w]$, where $w$ is the shortest element of the cell which is least lexicographically.
Left $2$-cells are written as $[w]_{2, C}$ or $[w]_{2, B}$ when they differ (as a set) from any of the $0$-cells, otherwise the labelling for the $0$-cell is re-used.
Cells on which $\operatorname{Schu}_L^p$ act nontrivially are marked with an asterisk.

\begin{figure}[h!]
\centering
\hspace{-24pt}\begin{tikzpicture}[scale=0.99,every node/.style={scale=0.85}]
\node [align=center] (c3p012132) at (0.00, 3.90) {$[12132]$};
\node [align=center] (c3p0132) at (0.00, 2.60) {$[132]$};
\node [align=center] (c3p023213) at (1.93, 5.85) {$[23213]^*$};
\node [align=center] (c3p013) at (1.10, 2.60) {$[13]$};
\node [align=center] (c3p0121321323) at (0.00, 7.15) {$[121321323]$};
\node [align=center] (c3p02) at (0.00, 1.30) {$[2]^*$};
\node [align=center] (c3p0id) at (0.00, -0.65) {$[id]$};
\node [align=center] (c3p01323) at (2.20, 2.60) {$[1323]$};
\node [align=center] (c3p0121) at (-1.65, 3.90) {$[121]$};
\node [align=center] (c3p01213) at (2.75, 3.90) {$[1213]$};
\node [align=center] (c3p01) at (-1.65, 1.30) {$[1]^*$};
\node [align=center] (c3p02323) at (0.55, 5.85) {$[2323]^*$};
\node [align=center] (c3p0121321) at (-1.10, 5.85) {$[121321]$};
\node [align=center] (c3p03) at (1.65, 1.30) {$[3]$};
\draw [-] (c3p0121321) -- (c3p0121321323);
\draw [-] (c3p0121) -- (c3p0121321);
\draw [-] (c3p023213) -- (c3p0121321323);
\draw [-] (c3p01213) -- (c3p023213);
\draw [-] (c3p013) -- (c3p01213);
\draw [-] (c3p01) -- (c3p0121);
\draw [-] (c3p01) -- (c3p013);
\draw [-] (c3p02323) -- (c3p0121321323);
\draw [-] (c3p012132) -- (c3p02323);
\draw [-] (c3p012132) -- (c3p0121321);
\draw [-] (c3p0132) -- (c3p0121);
\draw [-] (c3p0132) -- (c3p012132);
\draw [-] (c3p02) -- (c3p0132);
\draw [-] (c3p01323) -- (c3p02323);
\draw [-] (c3p01323) -- (c3p01213);
\draw [-] (c3p03) -- (c3p01323);
\draw [-] (c3p03) -- (c3p013);
\draw [-] (c3p0id) -- (c3p02);
\draw [-] (c3p0id) -- (c3p01);
\draw [-] (c3p0id) -- (c3p03);
\node [align=center] (c3p0two2323) at (3.74, 5.85) {$\bullet$};
\node [align=center] (c3p0two121321323) at (3.74, 7.15) {$\bullet$};
\node [align=center] (c3p0two121) at (3.74, 3.90) {$\bullet$};
\node [align=center] (c3p0two13) at (3.74, 2.60) {$\bullet$};
\node [align=center] (c3p0two1) at (3.74, 1.30) {$\bullet$};
\node [align=center] (c3p0twoid) at (3.74, -0.65) {$\bullet$};
\draw [-] (c3p0two2323) -- (c3p0two121321323);
\draw [-] (c3p0two121) -- (c3p0two2323);
\draw [-] (c3p0two13) -- (c3p0two121);
\draw [-] (c3p0two1) -- (c3p0two13);
\draw [-] (c3p0twoid) -- (c3p0two1);
\node [align=center] (c3p21321) at (5.50, 3.90) {$[1321]_{2,C}$};
\node [align=center] (c3p232) at (7.70, 1.30) {$[32]_{2,C}$};
\node [align=center] (c3p212132) at (7.70, 3.90) {$[12132]$};
\node [align=center] (c3p2132) at (7.70, 2.60) {$[132]$};
\node [align=center] (c3p2121) at (6.60, 3.25) {$[121]_{2,C}$};
\node [align=center] (c3p223213) at (9.62, 5.85) {$[23213]^*$};
\node [align=center] (c3p213) at (8.80, 2.60) {$[13]$};
\node [align=center] (c3p2121321323) at (7.70, 7.15) {$[121321323]$};
\node [align=center] (c3p21) at (5.50, 0.33) {$[1]_{2,C}$};
\node [align=center] (c3p2id) at (7.70, -0.65) {$[id]$};
\node [align=center] (c3p2321) at (5.50, 1.30) {$[321]_{2,C}$};
\node [align=center] (c3p21323) at (9.90, 2.60) {$[1323]$};
\node [align=center] (c3p21213) at (10.45, 3.90) {$[1213]$};
\node [align=center] (c3p22) at (7.70, 0.33) {$[2]_{2,C}$};
\node [align=center] (c3p22323) at (8.25, 5.85) {$[2323]^*$};
\node [align=center] (c3p2232132) at (6.60, 5.85) {$[232132]_{2,C}$};
\node [align=center] (c3p23) at (9.35, 1.30) {$[3]$};
\draw [-] (c3p2232132) -- (c3p2121321323);
\draw [-] (c3p21321) -- (c3p2232132);
\draw [-] (c3p223213) -- (c3p2121321323);
\draw [-] (c3p21213) -- (c3p223213);
\draw [-] (c3p213) -- (c3p21213);
\draw [-] (c3p2321) -- (c3p213);
\draw [-] (c3p2321) -- (c3p21321);
\draw [-] (c3p2121) -- (c3p21321);
\draw [-] (c3p21) -- (c3p2321);
\draw [-] (c3p21) -- (c3p2121);
\draw [-] (c3p22323) -- (c3p2121321323);
\draw [-] (c3p212132) -- (c3p22323);
\draw [-] (c3p212132) -- (c3p2232132);
\draw [-] (c3p2132) -- (c3p212132);
\draw [-] (c3p232) -- (c3p2132);
\draw [-] (c3p22) -- (c3p232);
\draw [-] (c3p22) -- (c3p2121);
\draw [-] (c3p21323) -- (c3p22323);
\draw [-] (c3p21323) -- (c3p21213);
\draw [-] (c3p23) -- (c3p21323);
\draw [-] (c3p23) -- (c3p213);
\draw [-] (c3p2id) -- (c3p21);
\draw [-] (c3p2id) -- (c3p22);
\draw [-] (c3p2id) -- (c3p23);
\node [align=center] (c3p2two2323) at (11.44, 5.85) {$\bullet$};
\node [align=center] (c3p2two121321323) at (11.44, 7.15) {$\bullet$};
\node [align=center] (c3p2two1213) at (11.44, 3.90) {$\bullet$};
\node [align=center] (c3p2two13) at (11.77, 2.60) {$\bullet$};
\node [align=center] (c3p2two121) at (11.11, 3.25) {$\bullet$};
\node [align=center] (c3p2two3) at (11.77, 1.30) {$\bullet$};
\node [align=center] (c3p2two1) at (11.44, 0.33) {$\bullet$};
\node [align=center] (c3p2twoid) at (11.44, -0.65) {$\bullet$};
\draw [-] (c3p2two2323) -- (c3p2two121321323);
\draw [-] (c3p2two1213) -- (c3p2two2323);
\draw [-] (c3p2two13) -- (c3p2two1213);
\draw [-] (c3p2two121) -- (c3p2two1213);
\draw [-] (c3p2two3) -- (c3p2two13);
\draw [-] (c3p2two1) -- (c3p2two3);
\draw [-] (c3p2two1) -- (c3p2two121);
\draw [-] (c3p2twoid) -- (c3p2two1);
\node [align=center] (b3p22323) at (15.73, 4.88) {$[2323]_{2,B}$};
\node [align=center] (b3p212132) at (14.85, 3.90) {$[12132]$};
\node [align=center] (b3p2132) at (14.85, 2.60) {$[132]$};
\node [align=center] (b3p213) at (15.95, 2.60) {$[13]$};
\node [align=center] (b3p2121321323) at (14.85, 7.15) {$[121321323]$};
\node [align=center] (b3p22) at (14.85, 1.30) {$[2]^*$};
\node [align=center] (b3p2id) at (14.85, -0.65) {$[id]$};
\node [align=center] (b3p2121323) at (15.40, 5.85) {$[121323]_{2,B}$};
\node [align=center] (b3p21213213) at (17.60, 5.85) {$[1213213]_{2,B}$};
\node [align=center] (b3p21323) at (17.05, 2.60) {$[1323]$};
\node [align=center] (b3p2121) at (13.20, 3.90) {$[121]$};
\node [align=center] (b3p223) at (16.78, 1.30) {$[23]_{2,B}$};
\node [align=center] (b3p21213) at (17.60, 3.90) {$[1213]$};
\node [align=center] (b3p21) at (13.20, 1.30) {$[1]^*$};
\node [align=center] (b3p223213) at (17.05, 4.88) {$[23213]_{2,B}$};
\node [align=center] (b3p23) at (15.95, -0.20) {$[3]_{2,B}$};
\node [align=center] (b3p2121321) at (13.75, 5.85) {$[121321]$};
\draw [-] (b3p2121321) -- (b3p2121321323);
\draw [-] (b3p2121) -- (b3p2121321);
\draw [-] (b3p21213213) -- (b3p2121321323);
\draw [-] (b3p21213) -- (b3p21213213);
\draw [-] (b3p223213) -- (b3p21213213);
\draw [-] (b3p213) -- (b3p21213);
\draw [-] (b3p213) -- (b3p223213);
\draw [-] (b3p21) -- (b3p2121);
\draw [-] (b3p21) -- (b3p213);
\draw [-] (b3p2121323) -- (b3p2121321323);
\draw [-] (b3p22323) -- (b3p2121323);
\draw [-] (b3p212132) -- (b3p2121321);
\draw [-] (b3p212132) -- (b3p2121323);
\draw [-] (b3p2132) -- (b3p2121);
\draw [-] (b3p2132) -- (b3p22323);
\draw [-] (b3p2132) -- (b3p212132);
\draw [-] (b3p22) -- (b3p2132);
\draw [-] (b3p21323) -- (b3p22323);
\draw [-] (b3p21323) -- (b3p21213);
\draw [-] (b3p223) -- (b3p21323);
\draw [-] (b3p23) -- (b3p213);
\draw [-] (b3p23) -- (b3p223);
\draw [-] (b3p2id) -- (b3p23);
\draw [-] (b3p2id) -- (b3p22);
\draw [-] (b3p2id) -- (b3p21);
\node [align=center] (b3p2two121321) at (18.59, 5.85) {$\bullet$};
\node [align=center] (b3p2two121321323) at (18.59, 7.15) {$\bullet$};
\node [align=center] (b3p2two121) at (18.26, 3.90) {$\bullet$};
\node [align=center] (b3p2two2323) at (18.92, 4.88) {$\bullet$};
\node [align=center] (b3p2two13) at (18.59, 2.60) {$\bullet$};
\node [align=center] (b3p2two1) at (18.59, 1.30) {$\bullet$};
\node [align=center] (b3p2two3) at (18.59, -0.20) {$\bullet$};
\node [align=center] (b3p2twoid) at (18.59, -0.65) {$\bullet$};
\draw [-] (b3p2two121321) -- (b3p2two121321323);
\draw [-] (b3p2two121) -- (b3p2two121321);
\draw [-] (b3p2two2323) -- (b3p2two121321);
\draw [-] (b3p2two13) -- (b3p2two121);
\draw [-] (b3p2two13) -- (b3p2two2323);
\draw [-] (b3p2two1) -- (b3p2two13);
\draw [-] (b3p2two3) -- (b3p2two1);
\draw [-] (b3p2twoid) -- (b3p2two3);
\node at (-3.0250000000000004, -0.65) {$(9, -)$};
\node at (-3.0250000000000004, 1.3) {$(3, +)$};
\node at (-3.0250000000000004, 2.6) {$(1, -)$};
\node at (-3.0250000000000004, 3.9000000000000004) {$(-1, +)$};
\node at (-3.0250000000000004, 5.8500000000000005) {$(-3, -)$};
\node at (-3.0250000000000004, 7.15) {$(-9, +)$};
\node at (-3.0250000000000004, 0.325) {$(3, -)$};
\node at (-3.0250000000000004, 3.25) {$(-3, -)$};
\node at (-3.0250000000000004, -0.195) {$(3, +)$};
\node at (-3.0250000000000004, 4.875) {$(-3, +)$};
\node at (-3.0250000000000004, 7.800000000000001) {$(\mathbf{x}, (-1)^{\mathbf{c}})$};
\node at (0.0, 7.800000000000001) {$C_3$ or $B_3$, $p = 0$};
\node at (7.700000000000001, 7.800000000000001) {$C_3$, $p = 2$};
\node at (14.850000000000001, 7.800000000000001) {$B_3$, $p = 2$};
\end{tikzpicture}
\caption{
	Left and two-sided $0$-cells and $2$-cells for types $C_3$ and $B_3$.
}
\end{figure}


In the above, there are some $0$-cells which are a union of $2$-cells, for example in type $C_3$ we have $[1] = [1]_{2, C} \sqcup [321]_{2, C}$ and $[2] = [2]_{2, C} \sqcup [32]_{2, C}$, while in type $B_3$ all $0$-cells are unions of $2$-cells.
However, in type $C_3$ we also have $[121] \sqcup [121321] = [121]_{C, 2} \sqcup [1321]_{C, 2} \sqcup [232132]_{C, 2}$, and no union of cells on the right is a cell on the left.
We also have that $132 <_L 121$ in the $0$-left-cell ordering on $C_3$, but $132 \not <_L 121$ in the $2$-left-cell ordering.
Finally, we can see that on the chain $[121]_{2, C} < [1321]_{2, C} < [232132]_{2, C}$ the eigenvalue $\mathbf{x}^2(-)$ takes values $-3, -1, -3$, meaning that we cannot expect that the implication $\lambda < \mu \implies \mathbf{x}^p(\lambda) < \mathbf{x}^p(\mu)$ will hold for $p$-cells, as it does for $0$-cells.

\begin{table}[h!]
\begin{tabular}{| r | r r l | r | r r l | }
\hline
 & $\mathbf{x}$ & $\mathbf{c}$ & $[w]$ & & $\mathbf{x}^2$ & $\mathbf{c}^2$ & $[w]_2$ \\
\hline
$[id]$ & $9$ & $-$ & id & &&& \\
\hline
$[1]^*$ & $3$ & $+$ & (1, 12321), (21, 2321), 321 & $[1]_{2,C}$ & $3$ & $-$ & 1, 21 \\
&&& & $[321]_{2,C}$ & $3$ & $+$ & 321, 2321, 12321 \\
\hline
$[2]^*$ & $3$ & $+$ & (2, 232), (12, 1232), 32 & $[2]_{2,C}$ & $3$ & $-$ & 2, 12 \\
&&& & $[32]_{2,C}$ & $3$ & $+$ & 32, 232, 1232 \\
\hline
$[3]$ & $3$ & $+$ & 3, 23, 123, 323 & $[3]_{2,B}$ & $3$ & $+$ & 3 \\
&&& & $[23]_{2,B}$ & $3$ & $+$ & 23, 123, 323 \\
\hline
$[13]$ & $1$ & $-$ & 13, 213, 3213 & &&& \\
\hline
$[121]$ & $-1$ & $+$ & 121, 1321, 21321 & $[121]_{2,C}$ & $-3$ & $-$ & 121 \\
$[121321]$ & $-3$ & $-$ & 121321, 232132, 1232132, 12132132 & $[1321]_{2,C}$ & $-1$ & $+$ & 1321, 21321, 121321 \\
&&& & $[232132]_{2,C}$ & $-3$ & $-$ & 232132, 1232132, 12132132 \\
\hline
$[132]$ & $1$ & $-$ & 132, 2132, 32132 & &&& \\
\hline
$[1213]$ & $-1$ & $+$ & 1213, 13213, 213213 & &&& \\
\hline
$[1323]$ & $1$ & $-$ & 1323, 21323, 321323 & &&& \\
\hline
$[2323]^*$ & $-3$ & $-$ & (2323, 21321323), (12323, 1321323), 121323 & $[2323]_{2,B}$ & $-3$ & $+$ & 2323, 12323 \\
&&& & $[121323]_{2,B}$ & $-3$ & $-$ & 121323, 1321323, 21321323 \\
\hline
$[12132]$ & $-1$ & $+$ & 12132, 132132, 2132132 & &&& \\
\hline
$[23213]^*$ & $-3$ & $-$ & (23213, 2321323), (123213, 12321323), 1213213 & $[23213]_{2,B}$ & $-3$ & $+$ & 23213, 123213 \\
&&& & $[1213213]_{2,B}$ & $-3$ & $-$ & 1213213, 2321323, 12321323 \\
\hline
$[121321323]$ & $-9$ & $+$ & 121321323 & &&& \\
\hline
\end{tabular}

\caption{ Contents of the left $0$-cells and $2$-cells of $C_3$ and $B_3$. Rows have been grouped together whenever a $2$-cell intersects a $0$-cell. When a $2$-cell and $0$-cell have the same elements and Schutzenberger involution, the content of the $2$-cell is omitted. For example, $1$ is still mapped to $12321$ under $\operatorname{Schu}_L^2$ in type $B_3$.
}
\end{table}


\end{landscape}

\subsection{Rank 4}

For the rank 4 types $B_4$, $C_4$, $D_4$, and $F_4$ we only show the two-sided cells.
Since words in the generators get quite long, we write $\overline{w} = w_0 w$ for a Coxeter group element, and when choosing representatives for a two-sided cell $\lambda$, we will either have $\lambda = [\![ w ]\!]$ where $w \in \lambda$ is the least (in short-lex order) element belonging to $\lambda$, or $\lambda = [\![ \overline{w} ]\!]$, where $\overline{w}$ is the greatest (in short-lex order) element belonging to $\lambda$.
When writing down a cell, we choose whichever representative will give us the shortest word.
Again we mark cells on which $\operatorname{Schu}_L^p$ acts nontrivially with an asterisk, and only label $p$-cells with a subscript $p$ when those cells do not appear as $0$-cells.

\begin{figure}[h!]
\label{fig:rank4-BC}
\centering
\begin{tikzpicture}[every node/.style={scale=0.85}]
\node [align=center] (c4p0232432434) at (0.00, 9.38) {$[\![\overline{1}]\!]^*$};
\node [align=center] (c4p01213214321432434) at (0.00, 10.50) {$[\![\overline{id}]\!]$};
\node [align=center] (c4p0121321) at (0.00, 8.25) {$[\![\overline{13}]\!]^*$};
\node [align=center] (c4p013434) at (0.00, 7.50) {$[\![\overline{121}]\!]$};
\node [align=center] (c4p01214) at (-0.80, 6.00) {$[\![1214]\!]$};
\node [align=center] (c4p03434) at (0.80, 6.00) {$[\![3434]\!]$};
\node [align=center] (c4p0121) at (0.00, 4.20) {$[\![121]\!]$};
\node [align=center] (c4p013) at (0.00, 3.00) {$[\![13]\!]^*$};
\node [align=center] (c4p01) at (0.00, 1.50) {$[\![1]\!]^*$};
\node [align=center] (c4p0id) at (0.00, 0.00) {$[\![id]\!]$};
\draw [-] (c4p0232432434) -- (c4p01213214321432434);
\draw [-] (c4p0121321) -- (c4p0232432434);
\draw [-] (c4p013434) -- (c4p0121321);
\draw [-] (c4p01214) -- (c4p013434);
\draw [-] (c4p03434) -- (c4p013434);
\draw [-] (c4p0121) -- (c4p01214);
\draw [-] (c4p0121) -- (c4p03434);
\draw [-] (c4p013) -- (c4p0121);
\draw [-] (c4p01) -- (c4p013);
\draw [-] (c4p0id) -- (c4p01);
\node [align=center] (c4p2232432434) at (4.80, 9.38) {$[\![\overline{1}]\!]^*_{2,C}$};
\node [align=center] (c4p21213214321432434) at (4.80, 10.50) {$[\![\overline{id}]\!]$};
\node [align=center] (c4p2121321434) at (4.80, 8.25) {$[\![\overline{13}]\!]_{2,C}$};
\node [align=center] (c4p213434) at (5.20, 7.50) {$[\![\overline{121}]\!]$};
\node [align=center] (c4p21213214) at (3.60, 6.75) {$[\![\overline{1434}]\!]_{2,C}$};
\node [align=center] (c4p21214) at (4.40, 6.00) {$[\![1214]\!]$};
\node [align=center] (c4p23434) at (6.00, 6.00) {$[\![3434]\!]_{2,C}$};
\node [align=center] (c4p22324) at (5.20, 4.20) {$[\![2324]\!]_{2,C}$};
\node [align=center] (c4p214) at (5.20, 3.00) {$[\![14]\!]^*_{2,C}$};
\node [align=center] (c4p24) at (5.20, 1.50) {$[\![4]\!]^*_{2,C}$};
\node [align=center] (c4p2121321) at (3.60, 5.25) {$[\![121321]\!]_{2,C}$};
\node [align=center] (c4p2121) at (3.60, 3.75) {$[\![121]\!]_{2,C}$};
\node [align=center] (c4p213) at (3.60, 2.62) {$[\![13]\!]_{2,C}$};
\node [align=center] (c4p21) at (4.80, 0.75) {$[\![1]\!]_{2,C}$};
\node [align=center] (c4p2id) at (4.80, 0.00) {$[\![id]\!]$};
\draw [-] (c4p2232432434) -- (c4p21213214321432434);
\draw [-] (c4p2121321434) -- (c4p2232432434);
\draw [-] (c4p213434) -- (c4p2121321434);
\draw [-] (c4p21213214) -- (c4p2121321434);
\draw [-] (c4p21214) -- (c4p21213214);
\draw [-] (c4p21214) -- (c4p213434);
\draw [-] (c4p23434) -- (c4p213434);
\draw [-] (c4p22324) -- (c4p21214);
\draw [-] (c4p22324) -- (c4p23434);
\draw [-] (c4p214) -- (c4p22324);
\draw [-] (c4p24) -- (c4p214);
\draw [-] (c4p2121321) -- (c4p21213214);
\draw [-] (c4p2121) -- (c4p2121321);
\draw [-] (c4p2121) -- (c4p22324);
\draw [-] (c4p213) -- (c4p214);
\draw [-] (c4p213) -- (c4p2121);
\draw [-] (c4p21) -- (c4p213);
\draw [-] (c4p21) -- (c4p24);
\draw [-] (c4p2id) -- (c4p21);
\node [align=center] (b4p2121321432434) at (9.60, 9.38) {$[\![\overline{1}]\!]^*_{2,B}$};
\node [align=center] (b4p21213214321432434) at (9.60, 10.50) {$[\![\overline{id}]\!]$};
\node [align=center] (b4p2121321) at (8.80, 8.25) {$[\![\overline{13}]\!]^*_{2,B}$};
\node [align=center] (b4p2232432434) at (10.40, 8.70) {$[\![\overline{343}]\!]_{2,B}$};
\node [align=center] (b4p21213434) at (9.60, 7.50) {$[\![\overline{121}]\!]_{2,B}$};
\node [align=center] (b4p213434) at (9.60, 4.65) {$[\![13434]\!]_{2,B}$};
\node [align=center] (b4p21214) at (8.00, 6.00) {$[\![1214]\!]$};
\node [align=center] (b4p2232432) at (11.20, 6.00) {$[\![\overline{3434}]\!]_{2,B}$};
\node [align=center] (b4p2121) at (8.00, 4.20) {$[\![121]\!]$};
\node [align=center] (b4p213) at (8.80, 3.00) {$[\![13]\!]_{2,B}$};
\node [align=center] (b4p23434) at (10.40, 3.75) {$[\![3434]\!]_{2,B}$};
\node [align=center] (b4p214) at (9.60, 2.17) {$[\![14]\!]_{2,B}$};
\node [align=center] (b4p21) at (9.60, 1.50) {$[\![1]\!]^*_{2,B}$};
\node [align=center] (b4p24) at (9.60, 0.75) {$[\![4]\!]_{2,B}$};
\node [align=center] (b4p2id) at (9.60, 0.00) {$[\![id]\!]$};
\draw [-] (b4p2121321432434) -- (b4p21213214321432434);
\draw [-] (b4p2121321) -- (b4p2121321432434);
\draw [-] (b4p2232432434) -- (b4p2121321432434);
\draw [-] (b4p21213434) -- (b4p2232432434);
\draw [-] (b4p21213434) -- (b4p2121321);
\draw [-] (b4p213434) -- (b4p21213434);
\draw [-] (b4p21214) -- (b4p21213434);
\draw [-] (b4p2232432) -- (b4p21213434);
\draw [-] (b4p2121) -- (b4p2232432);
\draw [-] (b4p2121) -- (b4p21214);
\draw [-] (b4p213) -- (b4p2121);
\draw [-] (b4p213) -- (b4p213434);
\draw [-] (b4p23434) -- (b4p2232432);
\draw [-] (b4p23434) -- (b4p213434);
\draw [-] (b4p214) -- (b4p23434);
\draw [-] (b4p214) -- (b4p213);
\draw [-] (b4p21) -- (b4p214);
\draw [-] (b4p24) -- (b4p21);
\draw [-] (b4p2id) -- (b4p24);
\node at (-2.2, 0.0) {$(16, +)$};
\node at (-2.2, 1.5) {$(8, -)$};
\node at (-2.2, 3.0) {$(4, +)$};
\node at (-2.2, 4.199999999999999) {$(2, -)$};
\node at (-2.2, 6.0) {$(0, +)$};
\node at (-2.2, 7.5) {$(-2, -)$};
\node at (-2.2, 8.25) {$(-4, +)$};
\node at (-2.2, 9.375) {$(-8, -)$};
\node at (-2.2, 10.5) {$(-16, +)$};
\node at (-2.2, 0.75) {$(8, +)$};
\node at (-2.2, 2.625) {$(4, +)$};
\node at (-2.2, 3.75) {$(0, +)$};
\node at (-2.2, 5.25) {$(-8, +)$};
\node at (-2.2, 6.75) {$(-4, -)$};
\node at (-2.2, 2.175) {$(4, -)$};
\node at (-2.2, 4.65) {$(-4, +)$};
\node at (-2.2, 8.7) {$(-8, +)$};
\node at (-2.2, 11.25) {$(\mathbf{x}, (-1)^{\mathbf{c}})$};
\node at (0.0, 11.25) {$C_4$ or $B_4$, $p = 0$};
\node at (4.800000000000001, 11.25) {$C_4$, $p = 2$};
\node at (9.600000000000001, 11.25) {$B_4$, $p = 2$};
\end{tikzpicture}
\caption{Two-sided $0$-cells and $2$-cells in types $C_4$ and $B_4$.}
\end{figure}
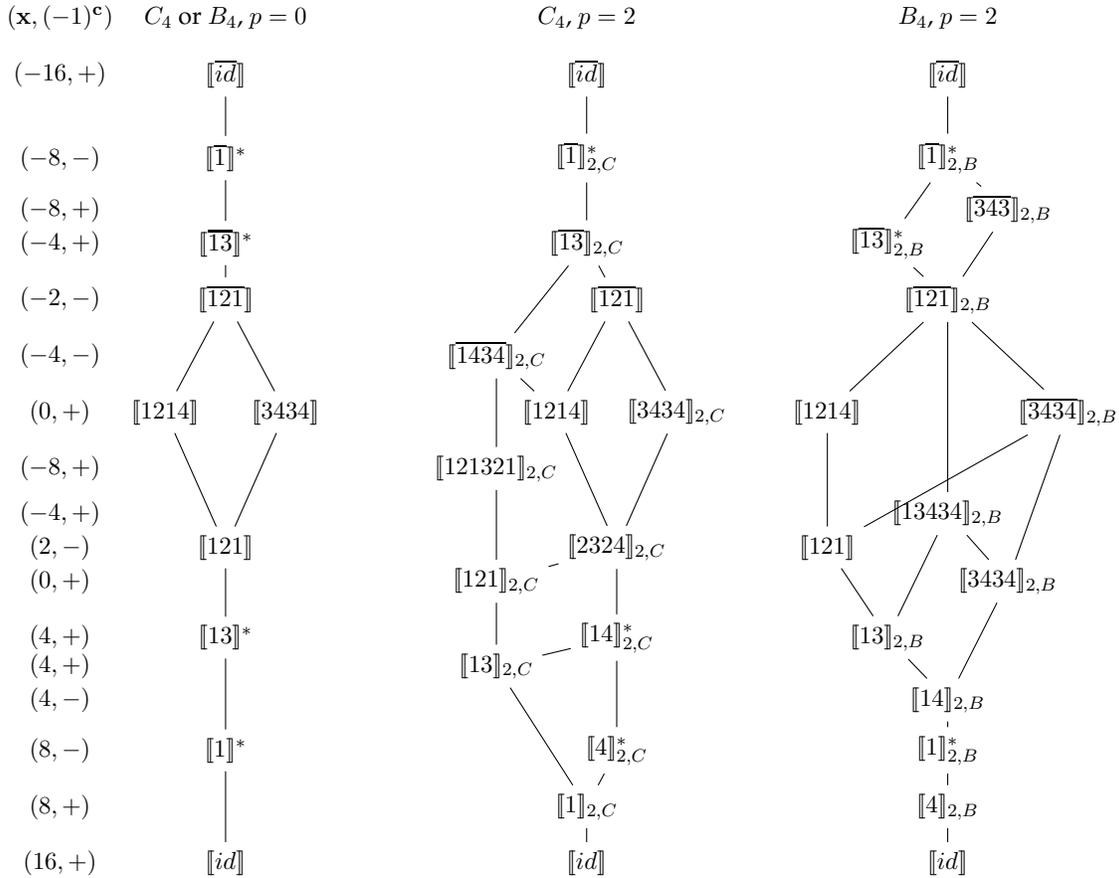

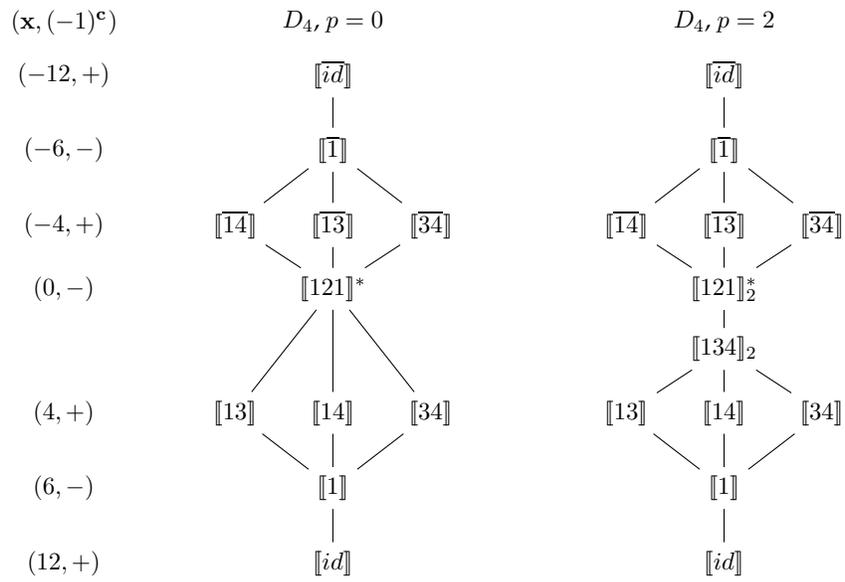
\begin{figure}[h!]
\label{fig:rank4-D}
\centering
\begin{tikzpicture}[every node/.style={scale=0.85}]
\node [align=center] (d4p0121321421) at (0.00, 5.50) {$[\![\overline{1}]\!]$};
\node [align=center] (d4p0121321421324) at (0.00, 6.50) {$[\![\overline{id}]\!]$};
\node [align=center] (d4p0121421) at (-1.30, 4.50) {$[\![\overline{14}]\!]$};
\node [align=center] (d4p0232423) at (1.30, 4.50) {$[\![\overline{34}]\!]$};
\node [align=center] (d4p0121321) at (0.00, 4.50) {$[\![\overline{13}]\!]$};
\node [align=center] (d4p0121) at (0.00, 3.65) {$[\![121]\!]^*$};
\node [align=center] (d4p014) at (0.00, 2.00) {$[\![14]\!]$};
\node [align=center] (d4p034) at (1.30, 2.00) {$[\![34]\!]$};
\node [align=center] (d4p013) at (-1.30, 2.00) {$[\![13]\!]$};
\node [align=center] (d4p01) at (0.00, 1.00) {$[\![1]\!]$};
\node [align=center] (d4p0id) at (0.00, 0.00) {$[\![id]\!]$};
\draw [-] (d4p0121321421) -- (d4p0121321421324);
\draw [-] (d4p0121421) -- (d4p0121321421);
\draw [-] (d4p0232423) -- (d4p0121321421);
\draw [-] (d4p0121321) -- (d4p0121321421);
\draw [-] (d4p0121) -- (d4p0121321);
\draw [-] (d4p0121) -- (d4p0232423);
\draw [-] (d4p0121) -- (d4p0121421);
\draw [-] (d4p014) -- (d4p0121);
\draw [-] (d4p034) -- (d4p0121);
\draw [-] (d4p013) -- (d4p0121);
\draw [-] (d4p01) -- (d4p014);
\draw [-] (d4p01) -- (d4p034);
\draw [-] (d4p01) -- (d4p013);
\draw [-] (d4p0id) -- (d4p01);
\node [align=center] (d4p2121321421) at (5.20, 5.50) {$[\![\overline{1}]\!]$};
\node [align=center] (d4p2121321421324) at (5.20, 6.50) {$[\![\overline{id}]\!]$};
\node [align=center] (d4p2121421) at (3.90, 4.50) {$[\![\overline{14}]\!]$};
\node [align=center] (d4p2232423) at (6.50, 4.50) {$[\![\overline{34}]\!]$};
\node [align=center] (d4p2121321) at (5.20, 4.50) {$[\![\overline{13}]\!]$};
\node [align=center] (d4p2121) at (5.20, 3.65) {$[\![121]\!]^*_2$};
\node [align=center] (d4p2134) at (5.20, 2.85) {$[\![134]\!]_2$};
\node [align=center] (d4p214) at (5.20, 2.00) {$[\![14]\!]$};
\node [align=center] (d4p234) at (6.50, 2.00) {$[\![34]\!]$};
\node [align=center] (d4p213) at (3.90, 2.00) {$[\![13]\!]$};
\node [align=center] (d4p21) at (5.20, 1.00) {$[\![1]\!]$};
\node [align=center] (d4p2id) at (5.20, 0.00) {$[\![id]\!]$};
\draw [-] (d4p2121321421) -- (d4p2121321421324);
\draw [-] (d4p2121421) -- (d4p2121321421);
\draw [-] (d4p2232423) -- (d4p2121321421);
\draw [-] (d4p2121321) -- (d4p2121321421);
\draw [-] (d4p2121) -- (d4p2121321);
\draw [-] (d4p2121) -- (d4p2232423);
\draw [-] (d4p2121) -- (d4p2121421);
\draw [-] (d4p2134) -- (d4p2121);
\draw [-] (d4p214) -- (d4p2134);
\draw [-] (d4p234) -- (d4p2134);
\draw [-] (d4p213) -- (d4p2134);
\draw [-] (d4p21) -- (d4p213);
\draw [-] (d4p21) -- (d4p234);
\draw [-] (d4p21) -- (d4p214);
\draw [-] (d4p2id) -- (d4p21);
\node at (-3.575, 0) {$(12, +)$};
\node at (-3.575, 1) {$(6, -)$};
\node at (-3.575, 2) {$(4, +)$};
\node at (-3.575, 3.65) {$(0, -)$};
\node at (-3.575, 4.5) {$(-4, +)$};
\node at (-3.575, 5.5) {$(-6, -)$};
\node at (-3.575, 6.5) {$(-12, +)$};
\node at (-3.575, 7.2) {$(\mathbf{x}, (-1)^{\mathbf{c}})$};
\node at (0.0, 7.2) {$D_4$, $p = 0$};
\node at (5.2, 7.2) {$D_4$, $p = 2$};
\end{tikzpicture}
\caption{Two-sided $0$-cells and $2$-cells in type $D_4$.}
\end{figure}

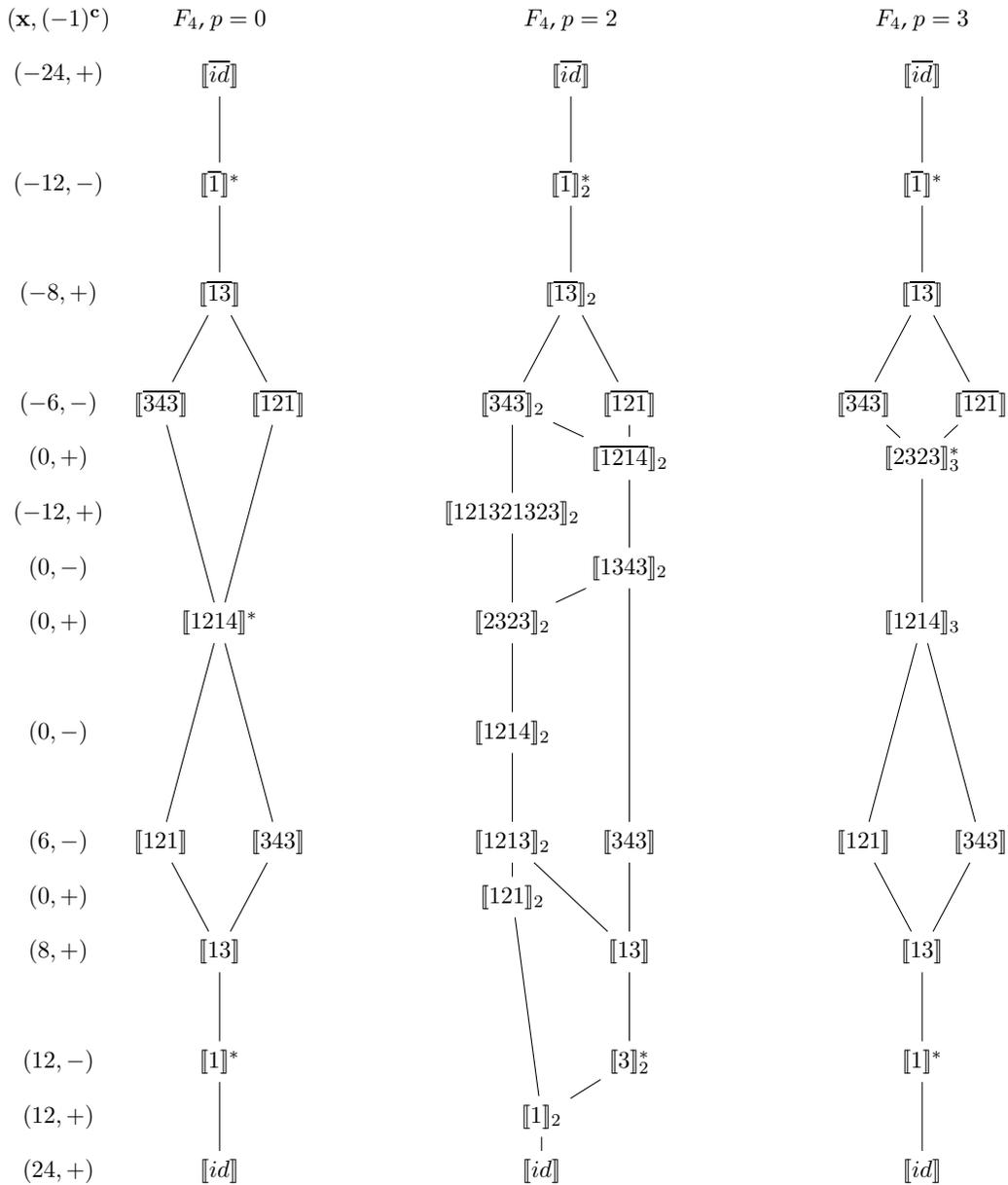
\begin{figure}[h!]
\label{fig:rank4-F}
\centering
\begin{tikzpicture}[every node/.style={scale=0.85}]
\node [align=center] (f4p01213213234321324321) at (0.00, 13.50) {$[\![\overline{1}]\!]^*$};
\node [align=center] (f4p0121321323432132343213234) at (0.00, 15.00) {$[\![\overline{id}]\!]$};
\node [align=center] (f4p012132132343234) at (0.00, 12.00) {$[\![\overline{13}]\!]$};
\node [align=center] (f4p0121321323) at (-0.80, 10.50) {$[\![\overline{343}]\!]$};
\node [align=center] (f4p0232343234) at (0.80, 10.50) {$[\![\overline{121}]\!]$};
\node [align=center] (f4p01214) at (0.00, 7.50) {$[\![1214]\!]^*$};
\node [align=center] (f4p0343) at (0.80, 4.50) {$[\![343]\!]$};
\node [align=center] (f4p0121) at (-0.80, 4.50) {$[\![121]\!]$};
\node [align=center] (f4p013) at (0.00, 3.00) {$[\![13]\!]$};
\node [align=center] (f4p01) at (0.00, 1.50) {$[\![1]\!]^*$};
\node [align=center] (f4p0id) at (0.00, 0.00) {$[\![id]\!]$};
\draw [-] (f4p01213213234321324321) -- (f4p0121321323432132343213234);
\draw [-] (f4p012132132343234) -- (f4p01213213234321324321);
\draw [-] (f4p0121321323) -- (f4p012132132343234);
\draw [-] (f4p0232343234) -- (f4p012132132343234);
\draw [-] (f4p01214) -- (f4p0232343234);
\draw [-] (f4p01214) -- (f4p0121321323);
\draw [-] (f4p0343) -- (f4p01214);
\draw [-] (f4p0121) -- (f4p01214);
\draw [-] (f4p013) -- (f4p0121);
\draw [-] (f4p013) -- (f4p0343);
\draw [-] (f4p01) -- (f4p013);
\draw [-] (f4p0id) -- (f4p01);
\node [align=center] (f4p22323432132343213234) at (4.80, 13.50) {$[\![\overline{1}]\!]^*_2$};
\node [align=center] (f4p2121321323432132343213234) at (4.80, 15.00) {$[\![\overline{id}]\!]$};
\node [align=center] (f4p212132132343234) at (4.80, 12.00) {$[\![\overline{13}]\!]_2$};
\node [align=center] (f4p2232343234) at (5.60, 10.50) {$[\![\overline{121}]\!]$};
\node [align=center] (f4p212132132343) at (4.00, 10.50) {$[\![\overline{343}]\!]_2$};
\node [align=center] (f4p21213432134) at (5.60, 9.75) {$[\![\overline{1214}]\!]_2$};
\node [align=center] (f4p21343) at (5.60, 8.25) {$[\![1343]\!]_2$};
\node [align=center] (f4p2121321323) at (4.00, 9.00) {$[\![121321323]\!]_2$};
\node [align=center] (f4p22323) at (4.00, 7.50) {$[\![2323]\!]_2$};
\node [align=center] (f4p21214) at (4.00, 6.00) {$[\![1214]\!]_2$};
\node [align=center] (f4p21213) at (4.00, 4.50) {$[\![1213]\!]_2$};
\node [align=center] (f4p2343) at (5.60, 4.50) {$[\![343]\!]$};
\node [align=center] (f4p213) at (5.60, 3.00) {$[\![13]\!]$};
\node [align=center] (f4p23) at (5.60, 1.50) {$[\![3]\!]^*_2$};
\node [align=center] (f4p2121) at (4.00, 3.75) {$[\![121]\!]_2$};
\node [align=center] (f4p21) at (4.40, 0.75) {$[\![1]\!]_2$};
\node [align=center] (f4p2id) at (4.40, 0.00) {$[\![id]\!]$};
\draw [-] (f4p22323432132343213234) -- (f4p2121321323432132343213234);
\draw [-] (f4p212132132343234) -- (f4p22323432132343213234);
\draw [-] (f4p2232343234) -- (f4p212132132343234);
\draw [-] (f4p212132132343) -- (f4p212132132343234);
\draw [-] (f4p21213432134) -- (f4p2232343234);
\draw [-] (f4p21213432134) -- (f4p212132132343);
\draw [-] (f4p21343) -- (f4p21213432134);
\draw [-] (f4p2121321323) -- (f4p212132132343);
\draw [-] (f4p22323) -- (f4p21343);
\draw [-] (f4p22323) -- (f4p2121321323);
\draw [-] (f4p21214) -- (f4p22323);
\draw [-] (f4p21213) -- (f4p21214);
\draw [-] (f4p2343) -- (f4p21343);
\draw [-] (f4p213) -- (f4p21213);
\draw [-] (f4p213) -- (f4p2343);
\draw [-] (f4p23) -- (f4p213);
\draw [-] (f4p2121) -- (f4p21213);
\draw [-] (f4p21) -- (f4p23);
\draw [-] (f4p21) -- (f4p2121);
\draw [-] (f4p2id) -- (f4p21);
\node [align=center] (f4p31213213234321324321) at (9.60, 13.50) {$[\![\overline{1}]\!]^*$};
\node [align=center] (f4p3121321323432132343213234) at (9.60, 15.00) {$[\![\overline{id}]\!]$};
\node [align=center] (f4p312132132343234) at (9.60, 12.00) {$[\![\overline{13}]\!]$};
\node [align=center] (f4p3121321323) at (8.80, 10.50) {$[\![\overline{343}]\!]$};
\node [align=center] (f4p3232343234) at (10.40, 10.50) {$[\![\overline{121}]\!]$};
\node [align=center] (f4p32323) at (9.60, 9.75) {$[\![2323]\!]^*_3$};
\node [align=center] (f4p31214) at (9.60, 7.50) {$[\![1214]\!]_3$};
\node [align=center] (f4p3343) at (10.40, 4.50) {$[\![343]\!]$};
\node [align=center] (f4p3121) at (8.80, 4.50) {$[\![121]\!]$};
\node [align=center] (f4p313) at (9.60, 3.00) {$[\![13]\!]$};
\node [align=center] (f4p31) at (9.60, 1.50) {$[\![1]\!]^*$};
\node [align=center] (f4p3id) at (9.60, 0.00) {$[\![id]\!]$};
\draw [-] (f4p31213213234321324321) -- (f4p3121321323432132343213234);
\draw [-] (f4p312132132343234) -- (f4p31213213234321324321);
\draw [-] (f4p3121321323) -- (f4p312132132343234);
\draw [-] (f4p3232343234) -- (f4p312132132343234);
\draw [-] (f4p32323) -- (f4p3232343234);
\draw [-] (f4p32323) -- (f4p3121321323);
\draw [-] (f4p31214) -- (f4p32323);
\draw [-] (f4p3343) -- (f4p31214);
\draw [-] (f4p3121) -- (f4p31214);
\draw [-] (f4p313) -- (f4p3121);
\draw [-] (f4p313) -- (f4p3343);
\draw [-] (f4p31) -- (f4p313);
\draw [-] (f4p3id) -- (f4p31);
\node at (-2.2, 0.0) {$(24, +)$};
\node at (-2.2, 1.5) {$(12, -)$};
\node at (-2.2, 3.0) {$(8, +)$};
\node at (-2.2, 4.5) {$(6, -)$};
\node at (-2.2, 7.5) {$(0, +)$};
\node at (-2.2, 10.5) {$(-6, -)$};
\node at (-2.2, 12.0) {$(-8, +)$};
\node at (-2.2, 13.5) {$(-12, -)$};
\node at (-2.2, 15.0) {$(-24, +)$};
\node at (-2.2, 0.75) {$(12, +)$};
\node at (-2.2, 3.75) {$(0, +)$};
\node at (-2.2, 6.0) {$(0, -)$};
\node at (-2.2, 8.25) {$(0, -)$};
\node at (-2.2, 9.0) {$(-12, +)$};
\node at (-2.2, 9.75) {$(0, +)$};
\node at (-2.2, 15.75) {$(\mathbf{x}, (-1)^{\mathbf{c}})$};
\node at (0.0, 15.75) {$F_4$, $p = 0$};
\node at (4.800000000000001, 15.75) {$F_4$, $p = 2$};
\node at (9.600000000000001, 15.75) {$F_4$, $p = 3$};
\end{tikzpicture}
\caption{Two-sided $0$, $2$, and $3$-cells in type $F_4$.}
\end{figure}
\clearpage

\begin{landscape}
\subsection{Ranks 5 and 6} \label{subsec:5and6}

The conjecture has also been verified and eigenvalues computed in ranks 5 and 6, and we display the high-level statistics here. All interesting primes appear below.

\begin{table}[h!]
\label{table:type-statistics5}
\centering
\makebox[\textwidth][c]{\begin{tabular}{r | r r r | r r}
Type & $C_5$ & $C_5$ & $B_5$ & $D_5$ & $D_5$ \\
$p$ & 0 & 2 & 2 & 0 & 2 \\
\hline
\# Left cells & 182 & 247 & 247 & 126 & 136 \\
\# Two-sided cells & 16 & 26 & 26 & 14 & 16 \\
\# Unique $(\mathbf{x}, \pm)$ pairs & 14 & 20 & 20 & 14 & 16 \\
\# $\operatorname{Schu}_L^p$ fixed points & 1912 & 2876 & 2876 & 540 & 640 \\
\# $\operatorname{Schu}_L^p$ moving points & 1928 & 964 & 964 & 1380 & 1280 \\
\end{tabular}
}
\caption{Statistics on $0$-cells vs $p$-cells for the Cartan types of rank 5.}
\end{table}

\begin{table}[h!]
\label{table:type-statistics6}
\centering
\makebox[\textwidth][c]{\begin{tabular}{r | r r r r r | r r r | r r r r r}
Type & $C_6$ & $C_6$ & $B_6$ & $C_6$ & $B_6$ & $D_6$ & $D_6$ & $D_6$ & $E_6$ & $E_6$ & $E_6$ & $E_6$ & $E_6$ \\
$p$ & 0 & 2 & 2 & 3 & 3 & 0 & 2 & 3 & 0 & 2 & 3 & 5 & 7 \\
\hline
\# Left cells & 752 & 1058 & 1058 & 752 & 752 & 578 & 622 & 578 & 652 & 742 & 662 & 652 & 652 \\
\# Two-sided cells & 26 & 45 & 45 & 26 & 26 & 27 & 32 & 27 & 17 & 20 & 18 & 17 & 17 \\
\# Unique $(\mathbf{x}, \pm)$ pairs & 19 & 26 & 26 & 19 & 19 & 19 & 22 & 19 & 17 & 20 & 17 & 17 & 17 \\
\# $\operatorname{Schu}_L^p$ fixed points & 22824 & 33702 & 33702 & 22824 & 22824 & 14904 & 15740 & 14904 & 4008 & 5868 & 4008 & 4008 & 4008 \\
\# $\operatorname{Schu}_L^p$ moving points & 23256 & 12378 & 12378 & 23256 & 23256 & 8136 & 7300 & 8136 & 47832 & 45972 & 47832 & 47832 & 47832 \\
\end{tabular}
}
\caption{Statistics on $0$-cells vs $p$-cells for the Cartan types of rank 6.}
\end{table}

\begin{rmk} We remark on the computation time needed to compute the $p$-canonical basis and verify the conjecture for a given type and characteristic. Unless otherwise specified, the computation took seconds.  Types $D_6$, $B_5$, $C_5$, and $(F_4, p=2)$ took minutes. Types $E_6$, $(B_6, p=3)$ and $(C_6, p=3)$ took hours. Type $(B_6, p=2)$ took two days, and type $(C_6, p=2)$ took seven days. \end{rmk}

\clearpage
\end{landscape}

\bibliographystyle{plain}
\bibliography{mastercopy}

\end{document}